\documentclass[bj,authoryear]{imsart}

\usepackage{amsthm,amsmath,amsfonts}
\usepackage{dsfont} 
\usepackage{mathtools} 
\usepackage[only,llfloor,rrfloor]{stmaryrd} 

\startlocaldefs
\numberwithin{equation}{section}
\theoremstyle{plain}
\newtheorem{theorem}{Theorem}[section]
\newtheorem{proposition}[theorem]{Proposition}
\newtheorem{lemma}[theorem]{Lemma}
\theoremstyle{remark}
\newtheorem{definition}{Definition}
\newtheorem{remark}{Remark}
\newtheorem{assumption}{Assumption}
\newcommand{\N}{\mathbb{N}}
\newcommand{\R}{\mathbb{R}}
\newcommand{\PP}{\mathsf{P}}
\newcommand{\EE}{\mathsf{E}}
\newcommand{\Var}{\mathsf{Var}}
\newcommand{\1}{\mathds{1}}
\newcommand{\rd}{\mathrm{d}}
\newcommand{\domain}{\mathcal{D}}
\newcommand{\voisin}{\mathcal{V}}
\newcommand{\pwrisk}{R_n}
\newcommand{\finfty}{\left\|f\right\|_{\voisin(\rho)}}
\endlocaldefs

\begin{document}

\begin{frontmatter}
\title{A new adaptive local polynomial density estimation procedure on complicated domains}
\runtitle{Local polynomial density estimation on complicated domains}

\begin{aug}
\author[a1]{\inits{K.B.}\fnms{Karine}~\snm{Bertin}\ead[label=e1]{karine.bertin@uv.cl}\orcid{0000-0002-8167-137X}}
\author[a2]{\inits{N.K.}\fnms{Nicolas}~\snm{Klutchnikoff}\ead[label=e2]{nicolas.klutchnikoff@univ-rennes2.fr}\orcid{0000-0001-9349-0771}}
\author[a3,a4]{\inits{F.O.}\fnms{Fr\'ed\'eric}~\snm{Ouimet}\ead[label=e3]{frederic.ouimet2@mcgill.ca}\orcid{0000-0001-7933-5265}}
\address[a1]{CIMFAV-INGEMAT, Universidad de Valpa{r}a\'{\i}so, Valpara\'iso, Chile\printead[presep={,\ }]{e1}}
\address[a2]{Univ Rennes, CNRS, IRMAR-UMR 6625, F-35000 Rennes, France\printead[presep={,\ }]{e2}}
\address[a3]{Centre de recherches math\'ematiques, Universit\'e de Montr\'eal, Montreal, Canada}
\address[a4]{Department of Mathematics and Statistics, McGill University, Montreal, Canada\printead[presep={,\ }]{e3}}
\end{aug}

\begin{abstract}
This paper presents a novel approach for pointwise estimation of multivariate density functions on known domains of arbitrary dimensions using nonparametric local polynomial estimators. Our method is highly flexible, as it applies to both simple domains, such as open connected sets, and more complicated domains that are not star-shaped around the point of estimation. This enables us to handle domains with sharp concavities, holes, and local pinches, such as polynomial sectors. Additionally, we introduce a data-driven selection rule based on the general ideas of Goldenshluger and Lepski. Our results demonstrate that the local polynomial estimators are minimax under a $L^2$ risk across a wide range of H\"older-type functional classes. In the adaptive case, we provide oracle inequalities and explicitly determine the convergence rate of our statistical procedure. Simulations on polynomial sectors show that our oracle estimates outperform those of the most popular alternative method, found in the \texttt{sparr} package for the \texttt{R} software. Our statistical procedure is implemented in an online \texttt{R} package which is readily accessible.
\end{abstract}


\end{frontmatter}

\section{Introduction}\label{sec:introduction}

    In practice, estimating a probability density function near or at the boundary of its support $\domain$ is often challenging. This issue is particularly prominent in standard kernel density estimators, which suffer from a well-known bias that negatively impacts their performance. In situations where $\domain$ represents a geometrically `simple' domain in $\R^d$, this behavior is more indicative of limitations in these statistical procedures rather than an inherently difficult estimation problem. Consequently, dedicated estimators have been devised to reduce or eliminate the bias term, providing more accurate density estimates.

    For instance, Bernstein estimators, introduced by \cite{Vitale1975} and studied further among others by \cite{GawronskiStadtmuller1981,Petrone1999,Ghosal2001,BabuChaubey2006,PetroneWasserman2002}, utilize a specific choice of discrete kernel function whose shape adapts locally to the support of the density and leads to reduced bias near the boundaries. Asymmetric kernel estimators, put forward independently by \cite{AitchisonLauder1985} and \cite{Chen1999}, extend the same idea to smooth kernels. The reflection method, proposed by \cite{Schuster1985} and further investigated by \cite{ClineHart1991}, involves extending the support of the density function symmetrically beyond the observed data range and incorporating reflected data points into the estimation process. Boundary kernel estimators, initially proposed by \cite{GasserMuller1979} and further refined by \cite{GasserMullerMammitzsch1985,Muller1991,Jones1993,ZhangKarunamuni1998,ZhangKarunamuni2000}, assign higher weights to data points located near the boundaries, effectively giving greater emphasis to the boundary regions during estimation and improving accuracy. Local polynomial smoothers with variable bandwidths, introduced by \cite{FanGijbels1992} in the regression context, see also \cite{FanGijbels1996,ChengFanMarron1997}, fit low-degree polynomials to local neighborhoods of data points, adaptively adjusting the polynomial degree and bandwidth to mitigate boundary bias. Other notable techniques include smoothing splines \citep{Gu1993,GuQiu1993}, generalized jackknifing \citep{Jones1993,JonesFoster1996}, and the transformation technique \citep{MarronRuppert1994,RupperCline1994}, which applies a transformation to the data before smoothing on a more suited domain. A combination that leverages the benefits of both reflection and transformation was proposed by \cite{ZhangKarunamuniJones1999}. For a modern review of kernel density estimation methods and bandwidth selection for spatial data, refer to \cite{DaviesMarshallHazelton2018}.

    In spite of the remarkable strides made throughout the years to alleviate the boundary bias of kernel density estimators, it often comes at the cost of producing other issues such as estimators that lack local adaptivity (i.e., are sensitive to outliers, generate spurious bumps, etc.), inadequate data-driven bandwidth selection rules (e.g., plug-in methods with reference/pilot densities that spill over boundaries), or density estimates that are not bona fide density functions (i.e., that are negative in some regions or do not integrate to one). Only recently, some methods have been able to address all these concerns at once.

    For example, \cite{BotevGrotowskiKroese2010} propose an adaptive kernel density estimator which is the solution, in a space-scale setting, to the Fokker-Planck partial differential equation (PDE) associated with an It\^o diffusion process having the empirical density function as initial condition. The scale variable plays the role of the bandwidth and the shape specifics of the domain are taken into account directly using the boundary conditions of the PDE. Their proposed estimator is locally adaptive, has a bandwidth selection procedure which is free of the so-called normal reference rules (and thus is truly nonparametric), deals well with boundary bias and is always a bona fide density function. Further, the overall computational cost remains low. Yet, as pointed out by \cite{Ferraccioli_et_al2021}, their estimator is unable to handle domains with complicated geometries.

    In recent efforts to tackle the above concerns regarding kernel density estimators, specific focus has been directed towards estimating multivariate density functions supported on domains with complicated geometries. Notably, these efforts include Gaussian field approaches \cite{LingrenRueLindstrom2011,Bakka_et_al2019,Niu_et_al2019}, PDE regularization schemes \cite{Azzimonti_et_al2014,Azzimonti_et_al2015,Sangalli2021,Ferraccioli_et_al2021,Arnone_et_al2022}, shape constraints \cite{Feng_et_al2021,XuSamworth2021}, soap film smoothing \cite{WoodBravingtonHedley2008}, spline methods \cite{Ramsay2002,WangRanalli2007,LaiSchumaker2007,GuillasLai2010,SangalliRansayRamsay2013,LaiWang2013,MillerWood2014,ZhouPan2014,Yu_et_al2021}, diffusion kernels \cite{BarryMcIntyre2011,McSwigganBaddeleyNair2017,McIntyreBarry2018,BarryMcIntyre2020}, and other hybrid methods such as the complex region spatial smoother (CReSS) of \cite{Scott-Hayward_et_al2015} based on improved geodesic distance estimation.

    Unfortunately, all these methods are either intrinsically adapted for global estimation instead of pointwise estimation (and thus cannot always uphold the best choice of bandwidth near specific points), or they have only been developed, implemented and studied for spatial data in two or three dimensions.
    For instance, the state of the art PDE regularization method on Euclidean spaces is the one by \cite{Ferraccioli_et_al2021}. In that paper, the authors introduce a log-likelihood approach with Laplacian operator regularization that penalizes high local curvatures and thus controls the global roughness of the estimates. Their method can handle complicated shapes, sharp concavities and holes. They also improve on the method of \cite{BotevGrotowskiKroese2010} by discretizing the aforementioned Fokker-Planck equation and solving numerically using a forward Euler integration scheme, which allows for complicated domain shapes. The problem is that their global choice of regularization parameter limits refinements in the vicinity of any particular point in the domain. Furthermore, their general method relies crucially on a triangulation of the space, which is only numerically viable in low dimensions. The other approaches have the same issue, as they also rely on a global regularization or a low domain dimension. For the few methods in arbitrary dimensions (e.g., \cite{Niu_et_al2019, BarryMcIntyre2020}), some face practical challenges or lack theoretical development.

    The present paper addresses these gaps in the literature by introducing a new adaptive nonparametric local polynomial estimator suitable for the pointwise estimation of multivariate density functions supported on known domains of arbitrary dimensions and investigating its theoretical properties and practical implementation. Our method is highly flexible, as it applies to both simple domains, such as open connected sets, and more complicated domains that are not star-shaped around the point of estimation. This enables us to deal with domains containing sharp concavities, holes, and local pinches, such as polynomial sectors. To elaborate, our approach involves estimating the density function locally using the minimizer of a continuous kernel-weighted least-squares problem on a space of polynomials. Our selection procedure follows the general ideas of \cite{Lepski1991adaptive} and \cite{GoldenshlugerLepski2008,GoldenshlugerLepski2009,GoldenshlugerLepski2011,GoldenshlugerLepski2014}. It consists of jointly optimizing the polynomial degree and bandwidth over a discrete set and selecting the pair that minimizes the sum of upper bounds on the bias and a penalized version of the standard deviation. Importantly, our estimator does not require any tuning parameters, making it genuinely nonparametric. Moreover, it addresses the classical issues of kernel density estimators mentioned earlier, as it is locally adaptive, exhibits relatively small boundary bias regardless of the domain's shape, and is asymptotically a bona fide density function. For earlier advancements in density estimation on simple bounded domains using a similar Goldenshluger-Lepski selection procedure, refer to \cite{BertinKlutchnikoff2014,BertinElKoleiKlutchnikoff2019}, and also to \cite{BertinKlutchnikoff2017,Bertin_et_al2020} under the assumption of weakly dependent observations.

    Although our bandwidth selection procedure differs significantly, a similar concept was recently developed by \cite{CattaneoJanssonMa2020JASA} for estimating univariate density functions using local polynomials. Their approach involves taking the slope factor of the linear monomial in a kernel-weighted $L^2$ projection of the empirical cumulative distribution function (cdf) onto a polynomial space. In simpler terms, they project the cdf and identify the derivative of the projection (i.e., the density estimator), while our method directly projects the density function itself. Their estimator is boundary adaptive, does not require prior knowledge of the boundary points, and always yields a valid density function. Moreover, their implementation is simple, requiring no specific data modifications or additional tuning parameters. For the \texttt{R} software implementation, refer to \cite{CattaneoJanssonMa2022JSS}, and for an extension of their method to conditional univariate density functions, see \cite{CattaneoChandakJanssonMa2024}. However, one serious limitation of their method is its unclear generalizability to multivariate density functions.


    Here is an outline of the paper. Section~\ref{sec:framework} provides prerequisite definitions and information on the pointwise $L^2$ risk measure, domains and H\"older classes considered for density functions, as well as the concept of adaptive rate of convergence. Section~\ref{sec:statistical.procedure} presents a detailed description of our statistical procedure, which involves a kernel-weighted $L^2$ projection of the density function onto a polynomial space, and the application of the Goldenshluger-Lepski method for the joint selection of the polynomial degree and bandwidth. Section~\ref{sec:main.results} states our main results. For the base local polynomial estimator, this corresponds to explicit bounds on its bias and variance, and its minimaxity on simple and complicated domains. For our full statistical procedure (i.e., including the selection procedure), we derive an oracle-type inequality and the estimator is shown to achieve the optimal adaptive rate of convergence on simple and complicated domains. To aid understanding, the results for complicated domains are specifically stated for polynomial sectors. The case of unknown compact convex domains is also briefly discussed in Section~\ref{sec:unknown.domains}. Section~\ref{sec:simulations} complements the theoretical results with simulations on polynomial sectors. Our oracle estimates are shown to outperform those of the most popular alternative method, found in the \texttt{sparr} package \cite{Davies_et_al2024} for the \texttt{R} software. For convenience, our statistical procedure is implemented in an online \texttt{R} package called \texttt{densityLocPoly}, see \cite{BertinKlutchnikoffOuimet2023Rcode}. The proofs are all deferred to Section~\ref{sec:proofs}.

\section{Framework}\label{sec:framework}

    Define the domain $\domain$ as the closure of a nonempty open subset of $\R^d$ endowed with the sup norm, denoted by $\|\cdot\|_{\infty}$. Our focus will be on studying our estimator locally at a point $t\in \domain$, inside its surrounding neighborhoods of radius $h\in (0,\infty)$,
    \begin{equation*}
        \voisin(h) = \left\{u\in \R^d : t+u\in \domain ~\text{and}~ \|u\|_{\infty} \leq h\right\} = (\domain - t) \cap (h \Delta),
    \end{equation*}
    where $\Delta = [-1,1]^d$.
    We assume that there exists a parameter $\rho\in (0,e^{-1}]$ such that the set $\voisin(\rho)$ is a neighborhood of the origin for the topology on $\domain$ inherited from $\R^d$.
    Since $\domain$, $d$, $t$ and $\rho$ are fixed throughout the paper, the dependence on these quantities is often omitted for readability.

    For $n\in \N=\{1,2,\ldots\}$, let $\mathbb{X}_n = (X_1,\ldots,X_n)$ be a vector of $n$ independent copies of a random variable $X$ which is supported on $\domain$ and has a probability density function $f$ with respect to the Lebesgue measure on $\R^d$, hereafter denoted by $\mathrm{Leb}(\cdot)$. Our goal is to estimate the density function $f$ at the fixed point $t\in \domain$. By estimator, we mean any map $\tilde{f} = S \circ \mathbb{X}_n$ such that $S:\domain^n \to \smash{\R^{\R^d}}$ is Borel-measurable.
    The accuracy of such an estimator is measured by the pointwise risk
    \begin{equation*}
        \pwrisk(\tilde{f}, f) = \left(\EE\left[\big\{\tilde{f}(t) - f(t)\big\}^2\right]\right)^{1/2},
    \end{equation*}
    where $\EE$ denotes the expectation with respect to the probability measure $\PP$ of the random sample $\mathbb{X}_n$.

    \begin{remark}\label{rem:decomposition}\upshape
        Utilizing a pointwise $L^2$ risk measure is not inherently essential for deriving our findings. However, this choice facilitates the mathematical analysis across various proofs by decomposing the squared risk into two distinct components: the squared bias and the variance:
        \begin{equation*}
            \big\{\pwrisk(\tilde{f}, f)\big\}^2 = \left[\EE\big\{\tilde{f}(t)\big\} - f(t)\right]^2 + \Var\big\{\tilde{f}(t)\big\}.
        \end{equation*}
        This breakdown simplifies the determination of a suitable bandwidth parameter when seeking the optimal convergence rate of our statistical procedure. Such arguments feature prominently in the proofs of our minimax and adaptivity results; see Propositions~\ref{prop:minimaxity.simple.domains}~and~\ref{prop:minimaxity.polynomial.sectors}, as well as Theorems~\ref{thm:adaptivity.simple.domains}~and~\ref{thm:adaptivity.polynomial.sectors}. More generally, one could explore a pointwise $L^p$ risk measure for a fixed $p\geq 1$. However, this approach necessitates additional steps at various junctures to accommodate the imperfect decomposition:
        \begin{equation*}
            \EE\left[\big\{\tilde{f}(t) - f(t)\big\}^p\right] \leq 2^{p-1} \times \left\{\left|\EE\big\{\tilde{f}(t)\big\} - f(t)\right| + \EE\left(\big[\tilde{f}(t) - \EE\big\{\tilde{f}(t)\big\}\big]^p\right)\right\}.
        \end{equation*}
        For instance, we would need to apply Marcinkiewicz–Zygmund inequalities \citep[p.~498]{MR1368405} to estimate the second expectation on the right-hand side in terms of the variance. Modifications to the proof of Theorem~\ref{thm:oracle.result} would also be required, shifting the focus from $\EE(T^2)$ to $\EE(T^p)$, etc. While it is possible to consider a more general risk measure, doing so would inevitably convolute the mathematical methods and ideas used in deriving our results with additional technical intricacies. Therefore, this avenue is not further explored here and is left open for future investigation.
    \end{remark}

    Since there is no ambiguity, we use $\PP$ and $\EE$ instead of the more cumbersome notations $\PP_{\!f,n}$ and $\EE_{f,n}$. However, we will maintain the dependency on $n$ for the pointwise risk $\pwrisk$ as we will investigate its asymptotic behavior in Section~\ref{sec:main.results}. Specifically, we will study the minimax and adaptive pointwise estimation of $f$ over a large collection of H\"older-type functional classes. The remainder of this section covers the fundamental concepts that are used in this paper.

\subsection{H\"older-type classes}

    First, consider a collection of polynomials, each having a degree which is less than or equal to a specified integer.

    \begin{definition}
        Let $m\in \N_0 = \N\cup\{0\}$ and define
        \begin{equation*}
            \mathcal{P}_{\!m} = \operatorname{Span}\left(\varphi_{\alpha} : \alpha\in \N_0^d ~\text{and}~ |\alpha| \leq m\right),
        \end{equation*}
        where, for any $u\in \R^d$, we have the monomials
        \begin{equation*}
            \varphi_{\alpha}(u) = \prod_{j=1}^d u_j^{\alpha_j}
            \quad \text{with} \quad
            \lvert\alpha\rvert = \sum_{j=1}^d \alpha_j.
        \end{equation*}
        In particular, note that $\varphi_0 \equiv 1$.
    \end{definition}

    Second, for any smoothness parameter $s\in (0,\infty)$ and any Lipschitz constant $L\in (0,\infty)$, define the H\"older-type functional class $\Sigma(s,L)$ inside which a density function $f$ can be approximated by a polynomial of prescribed degree in a neighborhood of $t\in \domain$.

    \begin{definition}\label{def:Sigma.s.L}
        Let $(s,L)\in (0,\infty)^2$ be given.
        A density function $f:\domain\to\R$ belongs to the H\"older class $\Sigma(s,L)$ if there exists a polynomial $q\in \mathcal{P}_{\llfloor s \rrfloor}$ such that, for any $u\in \voisin(\rho)$,
        \begin{equation}\label{eq:approx-holder}
            |f(t+u) - q(u) | \leq L\|u\|_{\infty}^s.
        \end{equation}
        Here, $\llfloor s \rrfloor$ denotes the greatest integer less than $s$. In particular, $\llfloor s \rrfloor=s-1$ if $s\in \N$.
    \end{definition}

    \begin{remark}\upshape
        The functional class $\Sigma(s,L)$ is not commonly used in statistics.
        However, we believe that it offers more flexibility than classical H\"older spaces in the context of a domain $\domain$ with a complicated boundary. A few comments are in order:
        \begin{enumerate}
            \item Knowing that a density function $f$ belongs to the H\"older class $\Sigma(s,L)$ only provides local information (in a neighborhood of $t$) about the regularity of $f$. Note also that the polynomial $q$  satisfies $q(0) = f(t)$.
            \item If the partial derivatives $\smash{D^{\alpha} f = {\partial^{|\alpha|} f}/{\partial x_1^{\alpha_1}\cdots\partial x_d^{\alpha_d}}}$ exist for all $|\alpha|\leq \llfloor s \rrfloor$ and if the interior of $\voisin(\rho)$, in the topology of $\R^d$, is star-shaped with respect to the origin, then the Taylor polynomial of $f$ at $t$ can be substituted for $q$ in~\eqref{eq:approx-holder}. In particular, this is true if the interior of $\voisin(\rho)$ is an open convex subset of $\R^d$ that contains the origin.
            \item The existence of $q$ does not imply the existence of directional derivatives for $f$. In particular, for some points on the boundary of the domain $\domain$, the partial derivatives (in the directions given by the canonical basis) may not even make sense. Let us give two simple examples of densities defined on a disk sector to illustrate these assertions. Define
                \begin{equation*}
                    \domain = \left\{(r\cos\theta,r\sin\theta)\in \R^2 : 0\leq r\leq 1, \frac{\pi}{6}\leq\theta\leq\frac{\pi}{3}\right\}.
                \end{equation*}
                The density function $f_1(x,y) = 12/\pi\times\1_{\domain}(x,y)$ is constant on $\domain$, but the partial derivatives do not exist at $0\in \domain$. Similarly, the density function $f_2(x,y) = A g(x+y) \1_{\domain}(x,y)$, where $g(u) = 1 + u + u^2 + u^3\sin(1/u)$ and $A$ is a normalization constant, can be approximated in a neighborhood of $0\in \domain$  by a polynomial of degree $2$, even if no directional derivatives of order $2$ exist at the origin.
        \end{enumerate}
    \end{remark}

    \begin{remark}\upshape
        Because the degrees of the polynomials in the basis $\mathcal{P}_{\llfloor s \rrfloor}$ are restricted by the value $\llfloor s \rrfloor$ across all dimensions, the functional class $\Sigma(s,L)$ can be categorized as isotropic. One might ponder the possibility of defining an anisotropic extension of $\Sigma(s,L)$ to provide greater flexibility. Here is one approach: For any integers $m_1,\ldots,m_d\in \N_0$, define the new polynomial basis
        \begin{equation*}
            \mathcal{P}_{\!m_1,\ldots,m_d} = \operatorname{Span}\left(\varphi_{\alpha} : \alpha\in \N_0^d ~\text{and}~ \alpha_i \leq m_i ~~\forall i\in \{1,\ldots,d\}\right).
        \end{equation*}
        Then, for any smoothness parameters $s_1,\ldots,s_d\in (0,\infty)$ and any Lipschitz constant $L\in (0,\infty)$, let the new anisotropic H\"older class $\Sigma(s_1,\ldots,s_d,L)$ be defined by the polynomials $q\in \smash{\mathcal{P}_{\llfloor s_1 \rrfloor, \ldots, \llfloor s_d \rrfloor}}$ such that, for any $u\in \voisin(\rho)$,
        \begin{equation*}
            |f(t+u) - q(u)| \leq L (|u_1|^{s_1} + \dots + |u_d|^{s_d}).
        \end{equation*}
        This new definition would naturally lead in Section~\ref{sec:statistical.procedure} to the adoption of a more general parameter $\gamma = (m_1,\ldots,m_d,h_1,\ldots,h_d)$ for our adaptive statistical procedure, where the $m_i$'s and $h_i$'s specify the estimator's smoothness and scale of estimation in each dimension, respectively. Consequently, numerous technical adjustments would be necessary in the statements of the results and their proofs, such as the additivity of the bias terms $\smash{h_i^{\beta_{m_i}(s_i)}}$ in Proposition~\ref{prop:control.bias}. While the conceptual framework would remain largely unchanged, these technical modifications would obscure the mathematical methodologies and ideas underpinning the derivation of the new results. This avenue is not further explored here, remaining open for future research. For recent advancements in multivariate density estimation within an anisotropic framework, refer to, e.g., \cite{Rebelles2015,BertinElKoleiKlutchnikoff2019,LiuWu2019,AmmousDedeckerDuval2024,VaretLacourMassartRivoirard2023}.
    \end{remark}

\subsection{Adaptive estimation}

    As mentioned at the beginning of Section~\ref{sec:framework}, we are interested in the pointwise adaptive estimation, in a minimax framework, of the density function $f$ over the collection of H\"older-type functional classes $\Sigma(s,L)$, where the nuisance parameter $(s,L)$ is assumed to belong to $\mathcal{K} = (0,\infty)^2$.
    Since the results of \cite{Lepski1991adaptive}, it is well-known that, in this context, it is impossible to construct a single estimation procedure $f_n^{\star}$ which attains the minimax rate of convergence
    \begin{equation}\label{eq:minimax.rate}
        N_n(s,L) = \inf_{\tilde{f}_n} \sup_{f\in \Sigma(s,L)}\pwrisk(\tilde{f}_n, f)
    \end{equation}
    simultaneously for all functional classes, i.e., which satisfies
    \begin{equation}\label{eq:minimax.rate.attained}
        \sup_{f\in \Sigma(s,L)}\pwrisk(f_n^{\star}, f) \lesssim N_n(s,L), \quad \text{for all } (s,L)\in \mathcal{K}.
    \end{equation}
    In~\eqref{eq:minimax.rate}, the infimum is taken over all possible estimators $\tilde{f}_n$ of $f$, while the general notation $u_n \lesssim v_n$ used in~\eqref{eq:minimax.rate.attained} means that $u_n,v_n \geq 0$ and $\limsup_{n\to\infty} u_n / v_n <\infty$.

    It is therefore necessary to define a specific notion of adaptive rate of convergence (ARC) before trying to construct an optimal estimation procedure. Several definitions were proposed in the last decades, see, e.g., \cite{Lepski1991adaptive}, \cite{Tsybakov1998}, \cite{K-MMS-2014} and \cite{Rebelles2015}. In this paper we adopt the definition of \cite{K-MMS-2014}, which coincides in our framework with the refined version introduced by \cite{Rebelles2015}. In order to be self-contained, we recall the definition of admissibility of a collection of normalizations and the notion of ARC employed in these two papers. It is always assumed that collections such as $\phi = \{\phi_n(s,L) : (s,L)\in \mathcal{K}, n\in \N\}$ have nonnegative components, i.e., $\phi_n(s,L)\geq 0$ for all $(s,L)\in \mathcal{K}$ and all $n\in \N$.

    \begin{definition}\label{def:admissibility}
        A collection $\phi = \{\phi_n(s,L) : (s,L)\in \mathcal{K}, n\in \N\}$ is said to be admissible if there exists an estimator $\tilde{f}_n$ such that
        \begin{equation*}
            \sup_{f\in \Sigma(s,L)}\pwrisk(\tilde{f}_n, f)\lesssim \phi_n(s,L),
            \quad
            \text{for all } (s,L)\in \mathcal{K}.
        \end{equation*}
    \end{definition}

    Given two collections $\phi = \{\phi_n(s,L) : (s,L)\in \mathcal{K}, n\in \N\}$ and  $\psi = \{\psi_n(s,L) : (s,L)\in \mathcal{K}, n\in \N\}$, we define two subsets of $\mathcal{K}$, namely
    \begin{equation*}
        [\psi \ll \phi] = \left\{(s,L)\in \mathcal{K} : \lim_{n\to\infty} \frac{\psi_n(s,L)}{\phi_n(s,L)} = 0\right\},
    \end{equation*}
    and
    \begin{equation*}
        [\psi \ggg \phi] = \left\{(s,L)\in \mathcal{K} : \limsup_{n\to\infty} \frac{\psi_n(s,L)}{\phi_n(s,L)}\times\frac{\psi_n(s',L')}{\phi_n(s',L')} =\infty, ~~ \forall (s',L')\in[\psi \ll \phi]\right\}.
    \end{equation*}
    Equipped with these notations, the notion of ARC is defined as follows.

    \begin{definition}\label{def:ARC}
        An admissible collection $\phi = \{\phi_n(s,L) : (s,L)\in \mathcal{K}, n\in \N\}$ is called an ARC if for any other admissible collection $\psi = \{\psi_n(s,L) : (s,L)\in \mathcal{K}, n\in \N\}$, we have
        \begin{itemize}
            \item $[\psi\ll\phi]$ is contained in a $(\mathrm{dim}(\mathcal{K})-1)$--dimensional  manifold.
            \item $[\psi\ggg\phi]$ contains an open subset of $\mathcal{K}$.
        \end{itemize}
    \end{definition}

    \begin{remark}\upshape
        In our context, where $\mathcal{K} = (0,\infty)^2$, we can grasp the essence of Definition~\ref{def:ARC} heuristically as follows: Given an admissible collection $\phi$, the set $[\psi\ll\phi]$ comprises collections $\psi$ exhibiting a superior rate of convergence compared to $\phi$. On the other hand, $[\psi\ggg\phi]$ encapsulates collections that suffer from a worse rate of convergence than all the multiplicative inverse gaps observed in $[\psi\ll\phi]$. Consequently, Definition~\ref{def:ARC} designates $\phi$ as an ARC if $[\psi\ll\phi]$, akin to an optimal frontier, manifests as $1$-dimensional. This implies that for a given $s$, the convergence rate of $\phi_n(s,L)$ cannot be enhanced. Conversely, the condition that $[\psi\ggg\phi]$ contains an open subset of $(0,\infty)^2$ suggests ample space for collections exhibiting sub-optimal convergence rates in comparison to $\phi$. To illustrate visually, envision $s$ along the $x$-axis and $L$ along the $y$-axis. It is shown in Theorem~\ref{thm:adaptivity.simple.domains} that on simple domains, if $\phi$ is an ARC, then the set $[\psi\ll\phi]$ is equal to the vertical line $\{s_0\} \times (0,\infty)$ for an appropriate $s_0 > 0$, and $[\psi\ggg\phi]$ corresponds to the half-space $(s_0,\infty) \times (0,\infty)$.
    \end{remark}

    Let us recall that the ARC is unique up to asymptotic order, i.e., if $\phi$ and $\psi$ are two ARCs then, for any $(s,L)\in \mathcal{K}$, we have
    \begin{equation*}
        \phi_n(s,L) \lesssim \psi_n(s,L)
        \quad \text{and} \quad
        \psi_n(s,L) \lesssim \phi_n(s,L).
    \end{equation*}
    Henceforth, we refer to an ARC as \textit{the} ARC using the above identification.
    Further, note that if the collection of minimax rates $N = \{N_n(s,L) : (s,L)\in \mathcal{K}, n\in \N\}$ is admissible, then it is the ARC.

\section{Statistical procedure}\label{sec:statistical.procedure}

    In this section, we present a novel estimation procedure which is both simple to understand and completely free of any tuning parameters, thus eliminating the need for a sophisticated calibration step. The main idea behind our approach is to use polynomials of varying degrees to approximate the density function $f$ in different neighborhoods of $t$ and then to simultaneously select the optimal neighborhood and degree of approximation in a data-driven manner. Let us formalize this idea.

\subsection{Local polynomial estimators}\label{sec:locpoly}

    Our collection of local polynomial estimators will be indexed by a two-dimensional parameter
    \begin{equation*}
        \gamma = (m,h)\in \Gamma = \N_0\times (0,\rho],
    \end{equation*}
    where $m$ represents the degree of the polynomial and $h$ denotes the bandwidth, the latter of which determines the size of the neighborhoods.
    Let $\1_A$ be the indicator function for any set $A\subseteq \R^d$.
    Recall that $\Delta = [-1,1]^d$, and consider the kernel function
    \begin{equation}\label{eq:K}
        K(u) = \1_{\Delta}(u), \quad u\in \R^d.
    \end{equation}
    Note that the kernel $K$ does not have to be a density function. While other positive kernels could be used, we limit our study to this particular kernel for the sake of simplicity. Let us define the polynomial $p_{\gamma}$ as the solution to the following continuous kernel-weighted least-squares minimization problem:
    \begin{equation*}
        p_{\gamma} = \mathop{\mathrm{arg\,min}}_{p\in \mathcal{P}_{\!m}} \int_{\voisin(h)} \left\{f(t+u) - p(u)\right\}^2 w_h(u) \rd u,
    \end{equation*}
    where
    \begin{equation}\label{eq:w.h}
        w_h(u) = h^{-d} K(h^{-1} u) \1_{\domain}(t + u).
    \end{equation}
    Notice that the support of $w_h$ corresponds to the neighborhood $\voisin(h)$.
    If we consider the weighted space $L^2(w_h)$ endowed with its natural inner product
    \begin{equation*}
        \langle g,\tilde{g} \rangle =\int_{\R^d} g(x) \tilde{g}(x) w_h(x) \rd x,
    \end{equation*}
    then the polynomial $p_{\gamma}$ can be interpreted as the orthogonal projection of the function $f(t+\cdot)$ onto the polynomial space $\mathcal{P}_{\!m}\subseteq L^2(w_h)$.
    A simple combinatorial argument shows that the dimension $D_m$ of $\mathcal{P}_{\!m}$ is the number of $d$-combinations from a set of $m+d$ elements. For $\gamma=(m,h)\in \Gamma$, consider the collection of rescaled monomials
    \begin{equation*}
        \Phi_{\gamma}(u) = \left\{\varphi_{\alpha}\left(\frac{u}{h}\right)\right\}_{|\alpha|\leq m}
    \end{equation*}
    which we organize into a $D_m\times 1$ vector using the total order defined for $\alpha \neq \tilde{\alpha}$ by $\alpha \prec \tilde{\alpha}$ if and only if $|\alpha| < |\tilde{\alpha}|$, or $|\alpha|=|\tilde\alpha|$ and $\alpha_{i^{\star}} < \tilde{\alpha}_{i^{\star}}$ for $i^{\star} = \min(i\in \{1,\ldots,d\} : \alpha_i \neq \tilde{\alpha}_i)$.
    Now that $\Phi_{\gamma}(u)$ is a vector function, define the $D_m\times D_m$ Gram matrix
    \begin{equation}\label{eq:matrix.B.gamma}
        \mathcal{B}_{\gamma} =\int_{\R^d} \Phi_{\gamma}(u) \Phi_{\gamma}^{\top}(u)w_h(u) \rd u.
    \end{equation}
    By virtue of Lemma~\ref{lem:positive.definite}, the Gram matrix $\mathcal{B}_{\gamma}$ is symmetric positive definite, so its smallest eigenvalue
    \begin{equation}\label{eq:lambda.gamma}
        \lambda_{\gamma} = \min_{v^{\top} v = 1} v^{\top} \mathcal{B}_{\gamma} v
    \end{equation}
    must be positive. Letting $\mathcal{B}_{\gamma}^{1/2}$ be the lower triangular matrix in the Cholesky decomposition of $\mathcal{B}_{\gamma}$ leads to the following orthonormal basis for the $\smash{L^2(w_h)}$ space:
    \begin{equation}\label{eq:orthogonalization}
        H_{\gamma}(u) = \mathcal{B}_{\gamma}^{-1/2} \Phi_{\gamma}(u),
        \quad u\in \R^d.
    \end{equation}
    If $a_{\gamma}$ denotes the $D_m\times 1$ coordinate vector of $p_{\gamma}$ in this orthonormal basis, then
    \begin{equation*}
        a_{\gamma}
        = \langle f(t+\cdot), H_{\gamma} \rangle
        = \EE \left\{H_{\gamma}(X-t) w_h(X-t)\right\}
        \approx \frac{1}{n}\sum_{i=1}^n H_{\gamma}(X_i-t) w_h(X_i-t)
        \eqqcolon \hat{a}_{\gamma}.
    \end{equation*}
    This leads to $\hat{p}_{\gamma}(u) = H_{\gamma}^{\top}(u) \hat{a}_{\gamma}$ being an estimator for $p_{\gamma}$.
    Therefore, for any $\gamma = (m,h)\in \Gamma$, we define our local polynomial estimator as follows:
    \begin{equation}\label{eq:hat.f.gamma}
        \hat{f}_{\gamma}(t) = \hat{p}_{\gamma}(0) = \frac{1}{n} \sum_{i=1}^n H_{\gamma}^{\top}(0) H_{\gamma}(X_i-t) w_h(X_i-t).
    \end{equation}

\subsection{Selection procedure}\label{sec:selection-procedure}

    The data-driven procedure selects a local polynomial estimator from a collection $\{\hat{f}_{\gamma}(t): \gamma\in \Gamma_{\!n}\}$, where $\Gamma_{\!n}$ is a discrete one-parameter subset of the pairs of polynomial degrees and bandwidths in $\Gamma$, see~\eqref{eq:Gamma.n} below.
    For any $\ell\in \N$, set
    \begin{equation}\label{eq:h.in.Gamma.n}
        h_{\ell} = \rho\exp(-\ell),
    \end{equation}
    and let $\{m_{\ell} : \ell\in \N\}\subseteq \{0,1,\ldots,\lfloor\log n\rfloor\}$ be a nonincreasing sequence of integers. These choices put approximately $\log n$ bandwidth levels on a logarithmic scale down to a smallest neighborhood of radius at least $\rho n^{-1}$ asymptotically. Consider the set of indexed pairs
    \begin{equation}\label{eq:Gamma.n}
        \Gamma_{\!n} = \{(m_{\ell}, h_{\ell}) : \ell\in \mathcal{L}_n\}\subseteq \Gamma,
    \end{equation}
    where
    \begin{equation}\label{eq:L.n.W.h}
        \mathcal{L}_n = \big\{1 \leq \ell \leq \lfloor\log n\rfloor : n h^d_{\ell} W_{h_{\ell}}\geq (\log n)^3\big\}
        \quad \text{with} \quad
        W_h =\int_{\R^d} w_h(u) \rd u.
    \end{equation}
    To provide clarity for the selection process, we introduce certain observables.
    Specifically, for each $\gamma=(m,h)\in \Gamma_{\!n}$, the quantity
    \begin{equation}\label{eq:hat.v.gamma}
        \hat{v}_{\gamma} = \frac{1}{n} \times \frac{1}{n} \sum_{i=1}^n \left\{H_{\gamma}^{\top}(0) H_{\gamma}(X_i-t) w_h(X_i-t)\right\}^2
    \end{equation}
    is an estimator of a natural upper bound on the variance of $\hat{f}_{\gamma}(t)$.
    Also, for some fixed constant $\delta > 1$, and the pairs $\gamma=(m,h)\in \Gamma_{\!n}$ and $(v,x)\in [0,\infty)^2$, define
    \begin{equation}\label{eq:c.gamma}
        c_{\gamma} = \frac{\sqrt{D_m}}{n h^d \lambda_{\gamma}},
        \qquad
        \varepsilon_{\gamma} = \frac{(\delta-1) D_m W_h}{n h^d \lambda_{\gamma}^2},
        \qquad
        \Lambda_{\gamma} = 2 \, \lvert\log(\lambda_{\gamma})\rvert,
    \end{equation}
    and
    \begin{equation}\label{eq:r.gamma.and.pen.gamma}
        r_{\gamma}(v,x) = \sqrt{2vx} + c_{\gamma} x,
        \qquad
        \mathrm{pen}(\gamma) = d\delta\lvert\log h\rvert + \Lambda_{\gamma}.
    \end{equation}
    Equipped with these notations, consider an upper bound estimator of the bias of $\hat{f}_{\gamma}(t)$, i.e.,
    \begin{equation}\label{eq:hat.A.gamma}
        \hat{A}_{\gamma} = \max_{\gamma'\in \Gamma_{\!n}} \left\{\left|\hat{f}_{\gamma\vee\gamma'}(t) - \hat{f}_{\gamma'}(t)\right| - \hat{\mathbb{U}}_{\gamma\vee\gamma'} - \hat{\mathbb{U}}_{\gamma'}\right\}_+,
    \end{equation}
    where $(\cdot)_+ = \max\{\cdot \, ,0\}$ and $\gamma\vee\gamma'$ denotes the maximum of $\gamma=(m_{\ell},h_{\ell})$ and $\gamma'=(m_{\ell'},h_{\ell'})$ with respect to the total order $\preceq$ defined on $\Gamma_{\!n}$ by $\gamma\preceq\gamma'$ if $h_{\ell} \leq h_{\ell'}$ (i.e., if $\ell\geq \ell'$).
    Moreover, consider a penalized version of the estimator of an upper bound on the standard deviation estimator of $\hat{f}_{\gamma}(t)$,
    \begin{equation}\label{eq:hat.U.gamma}
        \hat{\mathbb{U}}_{\gamma} = r_{\gamma}\left\{\hat{v}_{\gamma} +\varepsilon_{\gamma}, \mathrm{pen}(\gamma)\right\}.
    \end{equation}
    The final adaptive estimator is then defined by
    \begin{equation}\label{eq:selection.rule}
        \hat{f}(t) = \hat{f}_{\hat{\gamma}}(t),
        \quad \text{where} \quad
        \hat{\gamma} = \mathop{\mathrm{arg\,min}}_{\gamma\in \Gamma_{\!n}} \big(\hat{A}_{\gamma} + \hat{\mathbb{U}}_{\gamma}\big).
    \end{equation}

    \begin{remark}\upshape
        Before we state our main results in Section~\ref{sec:main.results}, it is worthwhile to first discuss a few points that are particularly relevant:
        \begin{enumerate}
            \item The selection rule~\eqref{eq:selection.rule} is inspired by the so-called Goldenshluger-Lepski method introduced in a series of papers, see \cite{Lepski1991adaptive} and \cite{GoldenshlugerLepski2008,GoldenshlugerLepski2009,GoldenshlugerLepski2011,GoldenshlugerLepski2014}. It consists of a trade-off between the bias upper bound estimators $\hat{A}_{\gamma}$ and the penalized standard deviation upper bound estimators $\hat{\mathbb{U}}_{\gamma}$. Finding tight upper variables $\hat{\mathbb{U}}_{\gamma}$ is the key point of this selection procedure. Ideally, we would like to choose
                \begin{equation}\label{eq:U.gamma}
                    \mathbb{U}_{\gamma} = r_{\gamma}\{v_{\gamma},\mathrm{pen}(\gamma)\},
                \end{equation}
                where the counterpart of $\hat{v}_{\gamma}$ in~\eqref{eq:hat.v.gamma}, namely
                \begin{equation}\label{eq:v.gamma}
                    v_{\gamma} = \frac{1}{n} \times \EE\left[\left\{H_{\gamma}^{\top}(0)H_{\gamma}(X-t) w_h(X-t)\right\}^2\right],
                \end{equation}
                is a natural upper bound on the true variance of $\hat{f}_{\gamma}(t)$.
            \item The form of $\mathbb{U}_{\gamma}$ can be easily explained. The function $r_{\gamma}$ in~\eqref{eq:r.gamma.and.pen.gamma} is used to apply a Bernstein inequality to control the bias term $\hat{A}_{\gamma}$, see Lemma~\ref{lem:Bernstein.1} and its proof for details. The penalty $\mathrm{pen}(\gamma)$ consists of two terms: the $\lvert\log h\rvert$-term is unavoidable in pointwise adaptive estimation (see \cite{Lepski1991adaptive}) while the quantity $\Lambda_{\gamma}$, which adapts to the geometry of the domain $\domain$, is specific to our framework. To define $\hat{\mathbb{U}}_{\gamma}$, the variance bound $v_{\gamma}$ is replaced by its estimator $\hat{v}_{\gamma}$ and the correction term $\varepsilon_{\gamma}$ is added to ensure that, with large probability, $\hat{\mathbb{U}}_{\gamma}$ becomes larger than $\mathbb{U}_{\gamma}$ while staying close to it.
            \item The Goldenshluger-Lepski method is generally defined using an order on $\Gamma_{\!n}$ which is, roughly speaking, induced by the variance of the estimators $\hat{f}_{\gamma}(t)$. This explains the definition of the total order $\gamma\preceq\gamma'$ stated just below \eqref{eq:hat.A.gamma}, which involves only the bandwidth $h$ in the comparison of the parameter pairs $(m,h)$ in $\Gamma_{\!n}$.
        \end{enumerate}
    \end{remark}

\section{Main results}\label{sec:main.results}

    Our main results are presented in three steps. First, in Section~\ref{sec:general.results}, we obtain bounds on the bias and variance of the base local polynomial estimators $\hat{f}_{\gamma}(t)$ defined in \eqref{eq:hat.f.gamma}, and we formulate an oracle-type inequality satisfied by the adaptive estimator $\hat{f}(t)$ defined in \eqref{eq:selection.rule}. Next, in Section~\ref{sec:simple.domain}, we focus on the most common case, where the geometry of $\domain$ is `simple' in a neighborhood of $t$, meaning that Assumption~\ref{ass:1} is satisfied. In this situation, the minimax and adaptive rates of convergence are established over the whole collection of H\"older-type functional classes $\{\Sigma(s,L) : (s,L)\in (0,\infty)^2\}$. In Section~\ref{sec:complicated.domain}, we consider a specific situation where the domain $\domain$ has a more complicated geometry. To aid understanding, the results are specifically stated for polynomial sectors. We obtain upper bounds on the rate of convergence in each class $\Sigma(s,L)$, which depend on the geometry of $\domain$, and we prove that, for small regularities $s$, these rates of convergence are minimax. Analogous adaptivity results are also derived. Finally, in Section~\ref{sec:unknown.domains}, the case of unknown compact convex domains is briefly discussed.

\subsection{General results}\label{sec:general.results}

    Our first two results pertain to the study of the bias and variance of each estimator in the collection $\{\hat{f}_{\gamma}(t) : \gamma\in \Gamma\}$. The proofs of these propositions are postponed to Section~\ref{sec:proofs}.

    \begin{proposition}[Bound on the bias]\label{prop:control.bias}
        Let $\gamma=(m,h)\in \Gamma$ be given.
        Assume that the density function $f:\domain\to\R$ belongs to the class $\Sigma(s,L)$ for some positive $s$ and $L$. Then, there exists a positive real constant $\mathfrak{L}_{m,s,L}$  that depends on $m$, $s$ and $L$ such that
        \begin{equation}\label{eq:bias.bound}
            \left|\EE\left\{\hat{f}_{\gamma}(t)\right\} - f(t)\right|
            \leq \frac{W_h\sqrt{D_m}}{\lambda_{\gamma}} \times \mathfrak{L}_{m,s,L} \, h^{\beta_m(s)},
        \end{equation}
        where $\beta_m(s) = \min(m+1, s)$.
        Moreover, we have $\mathfrak{L}_{m,s,L} = L$ as soon as $m\geq\llfloor s \rrfloor$.
    \end{proposition}

    \begin{proposition}[Bound on the variance]\label{prop:control.stochastic.term}
        Let $\gamma=(m,h)\in \Gamma$ be given.
        Assume that the density function $f:\domain\to\R$ satisfies
        \begin{equation}\label{eq:f.bounded}
            \finfty \coloneqq \sup_{u\in \voisin(\rho)} \lvert f(t+u) \rvert <\infty.
        \end{equation}
        Also, recall the definition of $v_{\gamma}$ from \eqref{eq:v.gamma}. Then, we have
        \begin{equation}\label{eq:variance.bound}
            \Var\left\{\hat{f}_{\gamma}(t)\right\}
            \leq v_{\gamma}
            \leq v_{\gamma}^{\star},
            \quad \text{with} \quad
            v_{\gamma}^{\star} = \left(\frac{W_h\sqrt{D_m}}{\lambda_{\gamma}}\right)^2 \times \frac{\finfty}{n h^d W_h}.
        \end{equation}
    \end{proposition}

    Let us highlight that both the bias bound~\eqref{eq:bias.bound} and the variance bound~\eqref{eq:variance.bound} can be broken down into two separate factors. The second factors, which are proportional to $h^{\beta_m(s)}$ and $1 / (n h^d W_h)$, can be readily explained as they correspond (up to a constant) to the bias and variance bounds of multivariate histograms with adaptive bin shapes. These histograms are defined as $\hat{p}_{\mathcal{B}} = \sum_{i=1}^n \1_{\mathcal{B}}(X_i)/\{n\mathrm{Leb}(\mathcal{B})\}$ on each multivariate bin $\mathcal{B}$ belonging to a partition of $\domain$, as documented for example by \cite{Klemala2009}. Specifically, when considering the bin $\mathcal{B} = \domain\cap(t+h\Delta)$, it is straightforward to show that the bias and the variance are proportional to $h^{\beta_m(s)}$ and $1 / (n h^d W_h)$, respectively, just as we would with our estimator when the polynomial degree is set to $m=0$. Notice however that in general the first factors introduce an additional quantity, the ratio $W_h / \lambda_{\gamma}$, which is not usually present and depends on the geometry of the domain $\domain$. In the case of simple domains, which are treated in Section~\ref{sec:simple.domain}, the ratio $W_h / \lambda_{\gamma}$ remains asymptotically bounded, thereby exerting no influence on the bias/variance analysis.

    In the case of complicated domains, which are treated in Section~\ref{sec:complicated.domain}, we know that the ratio $ W_h / \lambda_{\gamma} $ encodes the complexity of the domain in some way. Unfortunately, we do not have a heuristic interpretation that would let us guess its asymptotic behavior in the variable $n$ (or $h$) in general. This remains an open problem. However, in the specific case of the polynomial sectors $\mathcal{D}_k$ covered in Section~\ref{sec:complicated.domain}, we do have a (potential) interpretation based on the calculations in Equations~\eqref{eq:W.h.calculation}, \eqref{eq:v.top.B.v.decomp} and \eqref{eq:minoration.lambda}. Indeed, the quantity $W_h$ is the area of the domain at the scale of the bandwidth parameter, so the specific result $W_h = h^{k-1}/(k+1)$ in \eqref{eq:W.h.calculation} gives a `unit' of the pointiness of the polynomial sectors. Equations~\eqref{eq:v.top.B.v.decomp} and \eqref{eq:minoration.lambda} then show that each degree of a polynomial in our basis can capture one such unit. Since there are two dimensions and the polynomials in our space have degree at most $m$ (recall that $\gamma = (m,h)$), we are left with $W_h / \lambda_{\gamma_k}$ being asymptotic to $\smash{h^{-(k-1)\times 2m}}$, where $\gamma_k \equiv \gamma$ is the notation tailored to the polynomial sector $\domain_k$. One could therefore interpret $\smash{h^{(k-1)\times 2m}}$ as the finest amount of pointiness that our polynomial basis can model. The finer the pointiness of the domain, the bigger the bias and variance become in Propositions~\ref{prop:control.bias}~and~\ref{prop:control.stochastic.term}.

    For the forthcoming minimax and adaptive results in Sections~\ref{sec:simple.domain}~and~\ref{sec:complicated.domain}, it is also worth noting that Lemma~\ref{lem:infty.norm.f} implies that $\finfty$ remains uniformly bounded over $f\in \Sigma(s,L)$ in the variance bound \eqref{eq:variance.bound}.

    The main result of this section is an oracle-type inequality which provides an explicit asymptotic bound for the pointwise $L^2$ risk of our adaptive estimator $\hat{f}(t)$.

    \begin{theorem}[Oracle-type inequality]\label{thm:oracle.result}
        Assume that the density function $f:\domain\to\R$ satisfies~\eqref{eq:f.bounded}. Assume further that the domain $\domain$ satisfies the following property: there exists a positive real constant $b\in (0,\infty)$ such that
        \begin{equation}\label{eq:domain.lambda}
            \min_{\gamma=(m,h)\in \Gamma_{\!n}} (h^m)^{-b} \lambda_{\gamma} \geq 1.
        \end{equation}
        Define, for any $\gamma\in \Gamma_{\!n}$,
        \begin{equation}\label{eq:B.gamma.and.U.star.gamma}
            \mathbb{B}_{\gamma} = \max_{\substack{\gamma'\in \Gamma_{\!n} \\ \gamma' \preceq \gamma}} \left|\EE\left\{\hat{f}_{\gamma'}(t)\right\} - f(t)\right|,
            \quad \text{and} \quad
            \mathbb{U}_{\gamma}^{\star} = r_{\gamma}\{v_{\gamma}^{\star}, \mathrm{pen}(\gamma)\},
        \end{equation}
        where recall $v_{\gamma}^{\star}$ is the bound on the variance in~\eqref{eq:variance.bound}, and the quantities $r_{\gamma}(\cdot,\cdot)$ and $\mathrm{pen}(\cdot)$ are defined in~\eqref{eq:r.gamma.and.pen.gamma}. Recall also that $\delta > 1$ is a parameter that we fixed above~\eqref{eq:r.gamma.and.pen.gamma}. Then, we have
        \begin{equation}\label{eq:thm:oracle.result.to.prove}
            \pwrisk(\hat{f}, f)
            \leq \min_{\gamma\in \Gamma_{\!n}} \left[5 \mathbb{B}_{\gamma} + \left\{3 + 2 \sqrt{(\delta-1) / \finfty}\right\} \mathbb{U}_{\gamma}^{\star}\right] + \mathcal{O}\left\{\sqrt\frac{(\log n)^d}{n}\right\}.
        \end{equation}
        The notation $u_n = \mathcal{O}(v_n)$ here means that $\limsup_{n\to\infty} |u_n / v_n| < \infty$.
    \end{theorem}

    \begin{remark}\upshape
        On domains with simple geometries, which are characterized by Assumption~\ref{ass:1} below, the technical condition~\eqref{eq:domain.lambda} in Theorem~\ref{thm:oracle.result} is automatically satisfied when $m$ is constrained to a finite set of values, a situation often encountered in practical scenarios where the target density has a finite smoothness parameter $s$. This assertion is demonstrated in Lemma~\ref{lem:usual.behavior}; see also the first paragraph in the proof of Theorem~\ref{thm:adaptivity.simple.domains}. However, on domains with complicated geometries, verifying the technical condition~\eqref{eq:domain.lambda} becomes a case-by-case endeavor. For instance, in the polynomial sector example detailed in Section~\ref{sec:complicated.domain}, explicit verification is provided within the proof of Theorem~\ref{thm:adaptivity.polynomial.sectors} to establish the adaptivity of our estimator.
    \end{remark}

\subsection{Domains with simple geometries}\label{sec:simple.domain}

    Below, we introduce an assumption that qualifies what it means for a domain $\domain$ to be geometrically `simple' in a neighborhood of $t$.
    \begin{assumption}\label{ass:1}
        There exists an open subset $\Delta_0\subseteq \Delta \equiv [-1,1]^d$ such that
        \begin{equation*}
            h \Delta_0\subseteq \voisin(h), \quad \text{for all } h \in (0,\rho].
        \end{equation*}
    \end{assumption}
    This assumption is satisfied for example if $\voisin(h)$ is star-shaped with respect to the origin and has a non-empty interior. Assumption~\ref{ass:1} is a relaxation of Assumption~5 in Section~2.2 of \cite{Bertin_et_al2020}, which stipulates that there exists a finite set of linear transformations of determinant $1$ such that for a hypercube $[0,r]^d$ of a small enough width $r$, you can, for any point $t\in \domain$, choose a linear transformation in the finite set so that the linearly transformed hypercube (i.e., a rotated parallelotope) fits into $\domain - t$. Assumption~\ref{ass:1} above is weaker because we could simply take $\Delta_0$ to be a small enough open cone contained in the rotated parallelotope that fits into $\domain - t$.

    Under Assumption~\ref{ass:1}, Proposition~\ref{prop:minimaxity.simple.domains} below establishes that our local polynomial estimator $\hat{f}_{\gamma}(t)$ achieves the minimax rate of convergence over each individual class $\Sigma(s,L)$. Similarly, Theorem~\ref{thm:adaptivity.simple.domains} proves that our adaptive statistical procedure $\hat{f}(t)$ attains the ARC over the whole collection $\{\Sigma(s,L) : (s,L)\in (0,\infty)^2\}$. The ARC here corresponds to the minimax rate of convergence up to a logarithmic factor.

    \begin{proposition}[Minimaxity of $\hat{f}_{\gamma}(t)$ on simple domains]\label{prop:minimaxity.simple.domains}
        Suppose that Assumption~\ref{ass:1} holds.
        Let $(s,L)\in (0,\infty)^2$ be given.
        Let $\gamma = (m,h)\in \Gamma$ be such that
        \begin{equation}\label{eq:def-gamma-minimax-usual}
            m = \llfloor s \rrfloor
            \quad \text{and} \quad
            h = \left[\frac{1}{2s \mathfrak{L}_{\llfloor s \rrfloor,s,L}^2} \times \frac{d \, \mathfrak{F}_{s,L}}{2^d \, n}\right]^{1/(2s+d)},
        \end{equation}
        where $\mathfrak{L}_{m,s,L}$ is a positive real constant that appears in the bias upper bound in Proposition~\ref{prop:control.bias}, and $\mathfrak{F}_{s,L}$ is any upper bound on $\sup_{f\in \Sigma(s,L)} \finfty$ (the existence of which is guaranteed by Lemma~\ref{lem:infty.norm.f}).
        If we set
        \begin{equation}\label{eq:N.n}
            N_n(s,L) =  n^{-s/(2s+d)},
        \end{equation}
        then the following assertions hold:
        \begin{enumerate}
            \item $\{N_n(s,L) : n\in \N\}$ is the minimax rate of convergence over $\Sigma(s,L)$.
            \item The estimator $\hat{f}_{\gamma}(t)$ attains this rate of convergence. More precisely, for $n$ large enough, we have
            \begin{equation*}
                N_n^{-1}(s,L)\sup_{f\in \Sigma(s,L)}
                \pwrisk(\hat{f}_{\gamma}, f)
                \leq
                C(s,L),
            \end{equation*}
            where the right-hand is defined by the positive real constant
            \begin{equation}\label{eq:C.s.L}
                C(s,L) = \frac{2^d \sqrt{D_m}}{\lambda_{\star}(m)} \sqrt{\left(\frac{d}{2s} + 1\right) \left(\frac{2s}{4^s d}\right)^{d/(2s+d)}} \mathfrak{F}_{s,L}^{s/(2s+d)} (\mathfrak{L}_{\llfloor s \rrfloor,s,L})^{d/(2s+d)}
            \end{equation}
            and $\lambda_{\star}(m) > 0$ is a lower bound found in Lemma~\ref{lem:usual.behavior} on the smallest eigenvalue $\lambda_{\gamma}$ of the Gram matrix $\mathcal{B}_{\gamma}$.
        \end{enumerate}
    \end{proposition}

    \begin{remark}\upshape
        Understanding the optimal parameters in \eqref{eq:def-gamma-minimax-usual} unfolds as follows. On simple domains, which are characterized by Assumption~\ref{ass:1}, it turns out that for a finite $m$, the ratios $W_h / \lambda_{\gamma}$ in Proposition~\ref{prop:control.bias}~and~\ref{prop:control.stochastic.term} remain uniformly bounded in the variable~$h$. Therefore, the bound on the bias \eqref{eq:bias.bound} depends on the smoothness parameter~$m$ asymptotically only through the factor $h^{\beta_m(s)}$, since $D_m$ and $\mathfrak{L}_{m,s,L}$ are constants for a finite $m$. By the same argument, the bound on the variance \eqref{eq:variance.bound} is not affected asymptotically by the choice of $m$. Therefore, given the decomposition of the squared risk measure as the sum between the squared bias and the variance of the estimator, recall Remark~\ref{rem:decomposition}, it becomes clear that the optimal choice of $m$ in this case will minimize $h^{\beta_m(s)} \equiv h^{\min(m+1,s)}$, which happens when $m\geq \llfloor s \rrfloor$. Since the rate $h^{\beta_m(s)}$ does not improve for larger $m\geq\llfloor s \rrfloor$ and $D_m$ increases with $m$, the optimal choice of $m$ for our pointwise $L^2$ risk measure is necessarily the smallest $m$ that satisfies $m\geq \llfloor s \rrfloor$, namely $m = \llfloor s \rrfloor$. Once the optimal $m = \llfloor s \rrfloor$ is selected, the optimization of the bandwidth $h$ is analogous to the approach employed for classical kernel density estimators.

        Heuristically, our local polynomial estimator wants to adapt its smoothness level $m$ to the smoothness level $s$ of the target density. The information of the target density is utilized optimally when $m$ is the largest integer which is smaller than $s$ and leaves room for a positive H\"older exponent in Definition~\ref{def:Sigma.s.L}. That value is necessarily $m = \llfloor s \rrfloor$. In particular, if the practitioner decides to fix a suboptimal value, say $m = 0$ when $s > 1$, and then optimize over $h$, the bias term in Proposition~\ref{prop:control.bias} will simply not have the smallest rate of convergence possible.

        On complicated domains, the optimization in $h$ is more subtle since the ratio $W_h / \lambda_{\gamma}$ has to be taken into account, which requires a case-by-case analysis; see, e.g., Proposition~\ref{prop:minimaxity.polynomial.sectors} and its proof.
    \end{remark}

    In Section~\ref{sec:selection-procedure}, we have constructed several statistical procedures depending on the choice of the discrete collection of polynomial degrees $\{m_{\ell} : \ell\in \N\}$. For the next result, we fix this collection by choosing
    \begin{equation}\label{eq:m.l}
        m_{\ell} = \left\lfloor\frac{\log n}{2\ell}\right\rfloor, \quad \ell\in \N.
    \end{equation}

    \begin{theorem}[Adaptivity of $\hat{f}(t)$ on simple domains]\label{thm:adaptivity.simple.domains}
        Suppose that Assumption~\ref{ass:1} holds.
        Let $\mathcal{K} = (0,\infty)^2$ and consider the set of H\"older-type functional classes $\{\Sigma(s,L) : (s,L)\in \mathcal{K}\}$.
        The following assertions hold:
        \begin{enumerate}
            \item The ARC is the admissible collection $\phi$ given by
                \begin{equation}\label{eq:ARC.simple.geometry}
                    \phi_n(s,L) = \left(\frac{\log n}{n}\right)^{s/(2s+d)},
                    \quad (s,L)\in \mathcal{K}, ~~n\in \N.
                \end{equation}
                More precisely, if $\psi$ is another admissible collection of normalizations such that $(s_0,L_0)$ belongs to $[\psi\ll \phi]$, then
                \begin{equation*}
                    [\psi\ll \phi]\subseteq \{s_0\} \times (0,\infty)
                    \quad \text{and} \quad
                    [\psi\ggg \phi]\supseteq  (s_0,\infty) \times (0,\infty).
                \end{equation*}
            \item The adaptive estimator $\hat{f}(t)$ is such that, for any $(s,L)\in \mathcal{K}$,
                \begin{equation*}
                    \begin{aligned}
                        &\limsup_{n\to\infty} \sup_{f\in \Sigma(s,L)} \phi_n^{-1}(s,L) \pwrisk(\hat{f}, f) \\[-2mm]
                        &~~\leq C \, \left[4 + \big\{(\delta-1) / \finfty\big\}^{1/2}\right] \, e^s \, \frac{2^d \sqrt{D^{\star}(s,d)}}{\lambda_{\bullet}(s,d)} \sqrt{d \, \delta} \, \, \mathfrak{F}_{s,L}^{~s/(2s+d)} (\mathfrak{L}_{s,L}^{\star})^{d/(2s+d)},
                    \end{aligned}
                \end{equation*}
                where $C > 0$ is a universal constant (independent of all parameters), $\delta > 1$ is a parameter that we fixed above~\eqref{eq:r.gamma.and.pen.gamma}, $\mathfrak{F}_{s,L}$ is any upper bound on $\sup_{f\in \Sigma(s,L)} \finfty$ (the existence of which is guaranteed by Lemma~\ref{lem:infty.norm.f}), and
                \vspace{-1mm}
                \begin{equation*}
                    \begin{aligned}
                        &D^{\star}(s,d) = \max_{0 \leq m' \leq \lfloor 2s+d \rfloor} D_{m'} \quad (\text{recall that } D_{m'} = \binom{m'+d}{d}) \\[-1mm]
                        &\lambda_{\bullet}(s,d) = \min_{0\leq m' \leq \lfloor 2s+d \rfloor} \lambda_{\star}(m) > 0, \\[-1mm]
                        &\mathfrak{L}_{s,L}^{\star} = \max_{0\leq m' \leq \lfloor 2s+d \rfloor} \mathfrak{L}_{m',s,L}.
                    \end{aligned}
                \end{equation*}
                Note that the positivity of $\lambda_{\bullet}(s,d)$ is guaranteed by Lemma~\ref{lem:usual.behavior}.
        \end{enumerate}
    \end{theorem}

\subsection{Domains with complicated geometries}\label{sec:complicated.domain}

    Set $\rho=1$ and define the following polynomial sectors in $\R^2$:
    \begin{equation}\label{eq:domain.k}
        \domain_k = \big\{(x,y)\in \R^2 : 0\leq x \leq 1, 0\leq y \leq x^k\big\}, \quad k > 1.
    \end{equation}
    This section is dedicated to the local estimation, at the origin $t=(0,0)$, of density functions supported on a given polynomial sector domain $\domain_k$.
    In particular, we obtain minimaxity and adaptivity results similar to those in Section~\ref{sec:simple.domain} but on the complicated domain $\domain_k$, see Proposition~\ref{prop:minimaxity.polynomial.sectors}~and~Theorem~\ref{thm:adaptivity.polynomial.sectors}, respectively.

    Item~2 of Remark~\ref{rem:other.complicated.domains} below explains how our results could be extended to other complicated domains without much difficulty. We believe that treating these polynomial sector examples separately elucidates the overall arguments and techniques and thus provides greater clarity for the reader. Furthermore, our statements here complement the simulations presented in Section~\ref{sec:simulations}, which are also performed on polynomial sectors.

    \begin{proposition}[Minimaxity of $\hat{f}_{\gamma}(t)$ on \texorpdfstring{$\domain_k$}{D\_k}]\label{prop:minimaxity.polynomial.sectors}
        Let $(s,L)\in (0,\infty)^2$ be given. Define
        \begin{equation}\label{eq:s.k.and.theta.k}
            s_k = \mathop{\mathrm{arg\,max}}_{\beta\in (0,s]} \theta_k(\beta),
            \quad \text{where} \quad
            \theta_k(\beta) = \frac{\beta-2\llfloor\beta\rrfloor(k-1)}{2\beta+k+1}.
        \end{equation}
        The following assertions hold:
        \begin{enumerate}
            \item There exists a positive real constant $C(s,L,k)$  such that, for any $f\in \Sigma(s,L)$, we have
                \begin{equation}\label{eq:rebroussement.upper.bound}
                    R_n(\hat{f}_{\gamma_k}, f)\leq C(s,L,k) n^{-\theta_k(s_k)},
                \end{equation}
                where $\gamma_k=(m,h)$ with $m=\llfloor s_k\rrfloor$ and $h=n^{-1/(2s_k+k+1)}$.
            \item Assume further that $s\in (0,1]$, then there exists a positive real constant $c(s,L,k)$ such that
                \begin{equation}\label{eq:rebroussement.lower.bound}
                 \inf_{\tilde{f}}\sup_{f\in \Sigma(s,L)} R_n(\tilde{f}, f)\geq c(s,L,k) n^{-\theta_k(s)}.
                \end{equation}
                In particular, $\{n^{-s/(2s+k+1)} : n\in \N\}$ is the minimax rate of convergence on $\Sigma(s,L)$ as soon as $s\in (0,1]$; refer to Item~1 of Remark~\ref{rem:other.complicated.domains} below.
        \end{enumerate}
    \end{proposition}

    \begin{remark}\label{rem:other.complicated.domains}\upshape
        Let us make two comments:
        \begin{enumerate}
            \item Through the upper bound~\eqref{eq:rebroussement.upper.bound}, Proposition~\ref{prop:minimaxity.polynomial.sectors} shows that the estimator is consistent as soon as $\theta_k(s_k)>0$. This is always the case since, for any $\beta\in (0,\min(s,1)]$, we have $\theta_k(\beta) = \beta/(2\beta+k+1) >0$. Using similar arguments, it is easily seen that for any $s\in (0,1]$ and any $k>1$, we have $s_k = s$. However, for $k>2$ and $s>1$, we always have $s-2\llfloor s \rrfloor(k-1) < 0$, which implies $s_k = 1 < s$. Moreover, note that, for any $s>0$ and any $k>1$, we have
                \begin{equation*}
                    \theta_k(s_k) < \frac{s}{2s+2}.
                \end{equation*}
                Thus, the rate of convergence established in Proposition~\ref{prop:minimaxity.polynomial.sectors} is smaller compared with the one established for simple domains in Proposition~\ref{prop:minimaxity.simple.domains}. The estimation problem considered in the present section is more difficult because of the local pinch of $\domain_k$ near the origin.
            \item Results analogous to those of Proposition~\ref{prop:minimaxity.polynomial.sectors} can be obtained for other domains $\domain$ with complicated geometries.
                In particular, for $\alpha,\beta\in \mathbb{N}_0^d$ and $h>0$, define
                \begin{equation*}
                    I_{\domain}(\alpha,\beta,h)=\int_{\R^d} \varphi_{\alpha}\left(\frac{u}{h}\right)\varphi_{\beta}\left(\frac{u}{h}\right)w_h(u) \rd u.
                \end{equation*}
                Assume that there exist a function $a:\mathbb{N}_0^d\to (0,\infty)$ and an increasing function $G:\mathbb{N}_0^d\to (0,\infty)$ such that
                \begin{equation*}
                    I_{\domain}(\alpha,\beta,h)=a(\alpha+\beta)h^{G(\alpha+\beta)}, \quad \text{for any $\alpha,\beta\in \mathbb{N}_0^d$.}
                \end{equation*}
                Now define
                \begin{equation*}
                    s_{\domain} = \mathop{\mathrm{arg\,max}}_{\beta\in (0,s]} \theta_{\domain}(\beta),
                    \quad \text{where} \quad
                    \theta_{\domain}(\beta) = \frac{\beta+G(0)-G(2\llfloor\beta\rrfloor)}{2\beta+d+G(0)}.
                \end{equation*}
                In this particular setting, it is possible to derive a result akin to~\eqref{eq:rebroussement.upper.bound}, with the rate of convergence $n^{-\theta(s_{\domain})}$ and the bandwidth $h=n^{-1/\{2s_{\domain}+d+G(0)\}}$, by following the proof of Proposition~\ref{prop:minimaxity.polynomial.sectors}.
        \end{enumerate}
    \end{remark}

    Now consider the adaptive estimator $\hat{f}(t)$ constructed in Section~\ref{sec:selection-procedure} and choose $m_{\ell}=0$ for all $\ell\in \N$. The following theorem ensures that it attains the ARC over the collection of H\"older-type functional classes $\{\Sigma(s,L) : (s,L)\in (0,1]\times(0,\infty)\}$. The ARC here corresponds to the minimax rate of convergence stated just below \eqref{eq:rebroussement.lower.bound} up to a logarithmic factor.

    \begin{theorem}[Adaptivity of $\hat{f}(t)$ on \texorpdfstring{$\domain_k$}{D\_k}]\label{thm:adaptivity.polynomial.sectors}
        The following assertions hold:
        \begin{enumerate}
            \item The ARC is the admissible collection $\phi$ given by
                \begin{equation*}
                    \phi_n(s,L) = \left(\frac{\log n}{n}\right)^{s/(2s+k+1)},
                    \quad (s,L)\in (0,1]\times(0,\infty), ~~n\in \N.
                \end{equation*}
            \item
                Let $(s,L)\in (0,1] \times (0,\infty)$ be given.
                There exists a positive real constant $\tilde{C}(s,L,k,\delta) > 0$ such that
                \begin{equation} \label{eq:rate.rebrou.adapt}
                    \limsup_{n\to\infty} \sup_{f\in \Sigma(s,L)} \left(\frac{\log n}{n}\right)^{-s/(2s+k+1)} R(\hat{f}, f)
                    \leq \tilde{C}(s,L,k,\delta).
                \end{equation}
        \end{enumerate}
    \end{theorem}

    \begin{remark}\upshape
        Theorem~\ref{thm:adaptivity.polynomial.sectors} exclusively studies the case of a smoothness parameter $s$ less than or equal to $1$.
        We conjecture that a result similar to~\eqref{eq:rate.rebrou.adapt}, with rate of convergence $(n^{-1} \log n)^{\theta_k(s_k)}$, can be obtained for any $s>1$ by choosing
        \begin{equation*}
            m_\ell = \left\lfloor \frac{\log n}{2\ell}-\frac{k+1}{2}\right\rfloor.
        \end{equation*}
        Nevertheless, this rate of convergence may be not the ARC for $s > 1$.
    \end{remark}

\subsection{Extension to unknown domains}\label{sec:unknown.domains}

    Although the general problem of constructing optimal estimators in this setting is beyond the scope of this paper, we propose to consider a simple but common situation: assume that the support $\domain$ is an unknown compact convex set with finite Lebesgue measure $\lambda(\domain)$, and assume further that the target density $f$ is bounded, i.e., its supremum norm over $\R^d$, $\|f\|_{\infty}$, is finite. In order to control the error related to the estimation of the support, we have to consider an integrated risk (over $\R^d$) rather than a pointwise risk. That is, for any estimator $\tilde{f}$ of $f$, we are interested in bounding the $L^1$ integrated risk
    \begin{equation*}
        {IR}_n(\tilde{f}, f) = \EE\left\{\int_{\R^d} |\tilde{f}(x) - f(x)| \rd x\right\}.
    \end{equation*}

    We propose to estimate $f$ by splitting the data into two disjoint subsamples: we estimate the support of $f$ using the first subsample and then use our adaptive estimation procedure on the second subsample for each point in the estimated domain. More precisely, for some proportion parameter $\alpha\in (0,1)$ satisfying $\alpha n\in \N$, define the two subsamples
    \begin{equation*}
        \smash{\mathbb{X}_n^{(1)}} = \smash{(X_1,\ldots,X_{(1 - \alpha) n})}, \quad \smash{\mathbb{X}_n^{(2)}} = \smash{(X_{(1 - \alpha) n + 1},\ldots,X_n)},
    \end{equation*}
    so that $\smash{\mathbb{X}_n^{(1)}} \!\cap \smash{\mathbb{X}_n^{(2)}} = \emptyset$ and $\smash{\mathbb{X}_n^{(1)}} \!\cup \smash{\mathbb{X}_n^{(2)}} = \mathbb{X}_n$ in particular.

    Define $\hat{\domain}_n$ as the convex hull of the first subsample $\mathbb{X}_n^{(1)}$. This serves as an estimator for the convex support $\domain$ of the target density $f$, a subject that has been extensively explored in the literature. For an in-depth review, we direct the reader to \cite{Brunel2018}. Now, we remark that, conditionally on the observations in $\smash{\mathbb{X}_n^{(1)}}$, the observations in $\smash{\mathbb{X}_n^{(2)}} \!\cap \hat{\domain}_n$ are independent and identically distributed with a common density function given by
    \begin{equation*}
        g(x) = \frac{f(x) \1_{\hat{\domain}_n}(x)}{\int_{\hat{\domain}_n} f(y) \rd y}, \quad x\in \R^d,
    \end{equation*}
    where $f\1_{\hat{\domain}_n}\leq f$ (since $\hat{\domain}_n\subseteq \domain$).

    By definition, note that $\smash{\mathbb{X}_n^{(2)}}$ contains $\alpha n$ observations, i.e., $\smash{\mathrm{card}(\mathbb{X}_n^{(2)})} = \alpha n$, and let
    \begin{equation*}
        \hat{p}_n = \frac{\mathrm{card}(\mathbb{X}_n^{(2)} \!\cap \hat{\domain}_n)}{\mathrm{card}(\mathbb{X}_n^{(2)})}
    \end{equation*}
    denote the proportion of those observations which fall into the estimated support $\hat{\domain}_n$; $\hat{p}_n$ is the empirical counterpart of $p_n \equiv \smash{\int_{\hat{\domain}_n} f(y) \rd y}$. Moreover, conditionally on $\smash{\mathbb{X}_n^{(1)}}$, the support $\hat{\domain}_n$ of $g$ is known, so our local polynomial estimation procedure can be used to obtain the adaptive estimator $\hat{g}(t)$ defined in~\eqref{eq:selection.rule} for each point $t\in \hat{\domain}_n$. Given the above relationship between $g$ and $f$, it is natural to define the estimator of $f$:
    \begin{equation*}
        \hat{f}(x) = \hat{p}_n \hat{g}(x), \quad x\in \R^d.
    \end{equation*}
    Using the triangle inequality, Fubini's theorem, and Jensen's inequality, the integrated risk of $\hat{f}$ can be bounded as follows:
    \begin{align*}
        {IR}_n(\hat{f}, f)
        &= \EE\left\{\int_{\domain} \big|\hat{p}_n \hat{g}(x) - \hat{p}_n g(x) + \hat{p}_n g(x) - p_n g(x) + p_n g(x) - f(x)\big| \rd x\right\} \\
        &\leq \EE\left[\int_{\hat{\domain}_n} \EE\big\{\left|\hat{g}(x) - g(x)\right| \big| \, \mathbb{X}_n^{(1)}\big\} \rd x\right] + \EE\left(|\hat{p}_n - p_n|\right) + \EE\left\{\int_{\domain\backslash \hat{\domain}_n} f(x) \rd x\right\} \\
        &\leq \EE\left(\int_{\hat{\domain}_n} \left[\EE\big\{\left|\hat{g}(x) - g(x)\right|^2 \big| \, \mathbb{X}_n^{(1)}\big\}\right]^{1/2} \rd x\right) + \EE\left(|\hat{p}_n - p_n|\right) + \EE\left\{\int_{\domain\backslash \hat{\domain}_n} f(x) \rd x\right\}.
    \end{align*}
    To bound $\EE(|\hat{p}_n - p_n|)$, we use the fact that, conditionally to $\smash{\mathbb{X}_n^{(1)}}$, the quantity $\alpha n \, \hat{p}_n$ has a Binomial distribution with parameters $\alpha n$ and $p_n$. To bound $\smash{\EE\{\int_{\domain\backslash \hat{\domain}_n} f(x) \rd x\}}$, we use Theorem~7 of \cite{Brunel2018} with the choice $x = \log\{(1 - \alpha) n\} \geq 0$. As a result, for $n$ large enough, we have
    \begin{equation*}
        \begin{aligned}
            {IR}_n(\hat{f}, f)
            &\lesssim \lambda(\domain) \times \EE\left(\sup_{x\in \hat{\domain}_n} \left[\EE\big\{\left|\hat{g}(x) - g(x)\right|^2 \big| \, \mathbb{X}_n^{(1)}\big\}\right]^{1/2} \1_{\{\hat{p}_n \geq 1 / 2\}}\right) \\
            &\quad+ \lambda(\domain) \times \PP(\hat{p}_n < 1 / 2) + \frac{1}{\sqrt{\alpha n}} + \left[\frac{\|f\|_{\infty}+1}{\{(1 - \alpha) n\}^{2/(d+1)}} + \frac{\log\{(1 - \alpha) n\}}{(1 - \alpha) n} + \frac{1}{(1 - \alpha) n}\right].
        \end{aligned}
    \end{equation*}

    By Theorem~7 of \cite{Brunel2018}, the probability $\PP(\hat{p}_n < 1 / 2)$ is exponentially small in $n$. It remains to bound the conditional expectation in the last equation. Note that if $n$ is large enough and the target density $f$ belongs to $\Sigma(s,L)$, then $g$ belongs to $\Sigma(s,L/p_n)\subseteq \Sigma(s,2L)$ with probability larger than $1/\{(1 - \alpha) n\}$, again by Theorem~7 of \cite{Brunel2018}. Hence, conditionally on $\mathbb{X}_n^{(1)}$, by applying item~2 of Theorem~\ref{thm:adaptivity.simple.domains} for every $x = t\in \hat{\domain}_n$ in $[\EE\{|\hat{g}(x) - g(x)|^2 | \, \smash{\mathbb{X}_n^{\scriptscriptstyle(1)}}\}]^{1/2} \1_{\{\hat{p}_n \geq 1 / 2\}}$, we obtain
    \begin{equation*}
        {IR}_n(\hat{f}, f) \lesssim (n^{-1} \log n)^{s/(2s+d)} + n^{-2/(d+1)},
    \end{equation*}
    where $\lesssim$ depends on $d$, $\|f\|_{\infty}$, $s$ $L$ and $\alpha$.

    Assuming further that the target density $f$ is bounded away from zero on its support $\domain$, i.e., $f_0 \equiv \inf_{x\in \domain} f(x) > 0$, we also get a lower bound on the integrated risk of any estimator $\tilde{f}$ of $f$, including $\hat{f}$. Indeed, letting $\widetilde{\domain}_n = \{x\in \R^d : \tilde{f}(x) \geq f_0 / 2\}$, we have
    \begin{equation*}
        {IR}_n(\tilde{f}, f) = \EE\left\{\int_{\R^d} |\tilde{f}(x) - f(x)| \rd x\right\} \geq \frac{f_0}{2} \, \EE\big\{\lambda(\domain \triangle \widetilde{\domain}_n)\big\} \gtrsim n^{-2/(d+1)},
    \end{equation*}
    where $\domain \triangle \widetilde{\domain}_n$ denotes the symmetric difference between $\domain$ and $\widetilde{\domain}_n$, and the last inequality follows from a direct application of Theorem~2 of \cite{Brunel2018}. If the domain $\domain$ was known, we would also have ${IR}_n(\tilde{f}, f) \gtrsim n^{-s/(2s+d)}$.

    Therefore, combining the findings from the last two paragraphs, it is reasonable to conjecture that in the case of unknown compact convex domains, both the minimax rate and the ARC for the $L^1$ integrated risk measure should be $n^{-s/(2s+d)} + n^{-2/(d+1)}$ when assuming some mild regularity conditions on the target density $f$. This problem is left open for future research.

\section{Simulations}\label{sec:simulations}

    In this section, we conduct a concise investigation using simulated data, where we assess the performance of our estimator against the one provided in the \texttt{sparr} package \cite{Davies_et_al2024}; see \citep{DaviesMarshallHazelton2018} for a tutorial. The \texttt{sparr} package is extensively used within the \texttt{R} community for estimating bivariate density functions on complicated domains. It is worth noting that since 2020, based on the data gathered by David Robinson's Shiny app \cite{Robinson2015}, the download rate of the \texttt{sparr} package has consistently been approximately three times higher than that of the other commonly used \texttt{latticeDensity} package \citep{Barry2021}.

    We consider two polynomial sector domains for our simulations, namely $\domain_1$ and $\domain_{2.1}$; recall~\eqref{eq:domain.k}. Notice that the linear sector $\domain_1$ is a simple domain which satisfies Assumption~\ref{ass:1}, while $\domain_{2.1}$ is a complicated domain which is not star-shaped around the point of estimation $t=(0,0)$. On each domain, we consider two different types of density functions, which we describe below.

    The first two density functions, polynomial in nature, are defined by
    \begin{equation}\label{eq:poly.density}
        f_k(x,y) = C_k \big\{(x-0.6)^2 + (y - 0.2)^2\big\} \1_{\domain_k}(x,y), \quad k\in\{1, 2.1\},
    \end{equation}
    where $C_k$ is a positive normalizing constant which depends on $k$. The contour plots of these two densities are depicted in Figure~\ref{fig:dens1}.

    \begin{figure}[!b]
        \includegraphics[width=67mm]{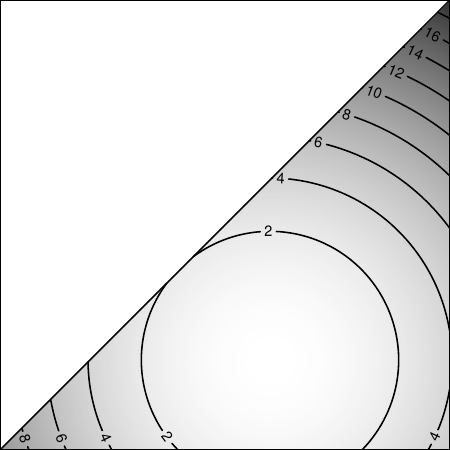}
        \qquad
        \includegraphics[width=67mm]{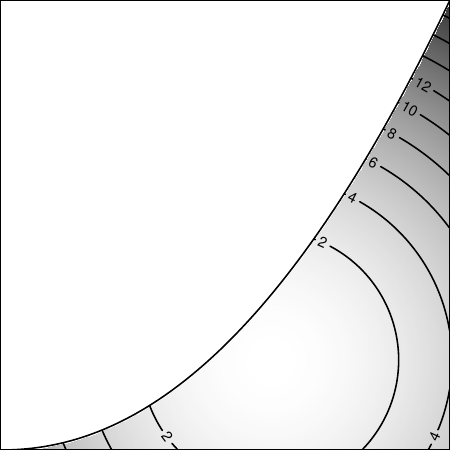}
        \caption{The polynomial density functions $f_1$ (left) and $f_{2.1}$ (right) defined in~\eqref{eq:poly.density}. Darker regions are associated with higher values of $f_k$.}
        \label{fig:dens1}
    \end{figure}

    The next two density functions are mixtures of truncated Gaussian distributions defined by
    \begin{equation}\label{eq:mixture.Gaussian.density}
        \begin{aligned}
            g_k(x,y)
            &= A_k \left[\exp\left\{-\frac{(x-a_k)^2 + (y-b_k)^2}{2(0.4)^2}\right\} + \exp\left\{-\frac{(x-c_k)^2 + (y-d_k)^2}{2(0.15)^2}\right\}\right], \quad k\in\{1, 2.1\},
        \end{aligned}
    \end{equation}
    where $A_k$ is a positive normalizing constant which depends on $k$, and
    \begin{equation*}
        a_k = 1/10, \quad b_k = (1/10)^k/2, \quad c_k = 3/4, \quad \text{and} \quad d_k =  (3/4)^k/2.
    \end{equation*}
    The contour plots of these two densities are depicted in Figure~\ref{fig:dens2}.

    \begin{figure}[ht]
        \includegraphics[width=67mm]{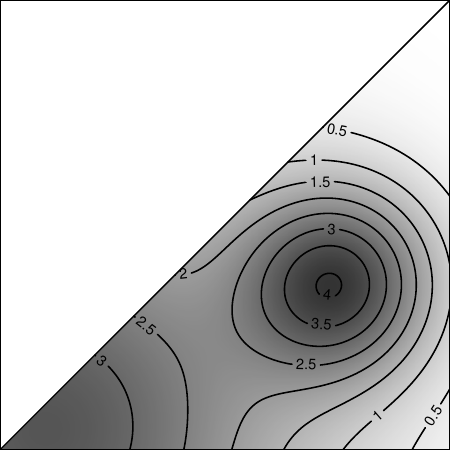}
        \qquad
        \includegraphics[width=67mm]{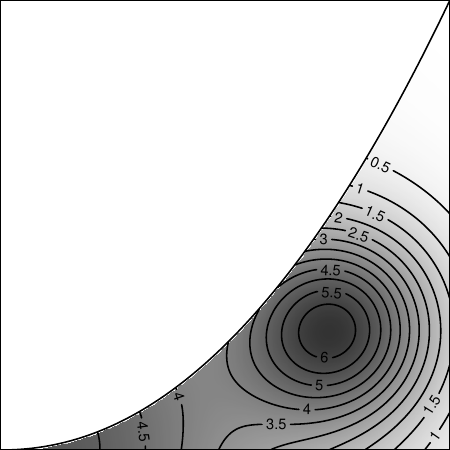}
        \caption{The mixtures of truncated Gaussian density functions $g_1$ (left) and $g_{2.1}$ (right) defined in~\eqref{eq:mixture.Gaussian.density}. Darker regions are associated with higher values of $g_k$.}
        \label{fig:dens2}
    \end{figure}

    The \texttt{sparr} estimator $\smash{\hat{f}_h^{\, \mathrm{SPARR}}}$ depends on a tuning parameter $h$ which plays the role of the bandwidth while our local polynomial estimator $\smash{\hat{f}_{m,h}^{\, \mathrm{LP}}}$ depends on both a bandwidth $h$ and a polynomial degree $m$. Below, we consider the family of polynomial degrees $\mathcal{M} = \{0, 1, 2, 3, 4, 5\}$ and the family of bandwidths $\mathcal{H} = \{0.01 + 0.001\times\ell : \ell = 0,\ldots, 599\}$. For each target density $f\in\{f_1, f_{2.1}, g_1, g_{2.1}\}$ and each sample size $n\in\{200, 500, 1000, 2000\}$, we aim to compare two oracle estimators defined by
    \begin{alignat*}{3}
        &\hat{f}_{m^*,h^*}^{\, \mathrm{LP}}
        && \quad \text{where} \quad
        (m^{*}, h^*)
        && = \mathop{\mathrm{arg\,min}}_{(m,h)\in \mathcal{M}\times \mathcal{H}} \EE\left\{\big|\hat{f}_{m,h}^{\, \mathrm{LP}}(t) - f(t)\big|^2\right\}, \\[-1mm]
        &\hat{f}_{h^{\circ}}^{\, \mathrm{SPARR}}
        && \quad \text{where} \quad
        ~~~~~~~~~~ h^{\circ}
        && = \mathop{\mathrm{arg\,min}}_{h\in \mathcal{H}} \EE\left\{\big|\hat{f}_h^{\, \mathrm{SPARR}}(t)\ - f(t)\big|^2\right\}.
    \end{alignat*}
    To accomplish this, we generate $R = 5000$ replications of the random samples of size $n$, and we compute for each replication $r\in \{1,\ldots,R\}$ the corresponding values of the estimators denoted by $\smash{[\hat{f}_{m,h}^{\, \mathrm{LP}}(t)]_r}$ and $\smash{[\hat{f}_h^{\, \mathrm{SPARR}}(t)]_r}$, respectively. Finally, we obtain the estimated oracles using
    \begin{equation*}
        \begin{aligned}
            (\hat{m}^*, \hat{h}^*)
            &= \mathop{\mathrm{arg\,min}}_{(m,h)\in \mathcal{M}\times \mathcal{H}} \frac{1}{R} \sum_{r=1}^R \big|[\hat{f}_{m,h}^{\, \mathrm{LP}}(t)]_r - f(t)\big|^2, \\[-1mm]
            \hat{h}^{\circ}
            &= \mathop{\mathrm{arg\,min}}_{h\in \mathcal{H}} \frac{1}{R} \sum_{r=1}^R \big|[\hat{f}_h^{\, \mathrm{SPARR}}(t)]_r - f(t)\big|^2.
        \end{aligned}
    \end{equation*}
    The results are presented in Figures~\ref{fig:boxplot.poly}~and~\ref{fig:boxplot.norm} below.

    \begin{figure}[!ht]
        \includegraphics[width=70mm]{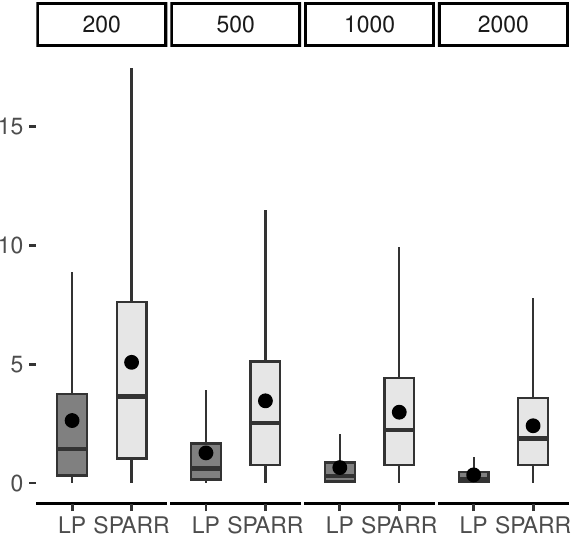}
        \includegraphics[width=70mm]{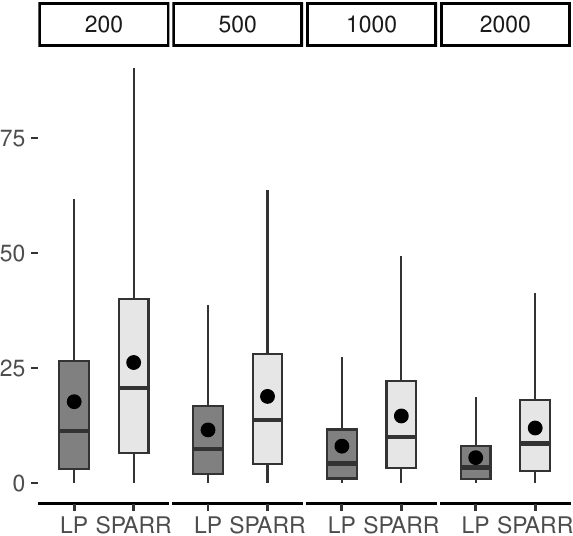}
        \caption{Boxplots of the $R$ replications of the estimated oracles. The target densities are $f_1$ (left) and $f_{2.1}$ (right).\vspace{-3mm}}
        \label{fig:boxplot.poly}
    \end{figure}

    \begin{figure}[!ht]
        \includegraphics[width=70mm]{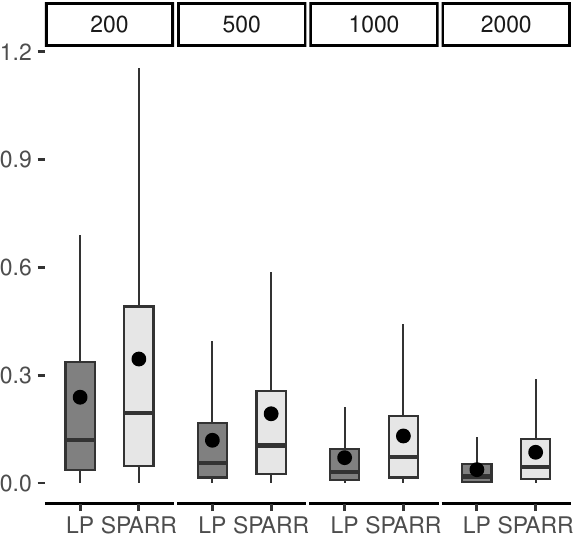}
        \includegraphics[width=70mm]{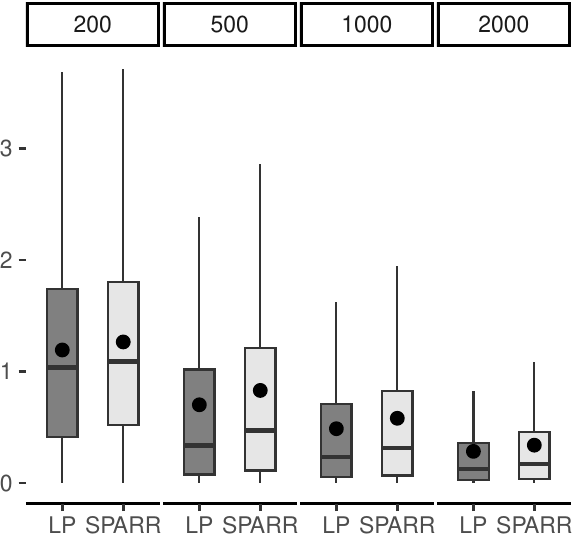}
        \caption{Boxplots of the $R$ replications of the estimated oracles. The target densities are $g_1$ (left) and $g_{2.1}$ (right).\vspace{-3mm}}
        \label{fig:boxplot.norm}
    \end{figure}

    Given a random sample of size $n\in \{200, 500, 1000, 2000\}$ for each replication, the boxplots of the $R$ estimated oracles using our local polynomial method, $\raisebox{0.2ex}{$\smash{[\hat{f}_{\hat{m}^*, \hat{h}^*}^{\, \mathrm{LP}}(t)]_{r=1}^R}$}$, are displayed in the left columns (LP) and the boxplots of the $R$ estimated oracles using the \texttt{sparr} package, $\smash{[\hat{f}_{\hat{h}^{\circ}}^{\, \mathrm{SPARR}}(t)]_{r=1}^R}$, are displayed in the right columns (SPARR). Figure~\ref{fig:boxplot.poly} corresponds with the polynomial-type densities $f_k$ defined in~\eqref{eq:poly.density} for $k=1$ (left) and $k=2.1$ (right), while Figure~\ref{fig:boxplot.norm} corresponds with the mixtures of truncated Gaussian distributions $g_k$ defined in \eqref{eq:mixture.Gaussian.density} for $k=1$ (left) and $k=2.1$ (right).

    Together, Figures~\ref{fig:boxplot.poly}~and~\ref{fig:boxplot.norm} reveal that, for both the simple domain $\domain_1$ and the complicated domain $\domain_{2.1}$, our oracle estimates almost always outperform those of the \texttt{sparr} package in mean (black dot inside each box), in median (horizontal line inside each box), and in interquartile range (height of each box), under both the polynomial densities $f_k$ and the mixtures of truncated Gaussian densities $g_k$, and across all sample sizes $n$.

    The only exception is observed under $g_{2.1}$ for $n = 200$, where the means, medians and interquartile ranges are close to being equal, still giving a slight edge to our method based on the means and quartiles.

    Overall, out of the 16 cases considered, our oracle estimates outperform those of the \texttt{sparr} package in every case. The best polynomial degree selected was always either $1$ or $2$, which justifies our restricted range of polynomial degrees $\mathcal{M} = \{0, 1, 2, 3, 4, 5\}$. It is typical of low degree polynomials to be optimal for local polynomial kernel estimators; see, e.g., p.~126 of \cite{MR1319818}.
    
    \setcounter{section}{5}
\section{Preliminary technical lemmas}\label{sec:technical.lemmas}

    \begin{lemma}\label{lem:positive.definite}
        Let $\gamma = (m,h)\in \Gamma$ be given.
        The Gram matrix $\mathcal{B}_{\gamma}$ defined in~\eqref{eq:matrix.B.gamma} is symmetric positive definite.
        In particular, its smallest eigenvalue $\lambda_{\gamma}$, defined in~\eqref{eq:lambda.gamma}, is positive.
    \end{lemma}

    \begin{proof}[Proof of Lemma~\ref{lem:positive.definite}]
        For all $v\in \R^{D_m}$, the fact that $w_h\geq 0$ implies
        \begin{equation}\label{eq:lem:positive.definite}
            \begin{aligned}
                v^{\top} \mathcal{B}_{\gamma}  v
                &=\int_{\voisin(h)} v^{\top} \Phi_{\gamma}(u) \Phi_{\gamma}^{\top}(u) v \, w_h(u) \rd u
                =\int_{\voisin(h)} \left\{v^{\top} \Phi_{\gamma}(u)\right\}^2 w_h(u) \rd u \geq 0.
            \end{aligned}
        \end{equation}
        Moreover, if some $v\in \R^{D_m}$ satisfies $v^{\top} \mathcal{B}_{\gamma} v = 0$, then $v^{\top} \Phi_{\gamma}(u) = 0$ for almost all $u\in \voisin(h)$. This means that the polynomial $v^{\top} \Phi_{\gamma}(u)$, with coefficients $v_i$, is zero for almost all $u\in \voisin(h) = \operatorname{Supp}(w_h)$. Since $\mathrm{Leb}\left\{\voisin(h)\right\} > 0$, we conclude that $v_i = 0$ for all $i\in \{1,\ldots,D_m\}$.
    \end{proof}

    \begin{lemma}\label{lem:usual.behavior}
        Suppose that Assumption~\ref*{ass:1} holds.
        First, we have
        \begin{equation*}
            \mathrm{Leb}(\Delta_0) \leq W_h \leq \mathrm{Leb}(\Delta) = 2^d, \quad \text{for all } h\in (0,\rho].
        \end{equation*}
        Second, for all $m\in \N_0$, there exists $\lambda_{\star}(m)> 0$, which also depends on $\rho$, such that
        \begin{equation*}
           \inf_{h\in (0,\rho]} \lambda_{\gamma} \geq \lambda_{\star}(m), \quad \text{where } \gamma = (m,h).
        \end{equation*}
    \end{lemma}

    \begin{proof}[Proof of Lemma~\ref{lem:usual.behavior}]
        Since $h \Delta_0\subseteq \voisin(h)\subseteq h \Delta$ for all $h\in (0,\rho]$ by Assumption~\ref*{ass:1}, we have
        \begin{equation*}
           \int_{h\Delta_0} w_h(u) \rd u \leq\int_{\voisin(h)} w_h(u) \rd u \leq\int_{h\Delta} w_h(u) \rd u.
        \end{equation*}
        Furthermore, $\operatorname{Supp}(w_h) = \voisin(h)$, so the middle integral is $W_h$ as defined in~\eqref{eq:L.n.W.h}.
        With a simple change of variable and the fact that $K$ is identically equal to $1$ on $\Delta$ according to~\eqref{eq:K}, we get
        \begin{equation*}
            \mathrm{Leb}(\Delta_0) \leq W_h \leq \mathrm{Leb}(\Delta) = 2^d, \quad \text{for all } h\in (0,\rho],
        \end{equation*}
        which proves the first claim of the lemma.

        Next, using the fact that $\voisin(h)\supseteq h \Delta_0$ for all $h\in (0,\rho]$ by Assumption~\ref*{ass:1}, and applying the linear change of variable $\tilde{u} = (\rho / h) u$, we have, for any $v\in \mathbb{R}^{D_m}$,
        \begin{equation*}
            \begin{aligned}
                v^{\top} \mathcal{B}_{\gamma} v
                &= \int_{\voisin(h)} \left\{v^{\top} \Phi_{\gamma}(u)\right\}^2 w_h(u) \rd u
                \geq h^{-d}\int_{h\Delta_0} \left\{v^{\top} \Phi_{\gamma}(u)\right\}^2 \rd u \\
                &= \rho^{-d}\int_{\rho\Delta_0} \left\{v^{\top} \Phi_{(m,\rho)}(\tilde{u})\right\}^2 \rd \tilde{u}
                = \rho^{-d} v^{\top} \left\{\int_{\rho\Delta_0} \Phi_{(m,\rho)}(\tilde{u}) \Phi_{(m,\rho)}^{\top}(\tilde{u}) \rd \tilde{u}\right\} \, v.
            \end{aligned}
        \end{equation*}
        The same line of reasoning as in the proof of Lemma~\ref{lem:positive.definite} now shows that the integral on the right-hand side is a symmetric positive definite matrix. If $\mu(m,\rho) > 0$ denotes its smallest eigenvalue, then the last equation yields
        \begin{equation*}
            \lambda_{\gamma} \geq \rho^{-d}\mu(m,\rho) \eqqcolon \lambda_{\star}(m),
        \end{equation*}
        where the dependence on $\rho$ is omitted.
        This proves the second claim of the lemma.
    \end{proof}

    \begin{lemma}\label{lem:infty.norm.f}
        Let $(s,L)\in (0,\infty)^2$ be given. There exists a positive real constant $\mathfrak{F}_{s,L}$ (that also depends possibly on $\domain$, $t$ and $\rho$) such that
        \begin{equation*}
            \sup_{f\in \Sigma(s,L)} \finfty \leq \mathfrak{F}_{s,L},
        \end{equation*}
        where recall from~\eqref{eq:f.bounded} that $\finfty = \sup_{u\in \voisin(\rho)} \lvert f(t+u) \rvert$.
    \end{lemma}

    \begin{proof}[Proof of Lemma~\ref{lem:infty.norm.f}]
        We adapt the proof found on p.~7 of \cite{Tsybakov2004frenchsupp} to our framework.
        Let $f\in \Sigma(s,L)$ and let $\gamma_{\rho} = (\llfloor s \rrfloor, \rho)\in \Gamma$.
        By definition of $\Sigma(s,L)$, there exists a polynomial $q\in \mathcal{P}_{\llfloor s \rrfloor}$ such that, for any $u\in \voisin(\rho)$,
        \begin{equation}\label{eq:lem:infty-norm-f.first}
            \left|f(t+u) - q(u)\right| \leq L\|u\|_{\infty}^s.
        \end{equation}
        Let us write
        \begin{equation*}
            q(u) = Q_{\llfloor s \rrfloor}^{\top} \Phi_{\gamma_{\rho}}(u),
        \end{equation*}
        where $\smash{Q_{\llfloor s \rrfloor}}$ denotes the $\smash{D_{\llfloor s \rrfloor} \times 1}$ coordinate vector of $q$ in the basis $\smash{\Phi_{\gamma_{\rho}}}$.
        Then, using the fact that $|\varphi_{\alpha}(u/\rho)| \leq 1$ for all $u\in \voisin(\rho)$ and all $\alpha\in \N_0^d$, we have
        \begin{equation}\label{eq:proof1}
            |q(u)| \leq \max_{|\alpha| \leq \llfloor s \rrfloor} |\varphi_{\alpha}(u/\rho)| \times \|Q_{\llfloor s \rrfloor}\|_1 \leq \sqrt{D_{\llfloor s \rrfloor}} \, \|Q_{\llfloor s \rrfloor}\|_2.
        \end{equation}
        Since $\mathcal{P}_{\llfloor s \rrfloor}$ is a vector space of finite dimension, the $L^1(w_{\rho})$ and $L^2(w_{\rho})$ norms are equivalent on that space. Thus, there exists some constant $c_{s,\rho}>0$ such that
        \begin{equation}\label{eq:proof2}
            \begin{aligned}
                \int_{\voisin(\rho)} |q(u)| w_{\rho}(u) \rd u
                &\geq c_{s,\rho} \sqrt{\int_{\voisin(\rho)} q^2(u) w_{\rho}(u) \rd u} \\
                &= c_{s,\rho} \sqrt{Q_{\llfloor s \rrfloor}^{\top} \mathcal{B}_{\gamma_{\rho}}  Q_{\llfloor s \rrfloor}} \geq c_{s,\rho} \sqrt{\lambda_{\gamma_{\rho}}} \, \|Q_{\llfloor s \rrfloor}\|_2,
            \end{aligned}
        \end{equation}
        where the equality and the last inequality are a consequence of~\eqref{eq:lem:positive.definite} and~\eqref{eq:lambda.gamma}, respectively.
        Combining~\eqref{eq:proof1} and~\eqref{eq:proof2}, followed by an application of the bound in~\eqref{eq:lem:infty-norm-f.first}, we obtain
        \begin{equation*}
            \begin{aligned}
                |q(u)|
                &\leq \frac{\sqrt{D_{\llfloor s \rrfloor}}}{c_{s,\rho} \sqrt{\lambda_{\gamma_{\rho}}}} \int_{\voisin(\rho)} |q(u)| w_{\rho}(u) \rd u \\
                &\leq \frac{\sqrt{D_{\llfloor s \rrfloor}}}{c_{s,\rho} \sqrt{\lambda_{\gamma_{\rho}}}} \left\{\int_{\voisin(\rho)} |f(t + u)| w_{\rho}(u) \rd u +\int_{\voisin(\rho)} L \|u\|_{\infty}^s w_{\rho}(u) \rd u\right\}.
            \end{aligned}
        \end{equation*}
        It remains to observe that $0 \leq w_{\rho} \leq \rho^{-d}$ and $f$ is a density function to deduce
        \begin{equation}\label{eq:proof10}
            \finfty \leq \sup_{u\in \voisin(\rho)}|q(u)| + L \rho^s \leq \frac{\sqrt{D_{\llfloor s \rrfloor}}}{c_{s,\rho} \sqrt{\lambda_{\gamma_{\rho}}}} \left(1 + L \rho^s W_{\!\rho}\right) + L \rho^s.
        \end{equation}
        The conclusion follows since $\lambda_{\gamma_{\rho}} > 0$ by Lemma~\ref{lem:positive.definite}.
    \end{proof}

    \begin{lemma}[Concentration of $\hat{f}_{\gamma}(t)$ around its mean]\label{lem:Bernstein.1}
        Let $\gamma = (m,h)\in \Gamma$ be given.
        Recall the definitions of $c_{\gamma}$ and $v_{\gamma}$ from~\eqref{eq:c.gamma} and~\eqref{eq:v.gamma}.
        Assuming that $\finfty < \infty$, we have, for all $x > 0$,
        \begin{equation*}
            \PP\left[\left|\hat{f}_{\gamma}(t) - \EE\left\{\hat{f}_{\gamma}(t)\right\}\right| > r_{\gamma}(v_{\gamma}, x)\right] \leq 2\exp(-x),
        \end{equation*}
        where $r_{\gamma}(v_{\gamma}, x) = \sqrt{2 v_{\gamma} x} + c_{\gamma} x$.
    \end{lemma}

    \begin{proof}[Proof of Lemma~\ref{lem:Bernstein.1}]
        By the definition of $\hat{f}_{\gamma}(t)$ in~\eqref{eq:hat.f.gamma}, observe that
        \begin{equation}\label{eq:hat.f.Y.i.decomp}
            \hat{f}_{\gamma}(t) - \EE\left\{\hat{f}_{\gamma}(t)\right\} = \sum_{i=1}^n \left[Y_i(\gamma) - \EE\{Y_i(\gamma)\}\right],
        \end{equation}
        where
        \begin{equation*}
            Y_i(\gamma) = \frac{1}{n} \left\{H_{\gamma}^{\top}(0) H_{\gamma}(X_i-t) w_h(X_i-t)\right\}.
        \end{equation*}
        This readily implies that $\sum_{i=1}^n \EE\left\{Y_i^2(\gamma)\right\} = v_{\gamma}$. Furthermore, by H\"older's inequality, the submultiplicativity property of the spectral norm and the fact that $\|\Phi_{\gamma}(0)\Phi_{\gamma}^{\top}(0)\|_2 = 1$, note that, for all $u\in \voisin(h)$,
        \begin{equation}\label{eq:H.H.2}
            \begin{aligned}
                \big\{H_{\gamma}^{\top}(0) H_{\gamma}(u)\big\}^2
                &= \big\{\Phi_{\gamma}^{\top}(0)\mathcal{B}_{\gamma}^{-1} \Phi_{\gamma}(u)\big\}^2 \\
                &=\Phi_{\gamma}^{\top}(u)\mathcal{B}_{\gamma}^{-1} \Phi_{\gamma}(0)\Phi_{\gamma}^{\top}(0)\mathcal{B}_{\gamma}^{-1}\Phi_{\gamma}(u) \\
                &\leq \|\Phi_{\gamma}(u)\|_2\|\mathcal{B}_{\gamma}^{-1} \Phi_{\gamma}(0)\Phi_{\gamma}^{\top}(0)\mathcal{B}_{\gamma}^{-1}\|_2\|\Phi_{\gamma}(u)\|_2 \\
                &\leq \|\Phi_{\gamma}(u)\|_2\|\mathcal{B}_{\gamma}^{-1}\|_2 \|\Phi_{\gamma}(0)\Phi_{\gamma}^{\top}(0)\|_2\|\mathcal{B}_{\gamma}^{-1}\|_2\|\Phi_{\gamma}(u)\|_2 \\[-1mm]
                &= \frac{1}{\lambda_{\gamma}^2}\|\Phi_{\gamma}(u)\|_2^2
                \leq \frac{D_m}{\lambda_{\gamma}^2}.
            \end{aligned}
        \end{equation}
        By the bound~\eqref{eq:H.H.2}, and the fact that $\operatorname{Supp}(w_h) = \voisin(h)$ and $0 \leq w_h \leq h^{-d}$ because of~\eqref{eq:w.h} and~\eqref{eq:K}, we deduce
        \begin{equation}\label{eq:bound.Y.i}
            |Y_i(\gamma)| = \frac{w_h(X_i-t)}{n} \left|H_{\gamma}^{\top}(0) H_{\gamma}(X_i-t)\right| \leq \frac{1}{n h^d} \sqrt{\frac{D_m}{\lambda_{\gamma}^2}} = c_{\gamma}.
        \end{equation}
        Using Bernstein's inequality \citep[see][Theorem~2.10]{BoucheronLugosiMassart2013supp}, the conclusion follows.
    \end{proof}

    \begin{lemma}[Concentration of $\hat{v}_{\gamma}$ around $v_{\gamma}$]\label{lem:Bernstein.2}
        Let $\gamma = (m,h)\in \Gamma$ be given.
        Recall the definitions of $\hat{v}_{\gamma}$ and $v_{\gamma}$ from~\eqref{eq:hat.v.gamma} and~\eqref{eq:v.gamma}.
        Assuming that $n h^d W_h \geq (\log n)^3$ and $\finfty < \infty$, there exists a constant $\kappa_1 > 0$ that depends on $\delta > 1$ and $\finfty$ such that
        \begin{equation*}
            \PP\left(|\hat{v}_{\gamma} - v_{\gamma}| > \varepsilon_{\gamma}\right) \leq 2 \exp\big\{-\kappa_1(\log n)^3\big\},
        \end{equation*}
        where recall from~\eqref{eq:c.gamma} that $\varepsilon_{\gamma} = (\delta-1) D_m W_h / (n h^d \lambda_{\gamma}^2)$.
    \end{lemma}

    \begin{proof}[Proof of Lemma~\ref{lem:Bernstein.2}]
        Similarly to the proof of Lemma~\ref{lem:Bernstein.1}, observe that
        \begin{equation*}
            \hat{v}_{\gamma} - v_{\gamma} = \sum_{i=1}^n \left[Z_i(\gamma) - \EE\{Z_i(\gamma)\}\right],
        \end{equation*}
        where
        \begin{equation*}
            Z_i(\gamma) = \frac{1}{n^2} \left\{H_{\gamma}^{\top}(0) H_{\gamma}(X_i-t) w_h(X_i-t)\right\}^2.
        \end{equation*}
        By the bound~\eqref{eq:H.H.2}, the fact that $w_h^4(X_i-t) = h^{-3d} w_h(X_i-t)$ and $0 \leq w_h \leq h^{-d}$ because of~\eqref{eq:w.h} and~\eqref{eq:K}, and $W_h = \smash{\int_{\R^d} w_h(u) \rd u}$ as defined in~\eqref{eq:L.n.W.h}, we have
        \begin{equation*}
            \sum_{i=1}^n \EE\{Z_i^2(\gamma)\} \leq \frac{\finfty D_{m}^2 W_h}{n^3 h^{3d} \lambda_{\gamma}^4} \eqqcolon \tilde{v}_{\gamma}
            \quad \text{and} \quad
            \lvert Z_i(\gamma)\rvert \leq \frac{D_{m}}{n^2 h^{2d} \lambda_{\gamma}^2} \eqqcolon \tilde{c}_{\gamma}.
        \end{equation*}
        Using another version of Bernstein's inequality \citep[see][Corollary~2.11]{BoucheronLugosiMassart2013supp}, we obtain
        \begin{equation*}
            \PP\left(|\hat{v}_{\gamma} - v_{\gamma}| > \varepsilon_{\gamma}\right) \leq 2 \exp\left\{-\frac{\varepsilon_{\gamma}^2}{2(\tilde{v}_{\gamma} + \tilde{c}_{\gamma}\varepsilon_{\gamma})}\right\} = 2 \exp\left\{-\frac{n h^dW_h(\delta-1)^2}{2(\finfty +\delta-1)}\right\}.
        \end{equation*}
        Since we assumed that $n h^dW_h\geq (\log n)^3$, the conclusion follows.
    \end{proof}

\section{Proofs of the propositions and theorems}\label{sec:proofs}

\phantomsection
\addcontentsline{toc}{subsection}{Proof of Proposition~\ref*{prop:control.bias}}

\begin{proof}[Proof of Proposition~\ref*{prop:control.bias}]
    Set $t\in \domain$ and $\gamma = (m,h)\in \Gamma$.
    Since $f\in \Sigma(s,L)$, there exists a polynomial $q\in \mathcal{P}_{\llfloor s \rrfloor}$ such that, for any $u\in \voisin(h)$,
    \begin{equation}\label{eq:lem:control.bias.first}
        f(t+u) = q(u) + R(u) \quad \text{where} \quad |R(u)| \leq L \|u\|_{\infty}^s.
    \end{equation}
    Consider the decomposition
    \begin{equation}\label{eq:q.decomposition}
        \begin{aligned}
            q(u)
            &= \sum_{|\alpha|\leq \llfloor s \rrfloor} h^{|\alpha|} q_{\alpha} \varphi_{\alpha}\left(\frac{u}{h}\right)
            = Q_m^{\top} \Phi_{\gamma}(u) + \1_{[m+1,\infty)}(\llfloor s \rrfloor) \sum_{m < |\alpha|\leq \llfloor s \rrfloor} q_{\alpha} \varphi_{\alpha}(u),
        \end{aligned}
    \end{equation}
    where $Q_m$ denotes the vector of the first $D_m$ components of the full $D_{\llfloor s \rrfloor} \times 1$ coordinate vector of $q$ in the basis $\Phi_{(\llfloor s \rrfloor,h)}$.
    Given~\eqref{eq:lem:control.bias.first} and the decomposition of $q$ in~\eqref{eq:q.decomposition}, we can write
    \begin{equation}\label{eq:truncated.decomp}
        f(t+u) = Q_m^{\top} \Phi_{\gamma}(u) + \widetilde{R}(u),
    \end{equation}
    where
    \begin{equation*}
        \begin{aligned}
            |\widetilde{R}(u)|
            &= \left|\1_{[m+1,\infty)}(\llfloor s \rrfloor) \sum_{m < |\alpha|\leq \llfloor s \rrfloor} q_{\alpha} \varphi_{\alpha}(u) + R(u)\right| \\
            &\leq \1_{[m+1,\infty)}(\llfloor s \rrfloor) \left(\sum_{m < |\alpha| \leq \llfloor s \rrfloor} |q_{\alpha}|\right) \|u\|_{\infty}^{\beta_m(s)} + L \|u\|_{\infty}^s
            \leq \mathfrak{L}_{m,s,L} \|u\|_{\infty}^{\beta_m(s)},
        \end{aligned}
    \end{equation*}
    for some appropriate positive real constant $\mathfrak{L}_{m,s,L}$ that depends on $m$, $s$ and $L$, and also denoting $\beta_m(s) = \min(m+1,s)$.
    Note that we can choose $\mathfrak{L}_{m,s,L} = L$ whenever $m\geq\llfloor s \rrfloor$ because the indicator $\1_{[m+1,\infty)}(\llfloor s \rrfloor)$ is simply equal to zero in that case.

    By using the definition of $\hat{f}_{\gamma}(t)$ in~\eqref{eq:hat.f.gamma} and the definition of $H_{\gamma}$ in~\eqref{eq:orthogonalization}, followed by an application of~\eqref{eq:truncated.decomp} together with the fact that $\Phi_{\gamma}^{\top}(0) Q_m = f(t)$, we have
    \begin{equation}\label{eq:calculation.expectation.f.hat}
        \begin{aligned}
            \EE\left\{\hat{f}_{\gamma}(t)\right\}
            &= \int_{\R^d} H_{\gamma}^{\top}(0) H_{\gamma}(u) w_h(u) f(t+u) \rd u \\
            &= \Phi_{\gamma}^{\top}(0) \mathcal{B}_{\gamma}^{-1}\int_{\R^d} \Phi_{\gamma}(u) w_h(u) f(t+u) \rd u \\
            &= f(t) + \Phi_{\gamma}^{\top}(0) \mathcal{B}_{\gamma}^{-1}\int_{\R^d} \Phi_{\gamma}(u) w_h(u) \widetilde{R}(u) \rd u.
        \end{aligned}
    \end{equation}
    Since $\|\Phi_{\gamma}(0)\|_2 = 1$, $\operatorname{Supp}(w_h) = \voisin(h)$, $\max_{|\alpha| \leq m} |\varphi_{\alpha}(u/h)| \leq 1$ for all $u\in \voisin(h)$ and $\|u\|_{\infty} \leq h$ for all $u\in \voisin(h)$, the error term on the right-hand side of~\eqref{eq:calculation.expectation.f.hat} satisfies
    \begin{equation*}
        \begin{aligned}
            \left|\Phi_{\gamma}^{\top}(0) \mathcal{B}_{\gamma}^{-1}\int_{\R^d} \Phi_{\gamma}(u) w_h(u) \widetilde{R}(u) \rd u\right|
            &\leq \|\Phi_{\gamma}(0)\|_2 \left\| \mathcal{B}_{\gamma}^{-1}\int_{\R^d} \Phi_{\gamma}(u) w_h(u) \widetilde{R}(u) \rd u \right\|_2 \\
            &\leq \big\| \mathcal{B}_{\gamma}^{-1} \big\|_2 \left\|\int_{\R^d} \Phi_{\gamma}(u) w_h(u) \widetilde{R}(u) \rd u \right\|_2 \\
            &\leq \lambda_{\gamma}^{-1} \sqrt{D_m}\int_{\R^d} \max_{|\alpha| \leq m} |\varphi_{\alpha}(u/h)| \times |\widetilde{R}(u)| w_h(u) \rd u \\
            &\leq \lambda_{\gamma}^{-1} \sqrt{D_m} \, \mathfrak{L}_{m,s,L}\int_{\R^d} \|u\|_{\infty}^{\beta_m(s)}w_h(u)\rd u \\
            &\leq \lambda_{\gamma}^{-1} \sqrt{D_m} \, \mathfrak{L}_{m,s,L} \, h^{\beta_m(s)} W_h.
        \end{aligned}
    \end{equation*}
    This concludes the proof.
\end{proof}

\phantomsection
\addcontentsline{toc}{subsection}{Proof of Proposition~\ref*{prop:control.stochastic.term}}

\begin{proof}[Proof of Proposition~\ref*{prop:control.stochastic.term}]
    By using $w_h^2(u) = h^{-d} w_h(u)$ because of~\eqref{eq:w.h} and~\eqref{eq:K}, the bound~\eqref{eq:H.H.2} from proof of Lemma~\ref{lem:Bernstein.1}, and $W_h = \smash{\int_{\R^d} w_h(u) \rd u}$ as defined in~\eqref{eq:L.n.W.h}, we have
    \begin{equation*}
        \begin{aligned}
            \Var\left\{\hat{f}_{\gamma}(t)\right\} \leq v_{\gamma}
            &= \frac{1}{n}\int_{\R^d} \left\{H_{\gamma}^{\top}(0) H_{\gamma}(u) w_h(u)\right\}^2 f(t+u) \rd u \\
            &\leq \frac{\finfty}{n h^d}\int_{\R^d} \left\{H_{\gamma}^{\top}(0) H_{\gamma}(u)\right\}^2 w_h(u) \rd u
            \leq \frac{\finfty}{n h^d} \times \frac{D_m W_h}{\lambda_{\gamma}^2}.
        \end{aligned}
    \end{equation*}
    This concludes the proof.
\end{proof}

\phantomsection
\addcontentsline{toc}{subsection}{Proof of Theorem~\ref*{thm:oracle.result}}

\begin{proof}[Proof of Theorem~\ref*{thm:oracle.result}]
    Set $t\in \domain$ and $\gamma=(m,h)\in \Gamma_{\!n}$.
    By following the line of proof of Theorem~1 (Steps~2 and~3) in \cite{BertinKlutchnikoff2017supp}, we obtain
    \begin{equation*}
        |\hat{f}(t) - f(t)| \leq 2\big(\hat{A}_{\gamma} + \hat{\mathbb{U}}_{\gamma}\big) + |\hat{f}_{\gamma}(t) - f(t)|,
    \end{equation*}
    where $\hat{A}_{\gamma}$ and $\hat{\mathbb{U}}_{\gamma}$ are defined in~\eqref{eq:hat.A.gamma} and~\eqref{eq:hat.U.gamma}, respectively.
    Furthermore, the paper also establishes that
    \begin{equation*}
        \hat{A}_{\gamma} \leq 2 \mathbb{B}_{\gamma} + 2T,
        \quad \text{where} \quad
        T = \max_{\gamma'\in \Gamma_{\!n}} \left[\left|\hat{f}_{\gamma'}(t) - \EE\left\{\hat{f}_{\gamma'}(t)\right\}\right| - \hat{\mathbb{U}}_{\gamma'}\right]_+.
    \end{equation*}
    Recall that $(\cdot)_+ = \max\{\cdot \, ,0\}$.
    Together, the last two equations imply
    \begin{equation}\label{eq:thm:oracle.result.bound.decomp.R.t}
        \pwrisk(\hat{f}, f) \leq 4 \mathbb{B}_{\gamma} + 2 \big\{\EE(\hat{\mathbb{U}}_{\gamma}^2)\big\}^{1/2} + \pwrisk(\hat{f}_{\gamma}, f) + 4 \big\{\EE(T^2)\big\}^{1/2}.
    \end{equation}

    Using the triangle inequality for the $\EE\{(\cdot)^2\}^{1/2}$ norm, $\EE(\hat{v}_{\gamma}) = v_{\gamma}$, and the subadditivity of the function $x\mapsto x^{1/2}$, we obtain
    \begin{equation}\label{eq:thm:oracle.result.bound.norm.hat.U.gamma}
        \begin{aligned}
            \big\{\EE(\hat{\mathbb{U}}_{\gamma}^2)\big\}^{1/2}
            &\leq \sqrt{\EE\left\{2(\hat{v}_{\gamma}+\varepsilon_{\gamma})\mathrm{pen}(\gamma)\right\}} + c_{\gamma} \mathrm{pen}(\gamma) \\
            &= \sqrt{2(v_{\gamma}+\varepsilon_{\gamma})\mathrm{pen}(\gamma)} + c_{\gamma} \mathrm{pen}(\gamma) \\
            &\leq \sqrt{2v_{\gamma} \mathrm{pen}(\gamma)}+ c_{\gamma} \mathrm{pen}(\gamma) +\sqrt{2\varepsilon_{\gamma} \mathrm{pen}(\gamma)} \\
            &= \mathbb{U}_{\gamma} + \sqrt{2\varepsilon_{\gamma} \mathrm{pen}(\gamma)} = \mathbb{U}_{\gamma} + \sqrt{(\delta-1) / \finfty} \sqrt{2 v_{\gamma}^{\star} \mathrm{pen}(\gamma)} \\
            &\leq \left\{1 + \sqrt{(\delta-1) / \finfty}\right\} \mathbb{U}_{\gamma}^{\star},
        \end{aligned}
    \end{equation}
    where the quantities $\hat{v}_{\gamma}$, $\varepsilon_{\gamma}$, $c_{\gamma}$, $\mathrm{pen}(\gamma)$ and $\mathbb{U}_{\gamma}$ are all defined in Section~\ref*{sec:selection-procedure}, and $\mathbb{U}_{\gamma}^{\star}$ is defined in~\eqref{eq:B.gamma.and.U.star.gamma}.
    Note also that, since $\delta$ and $\lvert\log h\rvert$ are both at least $1$ by assumption (we assumed $\delta > 1$ in Section~\ref*{sec:selection-procedure}, and also $h\in (0,\rho] \subseteq (0,e^{-1}]$ at the beginning of Section~\ref*{sec:framework}), we have $\mathrm{pen}(\gamma) \geq 1$, and thus $\sqrt{v_{\gamma}} \leq \sqrt{2v_{\gamma} \mathrm{pen}(\gamma)}+ c_{\gamma} \mathrm{pen}(\gamma) = \mathbb{U}_{\gamma}$. Therefore,
    \begin{equation}\label{eq:thm:oracle.result.bound.R.t}
        \begin{aligned}
            \pwrisk(\hat{f}_{\gamma}, f)
            &\leq \left|\EE\left\{\hat{f}_{\gamma}(t)\right\} - f(t)\right| + \left(\EE\left[\left|\hat{f}_{\gamma}(t) - \EE\left\{\hat{f}_{\gamma}(t)\right\}\right|^2\right]\right)^{1/2} \\
            &\leq \mathbb{B}_{\gamma} + \sqrt{v_{\gamma}} \leq \mathbb{B}_{\gamma} + \mathbb{U}_{\gamma} \leq \mathbb{B}_{\gamma} + \mathbb{U}_{\gamma}^{\star},
        \end{aligned}
    \end{equation}
    where the last inequality follows from $v_{\gamma} \leq v_{\gamma}^{\star}$ in~\eqref{eq:variance.bound}.

    By applying the bounds~\eqref{eq:thm:oracle.result.bound.norm.hat.U.gamma} and~\eqref{eq:thm:oracle.result.bound.R.t} back into~\eqref{eq:thm:oracle.result.bound.decomp.R.t}, we get
    \begin{equation}\label{eq:decomp.adaptive}
        \pwrisk(\hat{f}, f) \leq \left[5 \mathbb{B}_{\gamma} + \left\{3 + 2 \sqrt{(\delta-1) / \finfty}\right\} \mathbb{U}_{\gamma}^{\star}\right] + 4 \big\{\EE(T^2)\big\}^{1/2}.
    \end{equation}
    To conclude, it remains to study the term $\EE(T^2)$. To accomplish this, define the event
    \begin{equation}\label{eq:set.A}
        \mathcal{A} = \bigcap_{\gamma\in \Gamma_{\!n}}\left\{|\hat{v}_{\gamma} - v_{\gamma}| \leq \varepsilon_{\gamma}\right\}.
    \end{equation}
    We decompose the expectation of interest as follows:
    \begin{equation*}
        \EE(T^2) = \EE(T^2\1_{\mathcal{A}}) + \EE(T^2\1_{\mathcal{A}^c}).
    \end{equation*}

    Let us bound the term $\EE(T^2\1_{\mathcal{A}^c})$ first. Notice that $D_m = \smash{\binom{m+d}{d}} \leq \mathfrak{K}_d m^d$ for some positive real constant $\mathfrak{K}_d$, which implies
    \begin{equation*}
        T
        \leq \max_{\gamma\in \Gamma_{\!n}} \left|\hat{f}_{\gamma}(t) - \EE\left\{\hat{f}_{\gamma}(t)\right\}\right|
        \leq 2n\max_{\gamma\in \Gamma_{\!n}} c_{\gamma}
        \leq 2^{d+1} \sqrt{\mathfrak{K}_d} (\log n)^{d/2-3} n \, \max_{\gamma\in \Gamma_{\!n}} \frac{1}{\lambda_{\gamma}},
    \end{equation*}
    where the second inequality follows from~\eqref{eq:hat.f.Y.i.decomp} and~\eqref{eq:bound.Y.i}, and the last inequality follows from the definition of $c_{\gamma}$ in~\eqref{eq:c.gamma} and the aforementioned bound on $D_m$ together with the fact that $m \leq \log n$ and $(n h^d)^{-1} \leq W_h (\log n)^{-3} \leq 2^d (\log n)^{-3}$ for $(m,h)\in \Gamma_n$; recall~\eqref{eq:L.n.W.h}.
    Using our assumption on $\lambda_{\gamma}$ in~\eqref{eq:domain.lambda}, together with that fact that $m\leq \log n$ and $h \geq \rho n^{-1}$ (recall~\eqref{eq:h.in.Gamma.n}), we obtain, for $n$ large enough,
    \begin{equation*}
        T \leq 2^{d+1} \sqrt{\mathfrak{K}_d} (\log n)^{d/2-3} n \exp\big\{2 b(\log n)^2\big\}.
    \end{equation*}
    By a union bound and the concentration bound on $\hat{v}_{\gamma}$ in Lemma~\ref{lem:Bernstein.2}, this implies
    \begin{equation*}
        \begin{aligned}
            \EE(T^2\1_{\mathcal{A}^c})
            &\leq 4^{d+1} \mathfrak{K}_d (\log n)^{d-6} n^2 \exp\big\{4 b(\log n)^2\big\} \times \PP(\mathcal{A}^c) \\
            &\leq 4^{d+1} \mathfrak{K}_d (\log n)^{d-6} n^2 \exp\big\{4 b(\log n)^2\big\} \times \operatorname{card}(\Gamma_{\!n}) \times 2 \exp\big\{-\kappa_1 (\log n)^3\big\}.
        \end{aligned}
    \end{equation*}
    Since $\operatorname{card}(\Gamma_{\!n})\leq \log n$, we deduce that
    \begin{equation}\label{eq:T2barA}
        \EE(T^2\1_{\mathcal{A}^c}) \lesssim \exp\big\{-(\kappa_1 / 2) (\log n)^3\big\},
    \end{equation}
    with much room to spare.

    It remains to study $\EE(T^2\1_{\mathcal{A}})$.
    Note that under the event $\mathcal{A}$ defined in~\eqref{eq:set.A}, we have $\hat{\mathbb{U}}_{\gamma} \geq \mathbb{U}_{\gamma}$ for all $\gamma\in \Gamma_{\!n}$ (to see this, compare~\eqref{eq:hat.U.gamma} and~\eqref{eq:U.gamma}).
    This implies that
    \begin{equation*}
        \EE(T^2\1_{\mathcal{A}}) \leq \EE(\tilde{T}^2)
    \end{equation*}
    where
    \begin{equation*}
        \tilde{T} = \max_{\gamma'\in \Gamma_{\!n}} \left[\left|\hat{f}_{\gamma'}(t) - \EE\left\{\hat{f}_{\gamma'}(t)\right\}\right| - \mathbb{U}_{\gamma'}\right]_+.
    \end{equation*}
    For simplicity of notations, define $r_{\gamma'}(x) = r_{\gamma'}(v_{\gamma'}, x)$; recall \eqref{eq:r.gamma.and.pen.gamma}.
    Using integration by parts and the change of variable $u = r_{\gamma'}(x)$, we obtain
    \begin{equation*}
        \begin{aligned}
            \EE(\tilde{T}^2)
            &\leq \sum_{\gamma'\in \Gamma_{\!n}}\EE \left(\left[\left|\hat{f}_{\gamma'}(t) - \EE\left\{\hat{f}_{\gamma'}(t)\right\}\right| - \mathbb{U}_{\gamma'}\right]_+^2\right) \\
            &= \sum_{\gamma'\in \Gamma_{\!n}}\int_0^{\infty} 2 u \PP\left[\left|\hat{f}_{\gamma'}(t) - \EE\left\{\hat{f}_{\gamma'}(t)\right\}\right| > \mathbb{U}_{\gamma'} + u\right] \rd u \\
            &= \sum_{\gamma'\in \Gamma_{\!n}}\int_0^{\infty} 2r_{\gamma'}(x)r'_{\gamma'}(x) \PP\left[\left|\hat{f}_{\gamma'}(t) - \EE\left\{\hat{f}_{\gamma'}(t)\right\}\right| > r_{\gamma'}\left\{\mathrm{pen}(\gamma')\right\} + r_{\gamma'}(x)\right] \rd x.
        \end{aligned}
    \end{equation*}
    Now, using the fact that $r_{\gamma'}(\cdot)$ is a sub-additive function which also satisfies $x r'_{\gamma'}(x)\leq r_{\gamma'}(x)$ for all $x > 0$, we get
    \begin{equation*}
        \EE(\tilde{T}^2)
        \leq \sum_{\gamma'\in \Gamma_{\!n}}\int_0^{\infty} 2 x^{-1} r_{\gamma'}^2(x) \PP\left[\left|\hat{f}_{\gamma'}(t) - \EE\left\{\hat{f}_{\gamma'}(t)\right\}\right| > r_{\gamma'}\left\{\mathrm{pen}(\gamma')+x\right\}\right] \rd x.
    \end{equation*}
    As we noted earlier just above~\eqref{eq:thm:oracle.result.bound.R.t}, we have $\mathrm{pen}(\gamma') \geq 1$ because of our assumptions $\delta > 1$ and $-|\log \rho| \geq 1$, which allows us to apply the concentration of $\smash{\hat{f}_{\gamma'}(t)}$ around its mean in Lemma~\ref{lem:Bernstein.1}.
    Together with the identity $r_{\gamma'}^2(x) \leq 4 v_{\gamma'} x + 2 c_{\gamma'}^2 x^2$ and $\lvert\log h_{\ell'}\rvert = \lvert\log \rho - \ell'\rvert \geq \ell' -|\log \rho| \geq \ell'$, this yields
    \begin{equation*}
        \begin{aligned}
            \EE(\tilde{T}^2)
            &\leq \sum_{\gamma'\in \Gamma_{\!n}}\int_0^{\infty} 2 x^{-1} r_{\gamma'}^2(x) \times 2 \exp\left(-d\delta\lvert\log h_{\ell'}\rvert - \Lambda_{\gamma'} - x\right) \rd x \\
            &\leq I_1 \sum_{\gamma'\in \Gamma_{\!n}} (2v_{\gamma'} + c_{\gamma'}^2) \exp\left(-d\delta\ell' - \Lambda_{\gamma'}\right),
        \end{aligned}
    \end{equation*}
    where $I_1 =\int_0^{\infty} 8 \max(1,x) \exp(-x) \rd x <\infty$.
    Since $n h^d_{\ell'}W_{h_{\ell'}}\geq (\log n)^3$, $\Lambda_{\gamma'} = 2\lvert\log(\lambda_{\gamma'})\rvert$ and $h_{\ell'} = \rho \exp(-\ell')$, we can write, using $v_{\gamma} \leq v_{\gamma}^{\star}$ in Proposition~\ref*{prop:control.stochastic.term} and $c_{\gamma'}$ in~\eqref{eq:c.gamma},
    \begin{equation*}
        \EE(\tilde{T}^2) \leq \frac{I_1}{n \rho^d} \left\{2 \finfty + \frac{1}{(\log n)^3}\right\} \sum_{\ell'\in \mathcal{L}_n} D_{m_{\ell'}} W_{h_{\ell'}} \exp\left\{-d(\delta-1)\ell')\right\}.
    \end{equation*}
    Moreover, since $D_{m_{\ell'}} \leq \mathfrak{K}_d m_{\ell'}^d$, $m_{\ell'} \leq \log n$ for all $\ell'\in \mathcal{L}_n$, and $W_{h_{\ell'}} \leq 2^d$, we can write
    \begin{equation*}
        \EE(\tilde{T}^2) \leq \frac{I_1 \mathfrak{K}_d (\log n)^d 2^d}{n \rho^d} \left\{2 \finfty + \frac{1}{(\log n)^3}\right\} \sum_{\ell'\in \mathcal{L}_n} \exp\left\{-d(\delta-1)\ell'\right\}.
    \end{equation*}
    Since $d(\delta-1)$ is a positive number (we assumed $\delta > 1$ in Section~\ref*{sec:selection-procedure}), the exponential terms above are summable, so we have
    \begin{equation}\label{eq:T2A}
        \EE(T^2\1_{\mathcal{A}}) \leq \EE(\tilde{T}^2) \lesssim \frac{(\log n)^d}{n}.
    \end{equation}
    By combining~\eqref{eq:T2barA} and~\eqref{eq:T2A}, we conclude that $\EE(T^2) \lesssim (\log n)^d/n$.
    Together with~\eqref{eq:decomp.adaptive}, the proof of~\eqref{eq:thm:oracle.result.to.prove} is complete.
\end{proof}

\phantomsection
\addcontentsline{toc}{subsection}{Proof of Proposition~\ref*{prop:minimaxity.simple.domains}}

\begin{proof}[Proof of Proposition~\ref*{prop:minimaxity.simple.domains}]
    In this proof, the parameters $\gamma = (m,h)$ are chosen to be the ones in~\eqref{eq:def-gamma-minimax-usual}. By combining the bounds on the bias and variance found in Propositions~\ref*{prop:control.bias} and~\ref*{prop:control.stochastic.term} together with Lemma~\ref{lem:usual.behavior} ($W_h \leq 2^d$ and $\inf_{h\in (0,\rho]} \lambda_{\gamma} \geq \lambda_{\star}(m) > 0$ under Assumption~\ref*{ass:1}) and Lemma~\ref{lem:infty.norm.f}, we obtain
    \begin{equation*}
        \begin{aligned}
            \pwrisk(\hat{f}_{\gamma}, f)
            &\leq \frac{\sqrt{D_m}}{\lambda_{\gamma}} \times \left\{W_h^2 \mathfrak{L}_{m,s,L}^2 \, h^{2s} + \frac{W_h \lVert f\rVert_{\voisin(\rho)}}{nh^d}\right\}^{1/2} \\
            &\leq \frac{2^d \sqrt{D_m}}{\lambda_{\star}(m)} \times \left[\mathfrak{L}_{m,s,L}^2 \, h^{2s} + \frac{\mathfrak{F}_{s,L}}{2^d nh^d}\right]^{1/2}.
        \end{aligned}
    \end{equation*}
    Our choice of $h$ in~\eqref{eq:def-gamma-minimax-usual} is simply the minimizer of the last bound. In particular, it should be noted that $h$ decreases as $n$ increases.
    Elementary computations show that, with this choice of $h$, the above yields
    \begin{equation*}
        \pwrisk(\hat{f}_{\gamma}, f) \leq C(s,L) N_n(s,L),
    \end{equation*}
    where $C(s,L)$ is the positive real constant defined in~\eqref{eq:C.s.L} and $N_n(s,L) =  n^{-s/(2s+d)}$ is the rate of convergence defined in~\eqref{eq:N.n}. This proves the upper bound.

    To prove the lower bound, we use Lemma~3 of \cite{Lepski2015supp}.
    It is sufficient to construct two functions $f_0$ and $f_1$ that satisfy the following properties:
    \begin{enumerate}
        \item $f_0$ and $f_1$ are two density  functions that belong to $\Sigma(s,L)$.
        \item There exists a positive real constant $A = A(s,L) > 0$ such that $f_1(t) - f_0(t) = A h^s$.
        \item There exists a positive real constant $\mathfrak{a} = \mathfrak{a}(s,L) > 0$ such that
        \begin{equation*}
            \EE_{f_0,n}\left\{\prod_{i=1}^n \frac{f_1(X_i)}{f_0(X_i)}\right\}^2
            \leq \mathfrak{a},
        \end{equation*}
        where $\EE_{f_0,n}$ denotes the expectation with respect to the law $\PP_{\!f_0,n}$ of the random sample $\mathbb{X}_n$, provided $f_0$ is the true density of the observations $X_i$.
    \end{enumerate}
    To construct these densities, we consider an auxiliary density function $\psi\in \mathcal{C}^{\infty}(\R^d)$ whose support is $\Delta$ and which satisfies $\psi(0) = \mathrm{Leb}(\domain)$. Such a function can easily be constructed. We also define, for any $u\in \R^d$,
    \begin{equation*}
        \psi_h(u) = \psi\left(\frac{u}{h}\right) \1_{\voisin(h)}(u)
        \quad \text{and} \quad
        c_h = \frac{1}{\mathrm{Leb}(\domain)}\int_{\R^d} \psi_h(u) \rd u.
    \end{equation*}
    Furthermore, for all $x\in \R^d$, define
    \begin{equation*}
        f_0(x) = \frac{\1_{\domain}(x)}{\mathrm{Leb}(\domain)}
        \quad \text{and} \quad
        f_1(x) = f_0(x) + A h^s \tilde{\psi}_h(x-t) \1_{\domain}(x),
    \end{equation*}
    for some constant $A > 0$ to be chosen later, where
    \begin{equation*}
        \tilde{\psi}_h(u)
        =
        \frac{\{\psi_h(u) / \mathrm{Leb}(\domain)\}  - c_h}{1-c_h},
        \quad u\in \R^d.
    \end{equation*}
    Since $\tilde{\psi}(0) = 1$, the second point follows.

    Now, since $\psi$ is a density function, we have $c_h \leq h^d/\mathrm{Leb}(\domain)$.
    Hence, using the fact that the function $x\mapsto -x / (1 - x)$ is decreasing on $(0,1)$, and taking $n$ large enough (and thus $h$ small enough by~\eqref{eq:def-gamma-minimax-usual}) so that $h^d / \mathrm{Leb}(\domain)\leq 1/2$, we have
    \begin{equation*}
        \tilde{\psi}_h(u) \geq \frac{- c_h}{1-c_h} \geq \frac{- h^d / \mathrm{Leb}(\domain)}{1 - h^d / \mathrm{Leb}(\domain)} \geq - 2 \frac{h^d}{\mathrm{Leb}(\domain)},
    \end{equation*}
    which also implies that $f_1\geq 0$ for large $n$. We easily deduce that $f_1$ is a density function using the fact that $\smash{\int_{\domain - t} \tilde{\psi}_h(u) \rd u = 0}$. Since $f_0$ is also a density function, it remains to prove that both $f_0$ and $f_1$ belong to $\Sigma(s,L)$ to obtain the first point. This is obvious for $f_0$. Now, since $\psi\in \mathcal{C}^{\infty}(\R^d)$ is supported on $\Delta$, there exists a positive real constant $L_s(\psi) > 0$ such that $\psi\in \Sigma\{s,L_s(\psi)\}$. In turn, this implies that $f_1\in \Sigma\{s, A \times 2 \, L_s(\psi) / \mathrm{Leb}(\domain)\}$. By choosing the constant $A = \{2 \, L_s(\psi) / \mathrm{Leb}(\domain)\}^{-1} L$, we deduce that $f_1\in \Sigma(s,L)$.

    Lastly, it remains to prove the third point. Note that
    \begin{equation}\label{eq:I.develop}
        \begin{aligned}
            \EE_{f_0,n}\left\{\prod_{i=1}^n \frac{f_1(X_i)}{f_0(X_i)}\right\}^2
            &= \left\{\int_{\domain} \frac{f_1^2(x)}{f_0(x)} \rd x\right\}^n \\
            &= \left[\int_{\domain - t} \frac{\{1+\mathrm{Leb}(\domain) A h^s \tilde{\psi}_h(u)\}^2}{\mathrm{Leb}(\domain)} \rd u\right]^n \\
            &=\left\{1+ \mathrm{Leb}(\domain) A^2 h^{2s}\int_{\domain - t} \tilde{\psi}_h^2(u) \rd u\right\}^n.
        \end{aligned}
    \end{equation}
    To obtain the last equality, we used the fact that $\smash{\int_{\domain - t} \tilde{\psi}_h(u) \rd u = 0}$.
    It remains to study the integral that appears on the last line.
    Using the elementary identity $(a + b)^2 \leq 2 a^2 + 2 b^2$ together with the aforementioned bound $c_h \leq h^d/\mathrm{Leb}(\domain)$, and taking $n$ large enough (and thus $h$ small enough by~\eqref{eq:def-gamma-minimax-usual}) so that
    \begin{equation*}
        1 - c_h \geq 1/2 \quad \text{and} \quad h^d \leq 1/\mathrm{Leb}(\domain),
    \end{equation*}
    we have
    \begin{equation}\label{eq:I.develop.next}
        \begin{aligned}
            \int_{\domain - t} \tilde{\psi}_h^2(u) \rd x
            &\leq \frac{2}{\{\mathrm{Leb}(\domain)\}^2} \int_{\voisin(h)} \frac{\psi^2(u/h)}{(1-c_h)^2} \rd u + 2\int_{\domain} \frac{(- c_h)^2}{(1-c_h)^2} \rd u \\
            &\leq \frac{8 h^d}{\{\mathrm{Leb}(\domain)\}^2} \int_{\Delta} \psi^2(u) \rd u + \frac{8 h^{2d}}{\mathrm{Leb}(\domain)} \\
            &\leq \frac{8}{\{\mathrm{Leb}(\domain)\}^2} \left\{\int_{\Delta} \psi^2(u) \rd u + 1\right\} h^d.
        \end{aligned}
    \end{equation}
    Finally, since $1+x \leq e^x$, putting~\eqref{eq:I.develop.next} into~\eqref{eq:I.develop} yields the following bound:
    \begin{equation}\label{eq:end.proof.Prop.4.5}
        \begin{aligned}
            \EE_{f_0,n}\left\{\prod_{i=1}^n \frac{f_1(X_i)}{f_0(X_i)}\right\}^2
            &\leq \left[ 1 + \frac{8 A^2}{\mathrm{Leb}(\domain)} \left\{\int_{\Delta} \psi^2(u) \rd u + 1\right\} h^{2s+d} \right]^n \\[-1mm]
            &\leq \exp\left[\frac{8 A^2}{\mathrm{Leb}(\domain)} \left\{\int_{\Delta} \psi^2(u) \rd u + 1\right\} n h^{2s+d}\right] \\
            &= \exp\left[\frac{d \, \mathfrak{F}_{s,L}}{2^{d+1} s \mathfrak{L}_{m,s,L}^2} \times \frac{8 A^2}{\mathrm{Leb}(\domain)} \left\{\int_{\Delta} \psi^2(u) \rd u + 1\right\}\right] \eqqcolon \mathfrak{a}.
        \end{aligned}
    \end{equation}
    This concludes the proof.
\end{proof}

\phantomsection
\addcontentsline{toc}{subsection}{Proof of Theorem~\ref*{thm:adaptivity.simple.domains}}

\begin{proof}[Proof of Theorem~\ref*{thm:adaptivity.simple.domains}]
    We prove the upper bound first. Let $(s,L)\in (0,\infty)^2$ be given.
    Under Assumption~\ref*{ass:1}, Lemma~\ref{lem:usual.behavior} shows that
    \begin{equation*}
        \lambda_{\gamma} \geq \min_{0\leq m' \leq \lfloor 2s+d \rfloor} \lambda_{(m',h)} \geq \lambda_{\bullet}(s,d) > 0,
    \end{equation*}
    where recall the quantity $\lambda_{\bullet}(s,d) = \min_{0\leq m' \leq \lfloor 2s+d \rfloor} \lambda_{\star}(m')$ depends only on $s$ and $d$.
    In particular, the condition~\eqref{eq:domain.lambda} of Theorem~\ref*{thm:oracle.result} is satisfied.
    Moreover, Lemma~\ref{lem:infty.norm.f} shows that there exists an appropriate constant $\mathfrak{F}_{s,L} > 0$ such that
    \begin{equation*}
        \sup_{f\in \Sigma(s,L)} \finfty \leq \mathfrak{F}_{s,L} <\infty.
    \end{equation*}
    Therefore, Theorem~\ref*{thm:oracle.result} can be applied to obtain
    \begin{equation}\label{eq:oracle.adaptivity.simple.domains}
        \pwrisk(\hat{f}, f)
        \leq \min_{\gamma\in \Gamma_{\!n}} \left[5 \mathbb{B}_{\gamma} + \left\{3 + 2 \sqrt{(\delta-1) / \finfty}\right\} \mathbb{U}_{\gamma}^{\star}\right] + \mathcal{O}\left\{\sqrt\frac{(\log n)^d}{n}\right\}.
    \end{equation}

    Now, define
    \begin{equation*}
        \ell_0 = \left\lfloor\frac{1}{2s+d} \log\left\{\frac{(\mathfrak{L}_{s,L}^{\star})^2}{\mathfrak{F}_{s,L}} \times \frac{n}{\log n}\right\} + \log \rho\right\rfloor,
    \end{equation*}
    where recall $\mathfrak{L}_{s,L}^{\star} = \max_{0\leq m' \leq \lfloor 2s+d \rfloor} \mathfrak{L}_{m',s,L}$ depends only on $s$, $L$ and $d$. In the following, consider $n$ large enough so that $\ell_0\geq 1$.
    Using this definition, we have
    \begin{equation}\label{eq:h.asymptotics}
        1 \leq h_{\ell_0} \left\{\frac{\mathfrak{F}_{s,L}}{(\mathfrak{L}_{s,L}^{\star})^2} \times \frac{\log n}{n}\right\}^{-1/(2s+d)} \leq e.
    \end{equation}
    In the following, we consider $\gamma_0 = (m_{\ell_0}, h_{\ell_0})$ which belongs to $\Gamma_{\!n}$.
    From~\eqref{eq:m.l}, recall that
    \begin{equation*}
        m_{\ell_0} = \left\lfloor\frac{\log n}{2\ell_0}\right\rfloor.
    \end{equation*}
    Let us remark that if $\gamma = (m,h)\in \Gamma_{\!n}$ is such that $\gamma\preceq\gamma_0$, then $m\leq m_{\ell_0}\leq 2s+d$ for $n$ large enough.

    By Proposition~\ref*{prop:control.bias}, we have directly, for $n$ large enough,
    \begin{equation}\label{eq:bound.B.max}
        \begin{aligned}
            \mathbb{B}_{\gamma_0}
            &= \max_{\substack{\gamma\in \Gamma_{\!n} \\ \gamma \preceq \gamma_0}} \left|\EE\left\{\hat{f}_{\gamma}(t)\right\} - f(t)\right| \\
            &\leq \max_{\substack{\gamma\in \Gamma_{\!n} \\ \gamma \preceq \gamma_0}} \frac{W_h\sqrt{D_m}}{\lambda_{\gamma}} \times \mathfrak{L}_{m,s,L} \, h^{\beta_m(s)}
            \leq C^{\star}(s,d) \max_{\substack{\gamma\in \Gamma_{\!n} \\ \gamma \preceq \gamma_0}} \mathfrak{L}_{m,s,L} \, h^{\beta_m(s)},
        \end{aligned}
    \end{equation}
    where
    \begin{equation}\label{eq:notation.C.D.star}
        C^{\star}(s,d) = \frac{2^d \sqrt{D^{\star}(s,d)}}{\lambda_{\bullet}(s,d)}
        \quad \text{and} \quad
        D^{\star}(s,d) = \max_{0 \leq m' \leq 2s+d} D_{m'}.
    \end{equation}
    It remains to bound the maximum that appears on the right-hand side of~\eqref{eq:bound.B.max}.
    To achieve this, fix $\gamma\in \Gamma_{\!n}$ such that $\gamma\preceq\gamma_0$. By definition, there exists $\ell\in \mathcal{L}_n$ such that $\gamma = (m_{\ell}, h_{\ell})$.
    Assume first that $m_{\ell} + 1 < s$, so that $\beta_{m_{\ell}}(s)= m_{\ell}+1$. Then, using $h_{\ell} = \rho e^{-\ell} \leq e^{-\ell}$, we have
    \begin{equation}\label{eq:bound.B.max.case.1}
        \mathfrak{L}_{m_{\ell},s,L}h_{\ell}^{\beta_{m_{\ell}}(s)}
        \leq \mathfrak{L}_{s,L}^{\star} \exp\left(- \frac{\ell\log n}{2\ell}\right)
        = \frac{\mathfrak{L}_{s,L}^{\star}}{\sqrt{n}}.
    \end{equation}
    Similarly, if we assume $m_{\ell}+1 \geq s$, then using the upper bound in \eqref{eq:h.asymptotics}, we have
    \begin{equation}\label{eq:bound.B.max.case.2}
        \begin{aligned}
            \mathfrak{L}_{m_{\ell},s,L} h_{\ell}^{\beta_{m_{\ell}}(s)}
            &\leq \mathfrak{L}_{s,L}^{\star} h_{\ell}^s \leq \mathfrak{L}_{s,L}^{\star} h_{\ell_0}^s \\
            &\leq e^s \, \mathfrak{F}_{s,L}^{~s/(2s+d)} (\mathfrak{L}_{s,L}^{\star})^{d/(2s+d)} \left(\frac{\log n}{n}\right)^{s/(2s+d)}.
        \end{aligned}
    \end{equation}
    Therefore, by putting the estimates~\eqref{eq:bound.B.max.case.1} and~\eqref{eq:bound.B.max.case.2} back into~\eqref{eq:bound.B.max}, we have, for $n$ large enough,
    \begin{equation}\label{eq:bias.param.0}
        \mathbb{B}_{\gamma_0} \leq e^s \, C^{\star}(s,d) \, \mathfrak{F}_{s,L}^{~s/(2s+d)} (\mathfrak{L}_{s,L}^{\star})^{d/(2s+d)} \left(\frac{\log n}{n}\right)^{s/(2s+d)}.
    \end{equation}

    Now, let us bound $\mathbb{U}_{\gamma_0}^{\star}$. By combining~\eqref{eq:B.gamma.and.U.star.gamma} and~\eqref{eq:r.gamma.and.pen.gamma}, we can write
    \begin{equation*}
        \mathbb{U}_{\gamma_0}^{\star} = \sqrt{2v_{\gamma_0}^{\star}\mathrm{pen}(\gamma_0)} + c_{\gamma_0} \mathrm{pen}(\gamma_0),
        \quad \text{with} \quad
        \mathrm{pen}(\gamma_0) = d\delta\lvert\log h_{\ell_0}\rvert + 2\lvert\log(\lambda_{\gamma_0})\rvert.
    \end{equation*}
    Let us first control the terms $c_{\gamma_0}$ and $v_{\gamma_0}^{\star}$ defined~\eqref{eq:c.gamma} and~\eqref{eq:variance.bound}, respectively. Using the notation in~\eqref{eq:notation.C.D.star}, we have
    \begin{equation}\label{eq:four.bounds.c.v}
        \begin{aligned}
            c_{\gamma_0}
            &\leq \frac{\sqrt{D^{\star}(s,d)}}{\lambda_{\gamma_0}} \times \frac{1}{nh_{\ell_0}^d} \leq \frac{C^{\star}(s,d)}{2^d} \times \frac{1}{nh_{\ell_0}^d}, \\
            v_{\gamma_0}^{\star}
            &\leq \frac{2^d D^{\star}(s,d)}{\lambda_{\gamma_0}^2} \times \frac{\mathfrak{F}_{s,L}}{nh_{\ell_0}^d} \leq \frac{\{C^{\star}(s,d)\}^2 \mathfrak{F}_{s,L}}{2^d} \times \frac{1}{nh_{\ell_0}^d}.
        \end{aligned}
    \end{equation}
    Now, using the subadditivity of the function $x\mapsto \smash{\sqrt{2v_{\gamma_0}^{\star}x}} + c_{\gamma_0} x$, the four bounds in~\eqref{eq:four.bounds.c.v}, and the fact that the functions $\lambda \mapsto |\log \lambda| / \lambda^k$ ($k\in \{1,2\}$) are decreasing for $\lambda\in (0,1)$ and uniformly bounded by $1$ for $\lambda\geq 1$, we obtain
    \begin{equation*}
        \begin{aligned}
            \mathbb{U}_{\gamma_0}^{\star}
            &\leq \sqrt{2v_{\gamma_0}^{\star}d\delta\lvert\log h_{\ell_0}\rvert} + c_{\gamma_0} d\delta\lvert\log h_{\ell_0}\rvert + \sqrt{4v_{\gamma_0}^{\star} \lvert\log(\lambda_{\gamma_0})\rvert} + 2 c_{\gamma_0} \lvert\log(\lambda_{\gamma_0})\rvert \\
            &\leq \sqrt{\frac{\{C^{\star}(s,d)\}^2 \mathfrak{F}_{s,L} d\delta}{2^{d-1}} \times \frac{|\log h_{\ell_0}|}{n h_{\ell_0}^d}} + \frac{C^{\star}(s,d) d\delta}{2^d} \times \frac{|\log h_{\ell_0}|}{n h_{\ell_0}^d} \\
            &\quad+ \sqrt{\frac{2^{d+2} D^{\star}(s,d) \mathfrak{F}_{s,L}}{\min\left\{1, \frac{\lambda_{\bullet}^2(s,d)}{|\log \lambda_{\bullet}(s,d)}|\right\}} \times \frac{1}{n h_{\ell_0}^d}} + \frac{2 \sqrt{D^{\star}(s,d)}}{\min\left\{1, \frac{\lambda_{\bullet}(s,d)}{|\log \lambda_{\bullet}(s,d)}|\right\}} \times \frac{1}{n h_{\ell_0}^d}.
        \end{aligned}
    \end{equation*}
    As $n\to\infty$, note that $h_{\ell_0}\to 0$, $\smash{nh_{\ell_0}^d\to\infty}$ and $\smash{|\log h_{\ell_0}| / (nh_{\ell_0}^d)\to 0}$ by~\eqref{eq:h.asymptotics}.
    Therefore, the first term on the right-hand side of the last equation leads the asymptotic behavior of the upper bound on $\mathbb{U}_{\gamma_0}^{\star}$.
    In particular, for $n$ large enough, the second, third and fourth terms are all bounded from above by the first term, so that
    \begin{equation*}
        \mathbb{U}_{\gamma_0}^{\star} \leq 4 \, \sqrt{\frac{\{C^{\star}(s,d)\}^2 \mathfrak{F}_{s,L} d\delta}{2^{d-1}} \times \frac{|\log h_{\ell_0}|}{n h_{\ell_0}^d}} = \frac{C^{\star}(s,d) \sqrt{\mathfrak{F}_{s,L} d\delta}}{2^{(d-5)/2}} \times \left(\frac{|\log h_{\ell_0}|}{n h_{\ell_0}^d}\right)^{1/2}.
    \end{equation*}
    After straightforward calculations, this yields
    \begin{equation}\label{eq:U.star.param.0}
        \mathbb{U}_{\gamma_0}^{\star} \lesssim C^{\star}(s,d) \sqrt{d\delta} \, \, \mathfrak{F}_{s,L}^{~s/(2s+d)} (\mathfrak{L}_{s,L}^{\star})^{d/(2s+d)} \left(\frac{\log n}{n}\right)^{s/(2s+d)},
    \end{equation}
    where $\lesssim$ is independent of all parameters.

    By taking the upper bounds on $\mathcal{B}_{\gamma_0}$ and $\mathbb{U}_{\gamma_0}^{\star}$ in~\eqref{eq:bias.param.0} and~\eqref{eq:U.star.param.0} back into~\eqref{eq:oracle.adaptivity.simple.domains}, we get
    \begin{equation*}
        \pwrisk(\hat{f}, f) \lesssim \mathfrak{c}(\delta,\finfty) \, e^s \, C^{\star}(s,d) \sqrt{d\delta} \, \, \mathfrak{F}_{s,L}^{~s/(2s+d)} (\mathfrak{L}_{s,L}^{\star})^{d/(2s+d)} \left(\frac{\log n}{n}\right)^{s/(2s+d)},
    \end{equation*}
    where $\mathfrak{c}(\delta,\finfty) = 8 + 2 \{(\delta-1) / \finfty\}^{1/2}$ and $\lesssim$ is again independent of all parameters.
    This proves the upper bound.

    Next, let us prove that the admissible collection $\phi$ defined in~\eqref{eq:ARC.simple.geometry} is the ARC. To do so, we will use the result established in Section~5.6.1 of \cite{Rebelles2015supp}. In particular, it is sufficient to prove that for any $(s,L)\in \mathcal{K}$ and $(s',L')\in \mathcal{K}$ such that $s<s'$, there are two functions $f_0$ and $f_1$ that satisfy the following properties:
    \begin{enumerate}
        \item $f_0$ and $f_1$ are density functions that belong to $\Sigma(s',L')$ and $\Sigma(s,L)$, respectively.
        \item $|f_1(t) - f_0(t)| \asymp \phi_n(s,L) = (n^{-1} \log n)^{s/(2s+d)}$. (Here, $\asymp$ means both $\lesssim$ and $\gtrsim$ hold.)
        \item For any positive real $\tau$ such that  ${s}/{(2s+d)} < \tau < {s'}/{(2s'+d)}$, we have
            \begin{equation*}
                \EE_{f_0,n}\left\{\prod_{i=1}^n \frac{f_1(X_i)}{f_0(X_i)}\right\}^2\lesssim n^\tau \phi_n(s,L).
            \end{equation*}
    \end{enumerate}
    Similarly to the proof of Proposition~\ref*{prop:minimaxity.simple.domains}, let us define
    \begin{equation*}
        f_0(x) = \frac{\1_{\domain}(x)}{\mathrm{Leb}(\domain)}
        \quad \text{and} \quad
        f_1(x) = f_0(x) + A h^s \tilde{\psi}_h(x-t) \1_{\domain}(x),
    \end{equation*}
    where the quantities $\tilde{\psi}_h$ and $A$ have already been introduced in the proof of Proposition~\ref*{prop:minimaxity.simple.domains}, and $h$ is defined in a slightly different way by
    \begin{equation*}
        h = \left[\frac{1}{\frac{8 A^2}{\mathrm{Leb}(\domain)} \left\{\int_{\Delta} \psi^2(u) \rd u + 1\right\}} \times \frac{\log n^\varrho}{n}\right]^{1/(2s+d)}, \quad 0<\varrho < \frac{s'}{2s'+d} - \frac{s}{2s+d}.
    \end{equation*}
    It is straightforward to verify the first and second points by employing arguments that are virtually identical to those presented in the proof of Proposition~\ref*{prop:minimaxity.simple.domains}. For the third point, we have, by the same line of reasoning that led us to~\eqref{eq:end.proof.Prop.4.5},
    \begin{equation}\label{eq:bound.second.moment.adaptive.simple.domains}
        \begin{aligned}
            \EE_{f_0,n}\left\{\prod_{i=1}^n \frac{f_1(X_i)}{f_0(X_i)}\right\}^2
            &\leq \exp\left[\frac{8 A^2}{\mathrm{Leb}(\domain)} \left\{\int_{\Delta} \psi^2(u) \rd u + 1\right\} n h^{2s+d}\right] \\
            &\leq n^\varrho \leq n^{\varrho+s/(2s+d)}\left(\frac{\log n}{n}\right)^{s/(2s+d)} \leq n^{\tau} \phi_n(s,L),
        \end{aligned}
    \end{equation}
    where $\tau = \varrho+s/(2s+d)$. This shows that three conditions above, or equivalently the conditions of Lemma~4 of \cite{Rebelles2015supp}, are satisfied with $a_n^{-1} \asymp \phi_n(s,L)$ and $b_n = n^{\tau}$, which implies that the collection $\phi$ defined in~\eqref{eq:ARC.simple.geometry} is the ARC.
\end{proof}

\phantomsection
\addcontentsline{toc}{subsection}{Proof of Proposition~\ref*{prop:minimaxity.polynomial.sectors}}

\begin{proof}[Proof of Proposition~\ref*{prop:minimaxity.polynomial.sectors}]
    Let $\rho = 1$ so that $h\in (0,1]$.
    Also, throughout the proof, let $\gamma_k=(m,h)$ with $m=\llfloor s_k\rrfloor$.
    Recall the definition of $w_h$ from~\eqref{eq:w.h} and apply it to the domain $\domain_k$ in~\eqref{eq:domain.k} to obtain
    \begin{equation}\label{eq:W.h.calculation}
        W_h
        =\int_{\R^d} w_h(u) \rd u
        = h^{-2}\int_{0}^h\int_0^{x^k} 1 \rd y \rd x \\
        = \frac{h^{k-1}}{k+1}.
    \end{equation}
    Recall the definition of the Gram matrix $\mathcal{B}_{\gamma_k}$ from~\eqref{eq:matrix.B.gamma}.
    By the simple change of variables $(\xi,\eta) = (x/h,y/h)$ compared with the previous integral, we have, for any $v\in \R^{D_{m}}$,
    \begin{equation}\label{eq:v.top.B.v.decomp}
        \begin{aligned}
            v^{\top} \mathcal{B}_{\gamma_k} v
            &=\int_0^1\int_0^{h^{k-1}\xi^k} \left\{v^{\top} \Phi_{m,1}(\xi,\eta)\right\}^2 \rd \eta\rd \xi \\
            &= \sum_{|\alpha| \leq m} \sum_{|\beta| \leq m} v_{\alpha} v_{\beta} \int_0^1\int_0^{h^{k-1}\xi^k} \xi^{\alpha_1+\beta_1}\eta^{\alpha_2+\beta_2} \rd \eta\rd \xi \\
            &= \sum_{|\alpha| \leq m} \sum_{|\beta| \leq m} v_{\alpha} v_{\beta} \int_0^1 \xi^{\alpha_1+\beta_1}
            \frac{(h^{k-1}\xi^k)^{{\alpha_2+\beta_2+1}}}{\alpha_2+\beta_2+1}\rd \xi \\
            &= \sum_{|\alpha| \leq m} \sum_{|\beta| \leq m} c_k(\alpha,\beta)v_{\alpha} v_{\beta} h^{(k-1)(\alpha_2+\beta_2+1)},
        \end{aligned}
    \end{equation}
    where
    \begin{equation*}
        c_k(\alpha,\beta) = \left[(\alpha_2+\beta_2+1)\{\alpha_1+\beta_1 + k(\alpha_2+\beta_2+1) + 1\}\right]^{-1} > 0.
    \end{equation*}
    Denote $A_{\alpha,\beta}(v) = c_k(\alpha,\beta)v_{\alpha} v_{\beta}$ and $\alpha^{\star} = \beta^{\star} = (0, m)$, where the dependence in $k$ is omitted. By~\eqref{eq:v.top.B.v.decomp}, we can write, for all $v\in \R^{D_{m}}\setminus\{(0,\ldots,0)\}$,
    \begin{equation*}
        v^{\top} \mathcal{B}_{\gamma_k} v
        = \sum_{|\alpha| \leq m} \sum_{|\beta| \leq m} A_{\alpha,\beta}(v)  h^{(k-1)(\alpha_2+\beta_2+1)}
        = h^{(k-1)(2m+1)} \psi_v(1/h),
    \end{equation*}
    where $\psi_v$ is a polynomial of degree $(k-1)(2m)$ which satisfies $\psi_v(0) = A_{\alpha^{\star}, \beta^{\star}}(v) > 0$. In particular, for any $v\in \R^{D_{m}}$ such that $v^{\top} v=1$, we have $v^{\top} \mathcal{B}_{\gamma_k} v > 0$ and thus $\psi_v(x) > 0$ for any $x\in [0,\infty)$. Therefore, $\Psi_{\star}(v) = \min_{x\in [0,\infty)} \psi_v(x) > 0$ exists and is a continuous function of the coefficients of $\psi_v$, i.e., is a continuous function of $v$. Since $\mathbb{S}^1=\{v\in \R^{D_{m}}: v^{\top} v=1\}$ is a compact set, $\Psi_{\star} = \min_{v\in \mathbb{S}^1} \Psi_{\star}(v)>0$ exists. Finally, we deduce
    \begin{equation}\label{eq:minoration.lambda}
        \lambda_{\gamma_k} = \min_{v\in \mathbb{S}^1} v^{\top} \mathcal{B}_{\gamma_k} v \geq \Psi_{\star} h^{(k-1)(2m+1)} > 0.
    \end{equation}
    Since we assume $f\in \Sigma(s,L)$, a straightforward consequence of Definition~\ref*{def:Sigma.s.L} is that for all $\beta\in (0,s]$, there exists $\tilde{L}_{s,L,\beta} > 0$ such that $f\in \Sigma(\beta,\tilde{L}_{s,L,\beta})$. Indeed, if $f\in \Sigma(s,L)$, then there exists a polynomial $q = \sum_{|\alpha|\leq \llfloor s \rrfloor} q_{\alpha} \varphi_{\alpha}(\cdot)\in \mathcal{P}_{\llfloor s \rrfloor}$ such that, for any $u\in \voisin(1)$,
    \begin{equation*}
        \left|f(t+u) - \sum_{|\alpha|\leq \llfloor s \rrfloor} q_{\alpha} \varphi_{\alpha}(u)\right| \leq L\|u\|_{\infty}^s.
    \end{equation*}
    This implies that $f\in \Sigma(\beta,\tilde{L}_{s,L,\beta})$ since
    \begin{equation*}
        \begin{aligned}
            \left|f(t+u) - \sum_{|\alpha|\leq \llfloor \beta \rrfloor} q_{\alpha} \varphi_{\alpha}(u)\right|
            &\leq L\|u\|_{\infty}^s + \sum_{\llfloor \beta \rrfloor < |\alpha|\leq \llfloor s \rrfloor} |q_{\alpha}| |\varphi_{\alpha}(u)| \\
            &\leq \left(L + \sum_{\llfloor \beta \rrfloor < |\alpha|\leq \llfloor s \rrfloor} |q_{\alpha}|\right) \|u\|_{\infty}^{\beta}
            \eqqcolon \tilde{L}_{s,L,\beta} \|u\|_{\infty}^{\beta},
        \end{aligned}
    \end{equation*}
    where the last inequality follows from the fact that, for all $x\in [0,1]$, the map $\beta\mapsto x^{\beta}$ is decreasing on $(0,\infty)$ (notice that $u\in \voisin(1)\subseteq [-1,1]^d$ implies $x = \|u\|_{\infty}\in [0,1]$).
    In particular, for our choice of smoothness $\beta = s_k$ in~\eqref{eq:s.k.and.theta.k}, we have $f\in \Sigma(s_k,L_k)$, where we write $L_k = \tilde{L}_{s,L,s_k}$ for short.
    Therefore, by Proposition~\ref*{prop:control.bias}, Proposition~\ref*{prop:control.stochastic.term} and Lemma~\ref{lem:infty.norm.f}, together with the fact that $W_h = h^{k-1}/(k+1)$ from~\eqref{eq:W.h.calculation} and $\lambda_{\gamma_k} \geq \Psi_{\star} h^{(k-1)(2m+1)} > 0$ from~\eqref{eq:minoration.lambda}, we have
    \begin{equation*}
        \begin{aligned}
            \pwrisk(\hat{f}_{\gamma_k}, f)
            &\leq \frac{W_h \sqrt{D_{\llfloor s_k\rrfloor}}}{\lambda_{\gamma_k}} \times \left\{\mathfrak{L}_{\llfloor s_k\rrfloor,s_k,L_k}^2 \, h^{2s_k} + \frac{\mathfrak{F}_{s_k,L_k}}{W_h nh^2}\right\}^{1/2} \\
            &\lesssim_{s,L,k} h^{-2\llfloor s_k\rrfloor (k-1)} \times \left\{\mathfrak{L}_{\llfloor s_k\rrfloor,s_k,L_k}^2 \, h^{2s_k} + \frac{(k+1)\mathfrak{F}_{s_k,L_k}}{nh^{k+1}}\right\}^{1/2}.
        \end{aligned}
    \end{equation*}
    By taking $h=n^{-1/(2s_k+k+1)}$, the upper bound~\eqref{eq:rebroussement.upper.bound} follows.
    To be more transparent, our choice of smoothness $s_k\in (0,s]$ in~\eqref{eq:s.k.and.theta.k} is the one that minimizes the bound in the last equation if we were to replace $s_k$ by a general smoothness parameter $\beta\in (0,s]$.

    To prove the lower bound, we follow the scheme of proofs laid out in Sections~2.2~and~2.3 of \cite{Tsybakov2009supp}.
    Let $(s,L)\in (0,1] \times (0,\infty)$ and $k>1$ be given. Define the auxiliary function
    \begin{equation*}
        \tilde{\psi}(x) =  \exp\left(-\frac{1}{1-x^2}\right) \1_{(-1,1)} (x),
        \quad x\in \R,
    \end{equation*}
    together with the positive real constant
    \begin{equation*}
        b_{k} = \frac{\int_0^{1} x^k(1-x)\rd x}{\int_0^1 \tilde{\psi}(4x-3)x^k(1-x) \rd x}.
    \end{equation*}
    Now, consider the function
    \begin{equation*}
        \psi_k(x) = \{1 - b_{k} \tilde{\psi}(4x-3)\}(1-x) \1_{[0,1]}(x).
    \end{equation*}
    This is a bounded function which satisfies $\psi_k(0)=1$, $\psi_k(1)=0$ and $\smash{\int_0^1 x^k \psi_k(x) \rd x = 0}$.
    Moreover, $\psi_k$ is $s$-H\"older-continuous on $[0,1]$ for some positive Lipschitz constant $L_{k,s} > 0$.
    We now consider, for $u=(u_x,u_y)\in \domain_k$,
    \begin{equation}\label{eq:fct-LB}
        f_0(u) = \frac{\1_{\domain_k}(u)}{\mathrm{Leb}(\domain_k)}
        \quad \text{and} \quad
        f_1(u) = f_0(u) + 2\rho_n \psi_k\left(\frac{u_x}{h_n}\right) \1_{\voisin(h_n)}(u),
    \end{equation}
    where
    \begin{equation}\label{eq:def.h.n.rho.n}
        h_n = n^{-1/(2s+k+1)}
        \quad \text{and} \quad
        \rho_n = \{L / (2 L_{k,s})\} h_n^s.
    \end{equation}
    Observe that both $f_0$ and $f_1$ are density functions on $\domain_k$, for $n$ large enough. This is obvious for the function $f_0$. Let us prove this assertion for $f_1$. First, note that for $u\in \domain_k$,
    \begin{equation*}
        f_1(u) \geq \frac{1}{\mathrm{Leb}(\domain_k)} - 2\rho_n \|\psi_k\|_{\infty},
    \end{equation*}
    which is positive for large $n$.
    Using the change of variable $\xi = x / h_n$, it can also be seen that
    \begin{equation*}
        \begin{aligned}
            \int_{\domain_k} f_1(u) \rd u
            &= 1 + 2 \rho_n \int_0^{h_n} \psi_k\left(\frac{x}{h_n}\right) \left(\int_0^{x^k} 1 \rd y\right) \rd x \\[-1mm]
            &= 1 + 2 h_n^{k+1} \rho_n \int_0^1 \xi^k \psi_k(\xi) \rd \xi \\
            &= 1.
        \end{aligned}
    \end{equation*}

    Using the methodology developed in Section~2.5 of \cite{Tsybakov2009supp}, the proof of the lower bound~\eqref{eq:rebroussement.lower.bound} boils down to proving the following assertions:
    \begin{enumerate}
        \item $\{f_0, f_1\}\subseteq \Sigma(s,L)$.
        \item $|f_1(0) - f_0(0)| \geq 2\rho_n$.
        \item $n D_{\mathrm{KL}}(f_1, f_0)\leq A$ for some positive real constant $A = A(k,s,L) > 0$, where $D_{\mathrm{KL}}(\cdot,\cdot)$ denotes the Kullback-Leibler divergence between the associated probability measures.
    \end{enumerate}
    To prove the first assertion, notice that the function
    \begin{equation*}
        u = (u_x, u_y)\in \domain_k \mapsto 2\rho_n \psi_k\left(\frac{u_x}{h_n}\right) \1_{\voisin(h_n)}(u)
    \end{equation*}
    belongs to $\Sigma(s,L)$, using the definitions of $h_n$ and $\rho_n$ in~\eqref{eq:def.h.n.rho.n} and the aforementioned fact that $\psi_k$ is $s$-H\"older-continuous on $[0,1]$ for some Lipschitz constant $L_{k,s}>0$.
    The second assertion is satisfied because of~\eqref{eq:fct-LB} and the fact that $\psi_k(0) = 1$.
    It remains to prove the third assertion.
    Since $f_1(u)/f_0(u) = 1$ for all $u\in \domain_k \setminus \voisin(h_n)$, observe that
    \begin{equation}\label{eq:n.DKL.A.B.bound}
        n D_{\mathrm{KL}}(f_1, f_0)
        = n\int_{\voisin(h_n)}f_1(u) \log \left\{\frac{f_1(u)}{f_0(u)}\right\} \rd u
        = B_1 + B_2,
    \end{equation}
    where
    \begin{equation*}
        \begin{aligned}
            B_1 &= \frac{n}{\mathrm{Leb}(\domain_k)}\int_{\voisin(h_n)}\log \left\{\frac{f_1(u)}{f_0(u)}\right\} \rd u, \\
            B_2 &= 2n\rho_n\int_{\voisin(h_n)}\psi_k\left(\frac{u_x}{h_n}\right) \log \left\{\frac{f_1(u)}{f_0(u)}\right\} \rd u.
        \end{aligned}
    \end{equation*}
    Since $\psi_k$ is bounded and $\rho_n$ tends to zero as $n\to \infty$, let $n$ be large enough so that
    \begin{equation*}
        \left|2\rho_n\mathrm{Leb}(\domain_k)\psi_k\left(\frac{u_x}{h_n}\right)\right|\leq \frac{1}{2}.
    \end{equation*}
    Now, using
    \begin{equation*}
       \int_{\voisin(h_n)} \psi_k\left(\frac{u_x}{h_n}\right) \rd u = 0
       \qquad \text{and} \qquad
       |\log(1+x) - x|\leq 2 x^2 \quad\text{for } |x| \leq 1/2,
    \end{equation*}
    we have
    \begin{equation}\label{eq:bound.B1}
        \begin{aligned}
            \left|B_1 \right|
            &\leq \frac{n}{\mathrm{Leb}(\domain_k)} \int_{\voisin(h_n)} \left|\log\left\{\frac{f_1(u)}{f_0(u)}\right\}
            - 2\rho_n \mathrm{Leb}(\domain_k) \psi_k\left(\frac{u_x}{h_n}\right)\right| \rd u \\
            &\leq 8 n \rho_n^2 \mathrm{Leb}(\domain_k) \int_{\voisin(h_n)} \psi_k^2\left(\frac{u_x}{h_n}\right) \rd u \\
            &\leq 8 n \rho_n^2 h_n^{k+1} \mathrm{Leb}(\domain_k) \int_0^1 \xi^k \psi_k^2(\xi) \rd \xi.
        \end{aligned}
    \end{equation}
    Similarly, but instead using
    \begin{equation*}
        |\log(1+x)|\leq |x| \quad \text{for } x > 0,
    \end{equation*}
    we have
    \begin{equation}\label{eq:bound.B2}
        \begin{aligned}
            |B_2|
            &\leq 2 n \rho_n \int_{\voisin(h_n)} \psi_k\left(\frac{u_x}{h_n}\right) \left|\log\left\{\frac{f_1(u)}{f_0(u)}\right\}\right| \rd u \\
            &\leq 4 n \rho_n^2 \mathrm{Leb}(\domain_k) \int_{\voisin(h_n)} \psi_k^2\left(\frac{u_x}{h_n}\right) \rd u \\
            &\leq 4 n \rho_n^2 h_n^{k+1} \mathrm{Leb}(\domain_k) \int_0^1 \xi^k \psi_k^2(\xi) \rd \xi.
        \end{aligned}
    \end{equation}
    Taking the above bounds~\eqref{eq:bound.B1} and~\eqref{eq:bound.B2} together in~\eqref{eq:n.DKL.A.B.bound}, we obtain
    \begin{equation*}
        n D_{\mathrm{KL}}(f_1, f_0) \leq 12 \, \mathrm{Leb}(\domain_k) \left\{\int_0^1 \xi^k \psi_k^2(\xi) \rd \xi\right\} n\rho_n^2h_n^{k+1}.
    \end{equation*}
    Using the definition of $\rho_n$ and $h_n$ in~\eqref{eq:def.h.n.rho.n}, the right-hand side of the last equation is easily seen to be bounded, which proves the third point. Moreover, we have $\rho_n \asymp n^{-s/(2s+k+1)} = n^{-\theta_k(s)}$ since $s_k = s$ for $s\in (0,1]$; see Item~1 of Remark~\ref*{rem:other.complicated.domains}. The lower bound \eqref{eq:rebroussement.lower.bound} follows.
\end{proof}

\phantomsection
\addcontentsline{toc}{subsection}{Proof of Theorem~\ref*{thm:adaptivity.polynomial.sectors}}

\begin{proof}[Proof of Theorem~\ref*{thm:adaptivity.polynomial.sectors}]
    Let $(s,L)\in (0,1] \times (0,\infty)$ and $k>1$ be given. Using~\eqref{eq:minoration.lambda} with the polynomial degree set to $m=0$, there exists a positive real constant $\Psi_{\star} = \Psi_{\star}(k)$ which depends on $k$ such that, for any $\gamma=(0,h)\in \Gamma_n$,
    \begin{equation}\label{eq:info}
        \lambda_{\gamma}\ge \Psi_{\star}h^{k-1} > 0.
    \end{equation}
     Hence, condition~\eqref{eq:domain.lambda} is satisfied and Theorem~\ref*{thm:oracle.result} can be applied to obtain \begin{equation}\label{eq:oracle.adaptivity.polynomial.sectors}
        \pwrisk(\hat{f}, f)
        \leq \min_{\gamma\in \Gamma_{\!n}} \left[5 \mathbb{B}_{\gamma} + \left\{3 + 2 \sqrt{(\delta-1) / \finfty}\right\} \mathbb{U}_{\gamma}^{\star}\right] + \mathcal{O}\left\{\sqrt\frac{(\log n)^d}{n}\right\}.
    \end{equation}

    Now, similarly to the proof of Theorem~\ref*{thm:adaptivity.simple.domains}, we define
    \begin{equation*}
        \ell_0 = \left\lfloor \frac{1}{2s+k+1} \log\left(\frac{n}{\log n}\right) + \log \rho \right\rfloor
        \quad \text{and} \quad
        m_{\ell_0} = 0.
    \end{equation*}
    In the following, consider $n$ large enough so that $\ell_0\geq 1$.
    This choice of $\ell_0$ implies
    \begin{equation*}
        1 \leq h_{\ell_0}\left(\frac{\log n}{n}\right)^{-1/(2s+k+1)} \leq e.
    \end{equation*}
    By \eqref{eq:W.h.calculation}, we know that
    \begin{equation}\label{eq:bound.W.h.over.lambda.gamma}
        \frac{W_h}{\lambda_{\gamma}} = \frac{h^{k-1} / (k+1)}{\lambda_{\gamma}} \leq \frac{1}{(k+1) \Psi_{\star}}.
    \end{equation}
    Therefore, with the notation $\gamma_0=(0,h_{\ell_0})$, Proposition~\ref*{prop:control.bias} yields
    \begin{equation}\label{eq:bias.rebrou}
        \begin{aligned}
            \mathbb{B}_{\gamma_0}
            &= \max_{\substack{\gamma\in \Gamma_{\!n} \\ \gamma \preceq \gamma_0}} \left|\EE\left\{\hat{f}_{\gamma}(t)\right\} - f(t)\right|
            \leq \max_{\substack{\gamma\in \Gamma_{\!n} \\ \gamma \preceq \gamma_0}} \frac{W_h\sqrt{D_0}}{\lambda_{\gamma}} \times \mathfrak{L}_{0,s,L} \, h^s \\
            &\leq C_0(s,L,k) \max_{h\leq h_{\ell_0}} h^s
            \leq C_0(s,L,k) \, e^s \, \left(\frac{\log n}{n}\right)^{s/(2s+k+1)},
        \end{aligned}
    \end{equation}
    where $C_0(s,L,k) > 0$ is a positive real constant that depends on $s$, $L$ and $k$. Following the same line of argument as in the proof of Theorem~\ref*{thm:adaptivity.simple.domains}, we obtain, using~\eqref{eq:info}, \eqref{eq:bound.W.h.over.lambda.gamma} and Lemma~\ref{lem:infty.norm.f},
    \begin{equation}\label{eq:var.rebrou}
        \begin{aligned}
            \mathbb{U}_{\gamma_0}^{\star}
            &\leq C_1(s,L,k) \left(\sqrt{\frac{|\log h_{\ell_0}|}{nh_{\ell_0}^{k+1}}}+ \frac{|\log h_{\ell_0}|}{nh_{\ell_0}^{k+1}} + \sqrt{\frac{1}{nh_{\ell_0}^{k+1}}}+ \frac{1}{nh_{\ell_0}^{k+1}}\right) \\
            &\leq C_2(s,L,k) \left(\frac{\log n}{n}\right)^{s/(2s+k+1)},
        \end{aligned}
    \end{equation}
    where $C_1(s,L,k)$ and $C_2(s,L,k)$ are positive real constants that depend on $s$, $L$ and $k$.
    By plugging the bounds~\eqref{eq:bias.rebrou} and~\eqref{eq:var.rebrou} back into \eqref{eq:oracle.adaptivity.polynomial.sectors}, we get the upper bound~\eqref{eq:rate.rebrou.adapt}.

    To prove that the collection $\phi = \{\phi_n(s,L) : (s,L)\in (0,1]\times(0,\infty), n\in \N\}$ is the ARC, we follow the same line of argument as in the proof of the adaptive lower bound in Theorem~\ref*{thm:adaptivity.simple.domains}, but the density functions $f_0\in \Sigma(s',L')$ and $f_1\in \Sigma(s,L)$ are replaced by those defined in~\eqref{eq:fct-LB} together with the bandwidth
    \begin{equation*}
        h = \left\{\frac{1}{\{L / (2 L_{k,s})\}^2 \mathrm{Leb}(\domain_k) \int_0^1 \xi^k \psi_k^2(\xi) \rd \xi} \times \frac{\log n^\varrho}{n}\right\}^{1/(2s+k+1)},
    \end{equation*}
    where $0 < \varrho < s'/(2s'+k+1) - s/(2s+k+1)$. To understand the choice of $h$ here, simply look at \eqref{eq:I.develop} with $A = L / (2 L_{k,s})$ and use the inequality $(1+x) \leq e^x$, which is valid for $x > 0$, to obtain a bound analogous to \eqref{eq:bound.second.moment.adaptive.simple.domains}. This concludes the proof.
\end{proof}

\begin{acks}[Acknowledgments]
The authors express their sincere gratitude to the anonymous reviewers for their constructive comments. The quality and depth of the questions posed by the referees have significantly enhanced the caliber of this paper. It has been possible to conduct this research in part thanks to the support provided by Calcul Qu\'ebec and the Digital Research Alliance of Canada.
\end{acks}

\begin{funding}
K.\ Bertin is supported by FONDECYT regular grants 1221373 and 1230807 from ANID-Chile, and the Centro de Modelamiento Matemático (CMM) BASAL fund FB210005 for centers of excellence from ANID-Chile. N.\ Klutchnikoff acknowledges support from FONDECYT regular grant 1221373 from ANID-Chile. F.\ Ouimet was supported by a CRM-Simons postdoctoral fellowship from the Centre de recherches math\'ematiques (Montr\'eal, Canada) and the Simons Foundation. F.\ Ouimet is currently funded through a contribution to Christian Genest's research program from the Trottier Institute for Science and Public Policy.
\end{funding}

\bibliographystyle{imsart-nameyear}
\bibliography{bib}

\begin{thebibliography}{84}

\bibitem[\protect\citeauthoryear{Aitchison and
  Lauder}{1985}]{AitchisonLauder1985}
\begin{barticle}[author]
\bauthor{\bsnm{Aitchison},~\bfnm{J.}\binits{J.}} \AND
  \bauthor{\bsnm{Lauder},~\bfnm{I.~J.}\binits{I.~J.}}
(\byear{1985}).
\btitle{Kernel density estimation for compositional data}.
\bjournal{J. Roy. Statist. Soc. Ser. C}
\bvolume{34}
\bpages{129--137}.
\bdoi{10.2307/2347365}
\end{barticle}
\endbibitem

\bibitem[\protect\citeauthoryear{Ammous, Dedecker and
  Duval}{2024}]{AmmousDedeckerDuval2024}
\begin{barticle}[author]
\bauthor{\bsnm{Ammous},~\bfnm{S.}\binits{S.}},
  \bauthor{\bsnm{Dedecker},~\bfnm{J.}\binits{J.}} \AND
  \bauthor{\bsnm{Duval},~\bfnm{C.}\binits{C.}}
(\byear{2024}).
\btitle{Adaptive directional estimator of the density in {$\mathbb{R}^d$} for
  independent and mixing sequences}.
\bjournal{J. Multivariate Anal.}
\bvolume{203}
\bpages{Paper No. 105332, 20 pp.}
\bdoi{10.1016/j.jmva.2024.105332}
\bmrnumber{4756038}
\end{barticle}
\endbibitem

\bibitem[\protect\citeauthoryear{Arnone et~al.}{2022}]{Arnone_et_al2022}
\begin{barticle}[author]
\bauthor{\bsnm{Arnone},~\bfnm{E.}\binits{E.}},
  \bauthor{\bsnm{Ferraccioli},~\bfnm{F.}\binits{F.}},
  \bauthor{\bsnm{Pigolotti},~\bfnm{C.}\binits{C.}} \AND
  \bauthor{\bsnm{Sangalli},~\bfnm{L.~M.}\binits{L.~M.}}
(\byear{2022}).
\btitle{A roughness penalty approach to estimate densities over two-dimensional
  manifolds}.
\bjournal{Comput. Statist. Data Anal.}
\bvolume{174}
\bpages{Paper No. 107527, 14 pp.}
\bdoi{10.1016/j.csda.2022.107527}
\bmrnumber{4437662}
\end{barticle}
\endbibitem

\bibitem[\protect\citeauthoryear{Azzimonti et~al.}{2014}]{Azzimonti_et_al2014}
\begin{barticle}[author]
\bauthor{\bsnm{Azzimonti},~\bfnm{L.}\binits{L.}},
  \bauthor{\bsnm{Nobile},~\bfnm{F.}\binits{F.}},
  \bauthor{\bsnm{Sangalli},~\bfnm{L.~M.}\binits{L.~M.}} \AND
  \bauthor{\bsnm{Secchi},~\bfnm{P.}\binits{P.}}
(\byear{2014}).
\btitle{Mixed finite elements for spatial regression with {PDE} penalization}.
\bjournal{SIAM/ASA J. Uncertain. Quantif.}
\bvolume{2}
\bpages{305--335}.
\bdoi{10.1137/130925426}
\bmrnumber{3283911}
\end{barticle}
\endbibitem

\bibitem[\protect\citeauthoryear{Azzimonti et~al.}{2015}]{Azzimonti_et_al2015}
\begin{barticle}[author]
\bauthor{\bsnm{Azzimonti},~\bfnm{L.}\binits{L.}},
  \bauthor{\bsnm{Sangalli},~\bfnm{L.~M.}\binits{L.~M.}},
  \bauthor{\bsnm{Secchi},~\bfnm{P.}\binits{P.}},
  \bauthor{\bsnm{Domanin},~\bfnm{M.}\binits{M.}} \AND
  \bauthor{\bsnm{Nobile},~\bfnm{F.}\binits{F.}}
(\byear{2015}).
\btitle{Blood flow velocity field estimation via spatial regression with {PDE}
  penalization}.
\bjournal{J. Amer. Statist. Assoc.}
\bvolume{110}
\bpages{1057--1071}.
\bdoi{10.1080/01621459.2014.946036}
\bmrnumber{3420684}
\end{barticle}
\endbibitem

\bibitem[\protect\citeauthoryear{Babu and Chaubey}{2006}]{BabuChaubey2006}
\begin{barticle}[author]
\bauthor{\bsnm{Babu},~\bfnm{G.~J.}\binits{G.~J.}} \AND
  \bauthor{\bsnm{Chaubey},~\bfnm{Y.~P.}\binits{Y.~P.}}
(\byear{2006}).
\btitle{Smooth estimation of a distribution and density function on a hypercube
  using {B}ernstein polynomials for dependent random vectors}.
\bjournal{Statist. Probab. Lett.}
\bvolume{76}
\bpages{959--969}.
\bdoi{10.1016/j.spl.2005.10.031}
\bmrnumber{2270097}
\end{barticle}
\endbibitem

\bibitem[\protect\citeauthoryear{Bakka et~al.}{2019}]{Bakka_et_al2019}
\begin{barticle}[author]
\bauthor{\bsnm{Bakka},~\bfnm{H.}\binits{H.}},
  \bauthor{\bsnm{Vanhatalo},~\bfnm{J.}\binits{J.}},
  \bauthor{\bsnm{Illian},~\bfnm{J.~B.}\binits{J.~B.}},
  \bauthor{\bsnm{Simpson},~\bfnm{D.}\binits{D.}} \AND
  \bauthor{\bsnm{Rue},~\bfnm{H.}\binits{H.}}
(\byear{2019}).
\btitle{Non-stationary {G}aussian models with physical barriers}.
\bjournal{Spat. Stat.}
\bvolume{29}
\bpages{268--288}.
\bdoi{10.1016/j.spasta.2019.01.002}
\bmrnumber{3903698}
\end{barticle}
\endbibitem

\bibitem[\protect\citeauthoryear{Barry}{2021}]{Barry2021}
\begin{bmisc}[author]
\bauthor{\bsnm{Barry},~\bfnm{R.}\binits{R.}}
(\byear{2021}).
\btitle{latticeDensity: {D}ensity {E}stimation and {N}onparametric {R}egression
  on {I}rregular {R}egions.}
\bnote{\texttt{R} package version 1.2.6, available online at
  \href{https://cran.r-project.org/src/contrib/Archive/latticeDensity}{https://cran.r-project.org/src/contrib/Archive/latticeDensity}
  [This link was accessed on 2024--08--21.]}.
\end{bmisc}
\endbibitem

\bibitem[\protect\citeauthoryear{Barry and McIntyre}{2011}]{BarryMcIntyre2011}
\begin{barticle}[author]
\bauthor{\bsnm{Barry},~\bfnm{R.~P.}\binits{R.~P.}} \AND
  \bauthor{\bsnm{McIntyre},~\bfnm{J.}\binits{J.}}
(\byear{2011}).
\btitle{Estimating animal densities and home range in regions with irregular
  boundaries and holes: {A} lattice-based alternative to the kernel density
  estimator}.
\bjournal{Ecol. Model.}
\bvolume{222}
\bpages{1666--1672}.
\bdoi{10.1016/j.ecolmodel.2011.02.016}
\end{barticle}
\endbibitem

\bibitem[\protect\citeauthoryear{Barry and McIntyre}{2020}]{BarryMcIntyre2020}
\begin{barticle}[author]
\bauthor{\bsnm{Barry},~\bfnm{R.~P.}\binits{R.~P.}} \AND
  \bauthor{\bsnm{McIntyre},~\bfnm{J.}\binits{J.}}
(\byear{2020}).
\btitle{Lattice-based methods for regression and density estimation on
  complicated multidimensional regions}.
\bjournal{Environ. Ecol. Stat.}
\bvolume{27}
\bpages{571--589}.
\bdoi{10.1007/s10651-020-00459-z}
\end{barticle}
\endbibitem

\bibitem[\protect\citeauthoryear{Bertin, El~Kolei and
  Klutchnikoff}{2019}]{BertinElKoleiKlutchnikoff2019}
\begin{barticle}[author]
\bauthor{\bsnm{Bertin},~\bfnm{K.}\binits{K.}},
  \bauthor{\bsnm{El~Kolei},~\bfnm{S.}\binits{S.}} \AND
  \bauthor{\bsnm{Klutchnikoff},~\bfnm{N.}\binits{N.}}
(\byear{2019}).
\btitle{Adaptive density estimation on bounded domains}.
\bjournal{Ann. Inst. Henri Poincar\'{e} Probab. Stat.}
\bvolume{55}
\bpages{1916--1947}.
\bdoi{10.1214/18-AIHP938}
\bmrnumber{4029144}
\end{barticle}
\endbibitem

\bibitem[\protect\citeauthoryear{Bertin and
  Klutchnikoff}{2014}]{BertinKlutchnikoff2014}
\begin{barticle}[author]
\bauthor{\bsnm{Bertin},~\bfnm{K.}\binits{K.}} \AND
  \bauthor{\bsnm{Klutchnikoff},~\bfnm{N.}\binits{N.}}
(\byear{2014}).
\btitle{Adaptive estimation of a density function using beta kernels}.
\bjournal{ESAIM Probab. Stat.}
\bvolume{18}
\bpages{400--417}.
\bdoi{10.1051/ps/2014010}
\bmrnumber{3333996}
\end{barticle}
\endbibitem

\bibitem[\protect\citeauthoryear{Bertin and
  Klutchnikoff}{2017a}]{BertinKlutchnikoff2017}
\begin{barticle}[author]
\bauthor{\bsnm{Bertin},~\bfnm{K.}\binits{K.}} \AND
  \bauthor{\bsnm{Klutchnikoff},~\bfnm{N.}\binits{N.}}
(\byear{2017}a).
\btitle{Pointwise adaptive estimation of the marginal density of a weakly
  dependent process}.
\bjournal{J. Statist. Plann. Inference}
\bvolume{187}
\bpages{115--129}.
\bdoi{10.1016/j.jspi.2017.03.003}
\bmrnumber{3638047}
\end{barticle}
\endbibitem

\bibitem[\protect\citeauthoryear{Bertin and
  Klutchnikoff}{2017b}]{BertinKlutchnikoff2017supp}
\begin{barticle}[author]
\bauthor{\bsnm{Bertin},~\bfnm{K.}\binits{K.}} \AND
  \bauthor{\bsnm{Klutchnikoff},~\bfnm{N.}\binits{N.}}
(\byear{2017}b).
\btitle{Pointwise adaptive estimation of the marginal density of a weakly
  dependent process}.
\bjournal{J. Statist. Plann. Inference}
\bvolume{187}
\bpages{115--129}.
\bdoi{10.1016/j.jspi.2017.03.003}
\bmrnumber{3638047}
\end{barticle}
\endbibitem

\bibitem[\protect\citeauthoryear{Bertin, Klutchnikoff and
  Ouimet}{2024}]{BertinKlutchnikoffOuimet2023Rcode}
\begin{bmisc}[author]
\bauthor{\bsnm{Bertin},~\bfnm{K.}\binits{K.}},
  \bauthor{\bsnm{Klutchnikoff},~\bfnm{N.}\binits{N.}} \AND
  \bauthor{\bsnm{Ouimet},~\bfnm{F.}\binits{F.}}
(\byear{2024}).
\btitle{densityLocPoly: {D}ensity {E}stimation on {C}omplicated {D}omains;
  {I}mplementation of the local polynomial density estimator in 2d.}
\bnote{\texttt{R} package version 0.0.2, available online at
  \href{https://github.com/klutchnikoff/densityLocPoly}{https://github.com/klutchnikoff/densityLocPoly}
  [This link was accessed on 2024--08--21.]}.
\end{bmisc}
\endbibitem

\bibitem[\protect\citeauthoryear{Bertin et~al.}{2020}]{Bertin_et_al2020}
\begin{barticle}[author]
\bauthor{\bsnm{Bertin},~\bfnm{K.}\binits{K.}},
  \bauthor{\bsnm{Klutchnikoff},~\bfnm{N.}\binits{N.}},
  \bauthor{\bsnm{L\'{e}on},~\bfnm{J.~R.}\binits{J.~R.}} \AND
  \bauthor{\bsnm{Prieur},~\bfnm{C.}\binits{C.}}
(\byear{2020}).
\btitle{Adaptive density estimation on bounded domains under mixing
  conditions}.
\bjournal{Electron. J. Stat.}
\bvolume{14}
\bpages{2198--2237}.
\bdoi{10.1214/20-EJS1682}
\bmrnumber{4097810}
\end{barticle}
\endbibitem

\bibitem[\protect\citeauthoryear{Botev, Grotowski and
  Kroese}{2010}]{BotevGrotowskiKroese2010}
\begin{barticle}[author]
\bauthor{\bsnm{Botev},~\bfnm{Z.~I.}\binits{Z.~I.}},
  \bauthor{\bsnm{Grotowski},~\bfnm{J.~F.}\binits{J.~F.}} \AND
  \bauthor{\bsnm{Kroese},~\bfnm{D.~P.}\binits{D.~P.}}
(\byear{2010}).
\btitle{Kernel density estimation via diffusion}.
\bjournal{Ann. Statist.}
\bvolume{38}
\bpages{2916--2957}.
\bdoi{10.1214/10-AOS799}
\bmrnumber{2722460}
\end{barticle}
\endbibitem

\bibitem[\protect\citeauthoryear{Boucheron, Lugosi and
  Massart}{2013}]{BoucheronLugosiMassart2013supp}
\begin{bbook}[author]
\bauthor{\bsnm{Boucheron},~\bfnm{S.}\binits{S.}},
  \bauthor{\bsnm{Lugosi},~\bfnm{G.}\binits{G.}} \AND
  \bauthor{\bsnm{Massart},~\bfnm{P.}\binits{P.}}
(\byear{2013}).
\btitle{Concentration {I}nequalities}.
\bpublisher{Oxford University Press, Oxford}.
\bdoi{10.1093/acprof:oso/9780199535255.001.0001}
\bmrnumber{3185193}
\end{bbook}
\endbibitem

\bibitem[\protect\citeauthoryear{Brunel}{2018}]{Brunel2018}
\begin{barticle}[author]
\bauthor{\bsnm{Brunel},~\bfnm{V.~E.}\binits{V.~E.}}
(\byear{2018}).
\btitle{Methods for estimation of convex sets}.
\bjournal{Statist. Sci.}
\bvolume{33}
\bpages{615--632}.
\bdoi{10.1214/18-STS669}
\bmrnumber{3881211}
\end{barticle}
\endbibitem

\bibitem[\protect\citeauthoryear{Cattaneo, Jansson and
  Ma}{2020}]{CattaneoJanssonMa2020JASA}
\begin{barticle}[author]
\bauthor{\bsnm{Cattaneo},~\bfnm{M.~D.}\binits{M.~D.}},
  \bauthor{\bsnm{Jansson},~\bfnm{M.}\binits{M.}} \AND
  \bauthor{\bsnm{Ma},~\bfnm{X.}\binits{X.}}
(\byear{2020}).
\btitle{Simple local polynomial density estimators}.
\bjournal{J. Amer. Statist. Assoc.}
\bvolume{115}
\bpages{1449--1455}.
\bdoi{10.1080/01621459.2019.1635480}
\bmrnumber{4143477}
\end{barticle}
\endbibitem

\bibitem[\protect\citeauthoryear{Cattaneo, Jansson and
  Ma}{2022}]{CattaneoJanssonMa2022JSS}
\begin{barticle}[author]
\bauthor{\bsnm{Cattaneo},~\bfnm{M.~D.}\binits{M.~D.}},
  \bauthor{\bsnm{Jansson},~\bfnm{M.}\binits{M.}} \AND
  \bauthor{\bsnm{Ma},~\bfnm{X.}\binits{X.}}
(\byear{2022}).
\btitle{lpdensity: {L}ocal {P}olynomial {D}ensity {E}stimation and
  {I}nference}.
\bjournal{J. Stat. Softw.}
\bvolume{101}
\bpages{1--25}.
\bdoi{10.18637/jss.v101.i02}
\end{barticle}
\endbibitem

\bibitem[\protect\citeauthoryear{Cattaneo
  et~al.}{2024}]{CattaneoChandakJanssonMa2024}
\begin{barticle}[author]
\bauthor{\bsnm{Cattaneo},~\bfnm{M.~D.}\binits{M.~D.}},
  \bauthor{\bsnm{Chandak},~\bfnm{R.}\binits{R.}},
  \bauthor{\bsnm{Jansson},~\bfnm{M.}\binits{M.}} \AND
  \bauthor{\bsnm{Ma},~\bfnm{X.}\binits{X.}}
(\byear{2024}).
\btitle{Boundary adaptive local polynomial conditional density estimators}.
\bjournal{Bernoulli}
\bvolume{30}
\bpages{3193--3223}.
\bdoi{10.3150/23-bej1711}
\bmrnumber{4779862}
\end{barticle}
\endbibitem

\bibitem[\protect\citeauthoryear{Chen}{1999}]{Chen1999}
\begin{barticle}[author]
\bauthor{\bsnm{Chen},~\bfnm{S.~X.}\binits{S.~X.}}
(\byear{1999}).
\btitle{Beta kernel estimators for density functions}.
\bjournal{Comput. Statist. Data Anal.}
\bvolume{31}
\bpages{131--145}.
\bdoi{10.1016/S0167-9473(99)00010-9}
\bmrnumber{1718494}
\end{barticle}
\endbibitem

\bibitem[\protect\citeauthoryear{Cheng, Fan and
  Marron}{1997}]{ChengFanMarron1997}
\begin{barticle}[author]
\bauthor{\bsnm{Cheng},~\bfnm{M.~Y.}\binits{M.~Y.}},
  \bauthor{\bsnm{Fan},~\bfnm{J.}\binits{J.}} \AND
  \bauthor{\bsnm{Marron},~\bfnm{J.~S.}\binits{J.~S.}}
(\byear{1997}).
\btitle{On automatic boundary corrections}.
\bjournal{Ann. Statist.}
\bvolume{25}
\bpages{1691--1708}.
\bdoi{10.1214/aos/1031594737}
\bmrnumber{1463570}
\end{barticle}
\endbibitem

\bibitem[\protect\citeauthoryear{Cline and Hart}{1991}]{ClineHart1991}
\begin{barticle}[author]
\bauthor{\bsnm{Cline},~\bfnm{D.~B.~H.}\binits{D.~B.~H.}} \AND
  \bauthor{\bsnm{Hart},~\bfnm{J.~D.}\binits{J.~D.}}
(\byear{1991}).
\btitle{Kernel estimation of densities with discontinuities or discontinuous
  derivatives}.
\bjournal{Statistics}
\bvolume{22}
\bpages{69--84}.
\bdoi{10.1080/02331889108802286}
\bmrnumber{1097362}
\end{barticle}
\endbibitem

\bibitem[\protect\citeauthoryear{Davies, Marshall and
  Hazelton}{2018}]{DaviesMarshallHazelton2018}
\begin{barticle}[author]
\bauthor{\bsnm{Davies},~\bfnm{T.~M.}\binits{T.~M.}},
  \bauthor{\bsnm{Marshall},~\bfnm{J.~C.}\binits{J.~C.}} \AND
  \bauthor{\bsnm{Hazelton},~\bfnm{M.~L.}\binits{M.~L.}}
(\byear{2018}).
\btitle{Tutorial on kernel estimation of continuous spatial and spatiotemporal
  relative risk}.
\bjournal{Stat. Med.}
\bvolume{37}
\bpages{1191--1221}.
\bdoi{10.1002/sim.7577}
\bmrnumber{3777968}
\end{barticle}
\endbibitem

\bibitem[\protect\citeauthoryear{Davies and Marshall}{2024}]{Davies_et_al2024}
\begin{bmisc}[author]
\bauthor{\bsnm{Davies},~\bfnm{T.~M.}\binits{T.~M.}} \AND
  \bauthor{\bsnm{Marshall},~\bfnm{J.~C.}\binits{J.~C.}}
(\byear{2024}).
\btitle{sparr: {S}patial and {S}patiotemporal {R}elative {R}isk.}
\bnote{\texttt{R} package version 2.3-15, available online at
  \href{https://cran.r-project.org/web/packages/sparr/index.html}{https://cran.r-project.org/web/packages/sparr/index.html}
  [This link was accessed on 2024--08--21.]}.
\end{bmisc}
\endbibitem

\bibitem[\protect\citeauthoryear{Fan and Gijbels}{1992}]{FanGijbels1992}
\begin{barticle}[author]
\bauthor{\bsnm{Fan},~\bfnm{J.}\binits{J.}} \AND
  \bauthor{\bsnm{Gijbels},~\bfnm{I.}\binits{I.}}
(\byear{1992}).
\btitle{Variable bandwidth and local linear regression smoothers}.
\bjournal{Ann. Statist.}
\bvolume{20}
\bpages{2008--2036}.
\bdoi{10.1214/aos/1176348900}
\bmrnumber{1193323}
\end{barticle}
\endbibitem

\bibitem[\protect\citeauthoryear{Fan and Gijbels}{1996}]{FanGijbels1996}
\begin{bbook}[author]
\bauthor{\bsnm{Fan},~\bfnm{J.}\binits{J.}} \AND
  \bauthor{\bsnm{Gijbels},~\bfnm{I.}\binits{I.}}
(\byear{1996}).
\btitle{Local {P}olynomial {M}odelling and {I}ts {A}pplications}.
\bseries{Monographs on Statistics and Applied Probability}
\bvolume{66}.
\bpublisher{Chapman \& Hall, London}.
\bdoi{10.1007/978-1-4899-3150-4}
\bmrnumber{1383587}
\end{bbook}
\endbibitem

\bibitem[\protect\citeauthoryear{Feng et~al.}{2021}]{Feng_et_al2021}
\begin{barticle}[author]
\bauthor{\bsnm{Feng},~\bfnm{O.~Y.}\binits{O.~Y.}},
  \bauthor{\bsnm{Guntuboyina},~\bfnm{A.}\binits{A.}},
  \bauthor{\bsnm{Kim},~\bfnm{A.~K.~H.}\binits{A.~K.~H.}} \AND
  \bauthor{\bsnm{Samworth},~\bfnm{R.~J.}\binits{R.~J.}}
(\byear{2021}).
\btitle{Adaptation in multivariate log-concave density estimation}.
\bjournal{Ann. Statist.}
\bvolume{49}
\bpages{129--153}.
\bdoi{10.1214/20-AOS1950}
\bmrnumber{4206672}
\end{barticle}
\endbibitem

\bibitem[\protect\citeauthoryear{Ferraccioli
  et~al.}{2021}]{Ferraccioli_et_al2021}
\begin{barticle}[author]
\bauthor{\bsnm{Ferraccioli},~\bfnm{F.}\binits{F.}},
  \bauthor{\bsnm{Arnone},~\bfnm{E.}\binits{E.}},
  \bauthor{\bsnm{Finos},~\bfnm{L.}\binits{L.}},
  \bauthor{\bsnm{Ramsay},~\bfnm{J.~O.}\binits{J.~O.}} \AND
  \bauthor{\bsnm{Sangalli},~\bfnm{L.~M.}\binits{L.~M.}}
(\byear{2021}).
\btitle{Nonparametric density estimation over complicated domains}.
\bjournal{J. R. Stat. Soc. Ser. B. Stat. Methodol.}
\bvolume{83}
\bpages{346--368}.
\bdoi{10.1111/rssb.12415}
\bmrnumber{4250279}
\end{barticle}
\endbibitem

\bibitem[\protect\citeauthoryear{Gasser and
  M\"{u}ller}{1979}]{GasserMuller1979}
\begin{binproceedings}[author]
\bauthor{\bsnm{Gasser},~\bfnm{T.}\binits{T.}} \AND
  \bauthor{\bsnm{M\"{u}ller},~\bfnm{H.}\binits{H.}}
(\byear{1979}).
\btitle{Kernel estimation of regression functions}.
In \bbooktitle{Smoothing techniques for curve estimation ({P}roc. {W}orkshop,
  {H}eidelberg, 1979)}.
\bseries{Lecture Notes in Math.}
\bvolume{757}
\bpages{23--68}.
\bpublisher{Springer, Berlin}.
\bdoi{10.1007/BFb0098489}
\bmrnumber{564251}
\end{binproceedings}
\endbibitem

\bibitem[\protect\citeauthoryear{Gasser, M\"{u}ller and
  Mammitzsch}{1985}]{GasserMullerMammitzsch1985}
\begin{barticle}[author]
\bauthor{\bsnm{Gasser},~\bfnm{T.}\binits{T.}},
  \bauthor{\bsnm{M\"{u}ller},~\bfnm{H.~G.}\binits{H.~G.}} \AND
  \bauthor{\bsnm{Mammitzsch},~\bfnm{V.}\binits{V.}}
(\byear{1985}).
\btitle{Kernels \hspace{-0.2mm}for \hspace{-0.2mm}nonparametric
  \hspace{-0.2mm}curve \hspace{-0.2mm}estimation}.
\bjournal{J. Roy. Statist. Soc. Ser. B}
\bvolume{47}
\bpages{238--252}.
\bdoi{10.1111/j.2517-6161.1985.tb01350.x}
\bmrnumber{816088}
\end{barticle}
\endbibitem

\bibitem[\protect\citeauthoryear{Gawronski and
  Stadtm\"{u}ller}{1981}]{GawronskiStadtmuller1981}
\begin{barticle}[author]
\bauthor{\bsnm{Gawronski},~\bfnm{W.}\binits{W.}} \AND
  \bauthor{\bsnm{Stadtm\"{u}ller},~\bfnm{U.}\binits{U.}}
(\byear{1981}).
\btitle{Smoothing histograms by means of lattice and continuous distributions}.
\bjournal{Metrika}
\bvolume{28}
\bpages{155--164}.
\bdoi{10.1007/BF01902889}
\bmrnumber{638651}
\end{barticle}
\endbibitem

\bibitem[\protect\citeauthoryear{Ghosal}{2001}]{Ghosal2001}
\begin{barticle}[author]
\bauthor{\bsnm{Ghosal},~\bfnm{S.}\binits{S.}}
(\byear{2001}).
\btitle{Convergence rates for density estimation with {B}ernstein polynomials}.
\bjournal{Ann. Statist.}
\bvolume{29}
\bpages{1264--1280}.
\bdoi{10.1214/aos/1013203453}
\bmrnumber{1873330}
\end{barticle}
\endbibitem

\bibitem[\protect\citeauthoryear{Goldenshluger and
  Lepski}{2008}]{GoldenshlugerLepski2008}
\begin{barticle}[author]
\bauthor{\bsnm{Goldenshluger},~\bfnm{A.}\binits{A.}} \AND
  \bauthor{\bsnm{Lepski},~\bfnm{O.}\binits{O.}}
(\byear{2008}).
\btitle{Universal pointwise selection rule in multivariate function
  estimation}.
\bjournal{Bernoulli}
\bvolume{14}
\bpages{1150--1190}.
\bdoi{10.3150/08-BEJ144}
\bmrnumber{2543590}
\end{barticle}
\endbibitem

\bibitem[\protect\citeauthoryear{Goldenshluger and
  Lepski}{2009}]{GoldenshlugerLepski2009}
\begin{barticle}[author]
\bauthor{\bsnm{Goldenshluger},~\bfnm{A.}\binits{A.}} \AND
  \bauthor{\bsnm{Lepski},~\bfnm{O.}\binits{O.}}
(\byear{2009}).
\btitle{Structural adaptation via {$\mathbb{L}_p$}-norm oracle inequalities}.
\bjournal{Probab. Theory Related Fields}
\bvolume{143}
\bpages{41--71}.
\bdoi{10.1007/s00440-007-0119-5}
\bmrnumber{2449122}
\end{barticle}
\endbibitem

\bibitem[\protect\citeauthoryear{Goldenshluger and
  Lepski}{2011}]{GoldenshlugerLepski2011}
\begin{barticle}[author]
\bauthor{\bsnm{Goldenshluger},~\bfnm{A.}\binits{A.}} \AND
  \bauthor{\bsnm{Lepski},~\bfnm{O.}\binits{O.}}
(\byear{2011}).
\btitle{Bandwidth selection in kernel density estimation: oracle inequalities
  and adaptive minimax optimality}.
\bjournal{Ann. Statist.}
\bvolume{39}
\bpages{1608--1632}.
\bdoi{10.1214/11-AOS883}
\bmrnumber{2850214}
\end{barticle}
\endbibitem

\bibitem[\protect\citeauthoryear{Goldenshluger and
  Lepski}{2014}]{GoldenshlugerLepski2014}
\begin{barticle}[author]
\bauthor{\bsnm{Goldenshluger},~\bfnm{A.}\binits{A.}} \AND
  \bauthor{\bsnm{Lepski},~\bfnm{O.}\binits{O.}}
(\byear{2014}).
\btitle{On adaptive minimax density estimation on {$R^d$}}.
\bjournal{Probab. Theory Related Fields}
\bvolume{159}
\bpages{479--543}.
\bdoi{10.1007/s00440-013-0512-1}
\bmrnumber{3230001}
\end{barticle}
\endbibitem

\bibitem[\protect\citeauthoryear{Gu}{1993}]{Gu1993}
\begin{barticle}[author]
\bauthor{\bsnm{Gu},~\bfnm{C.}\binits{C.}}
(\byear{1993}).
\btitle{Smoothing spline density estimation: a dimensionless automatic
  algorithm}.
\bjournal{J. Amer. Statist. Assoc.}
\bvolume{88}
\bpages{495--504}.
\bdoi{10.1080/01621459.1993.10476300}
\bmrnumber{1224374}
\end{barticle}
\endbibitem

\bibitem[\protect\citeauthoryear{Gu and Qiu}{1993}]{GuQiu1993}
\begin{barticle}[author]
\bauthor{\bsnm{Gu},~\bfnm{C.}\binits{C.}} \AND
  \bauthor{\bsnm{Qiu},~\bfnm{C.}\binits{C.}}
(\byear{1993}).
\btitle{Smoothing spline density estimation: theory}.
\bjournal{Ann. Statist.}
\bvolume{21}
\bpages{217--234}.
\bdoi{10.1214/aos/1176349023}
\bmrnumber{1212174}
\end{barticle}
\endbibitem

\bibitem[\protect\citeauthoryear{Guillas and Lai}{2010}]{GuillasLai2010}
\begin{barticle}[author]
\bauthor{\bsnm{Guillas},~\bfnm{S.}\binits{S.}} \AND
  \bauthor{\bsnm{Lai},~\bfnm{M.~J.}\binits{M.~J.}}
(\byear{2010}).
\btitle{Bivariate splines for spatial functional regression models}.
\bjournal{J. Nonparametr. Stat.}
\bvolume{22}
\bpages{477--497}.
\bdoi{10.1080/10485250903323180}
\bmrnumber{2662608}
\end{barticle}
\endbibitem

\bibitem[\protect\citeauthoryear{Jones}{1993}]{Jones1993}
\begin{barticle}[author]
\bauthor{\bsnm{Jones},~\bfnm{M.~C.}\binits{M.~C.}}
(\byear{1993}).
\btitle{Simple boundary correction for kernel density estimation}.
\bjournal{Stat Comput .}
\bvolume{3}
\bpages{135--146}.
\bdoi{10.1007/BF00147776}
\end{barticle}
\endbibitem

\bibitem[\protect\citeauthoryear{Jones and Foster}{1996}]{JonesFoster1996}
\begin{barticle}[author]
\bauthor{\bsnm{Jones},~\bfnm{M.~C.}\binits{M.~C.}} \AND
  \bauthor{\bsnm{Foster},~\bfnm{P.~J.}\binits{P.~J.}}
(\byear{1996}).
\btitle{A simple nonnegative boundary correction method for kernel density
  estimation}.
\bjournal{Statist. Sinica}
\bvolume{6}
\bpages{1005--1013}.
\bnote{\href{https://www.jstor.org/stable/24306056}{https://www.jstor.org/stable/24306056}}.
\bmrnumber{1422417}
\end{barticle}
\endbibitem

\bibitem[\protect\citeauthoryear{Klemel\"{a}}{2009}]{Klemala2009}
\begin{barticle}[author]
\bauthor{\bsnm{Klemel\"{a}},~\bfnm{J.}\binits{J.}}
(\byear{2009}).
\btitle{Multivariate histograms with data-dependent partitions}.
\bjournal{Statist. Sinica}
\bvolume{19}
\bpages{159--176}.
\bnote{\href{https://www.jstor.org/stable/24308713}{https://www.jstor.org/stable/24308713}}.
\bmrnumber{2487883}
\end{barticle}
\endbibitem

\bibitem[\protect\citeauthoryear{Klutchnikoff}{2014}]{K-MMS-2014}
\begin{barticle}[author]
\bauthor{\bsnm{Klutchnikoff},~\bfnm{N.}\binits{N.}}
(\byear{2014}).
\btitle{Pointwise adaptive estimation of a multivariate function}.
\bjournal{Math. Methods Statist.}
\bvolume{23}
\bpages{132--150}.
\bdoi{10.3103/S1066530714020045}
\bmrnumber{3224636}
\end{barticle}
\endbibitem

\bibitem[\protect\citeauthoryear{Lai and Schumaker}{2007}]{LaiSchumaker2007}
\begin{bbook}[author]
\bauthor{\bsnm{Lai},~\bfnm{M.~J.}\binits{M.~J.}} \AND
  \bauthor{\bsnm{Schumaker},~\bfnm{L.~L.}\binits{L.~L.}}
(\byear{2007}).
\btitle{Spline {F}unctions on {T}riangulations}.
\bseries{Encyclopedia of Mathematics and its Applications}
\bvolume{110}.
\bpublisher{Cambridge University Press, Cambridge}.
\bdoi{10.1017/CBO9780511721588}
\bmrnumber{2355272}
\end{bbook}
\endbibitem

\bibitem[\protect\citeauthoryear{Lai and Wang}{2013}]{LaiWang2013}
\begin{barticle}[author]
\bauthor{\bsnm{Lai},~\bfnm{M.~J.}\binits{M.~J.}} \AND
  \bauthor{\bsnm{Wang},~\bfnm{L.}\binits{L.}}
(\byear{2013}).
\btitle{Bivariate penalized splines for regression}.
\bjournal{Statist. Sinica}
\bvolume{23}
\bpages{1399--1417}.
\bdoi{10.5705/ss.2010.278}
\bmrnumber{3114719}
\end{barticle}
\endbibitem

\bibitem[\protect\citeauthoryear{Lepski}{2015}]{Lepski2015supp}
\begin{barticle}[author]
\bauthor{\bsnm{Lepski},~\bfnm{O.}\binits{O.}}
(\byear{2015}).
\btitle{Adaptive estimation over anisotropic functional classes via oracle
  approach}.
\bjournal{Ann. Statist.}
\bvolume{43}
\bpages{1178--1242}.
\bdoi{10.1214/14-AOS1306}
\bmrnumber{3346701}
\end{barticle}
\endbibitem

\bibitem[\protect\citeauthoryear{Lepskii}{1991}]{Lepski1991adaptive}
\begin{barticle}[author]
\bauthor{\bsnm{Lepskii},~\bfnm{O.~V.}\binits{O.~V.}}
(\byear{1991}).
\btitle{On a problem of adaptive estimation in {G}aussian white noise}.
\bjournal{Theory Probab. Appl.}
\bvolume{35}
\bpages{454--466}.
\bdoi{10.1137/1135065}
\end{barticle}
\endbibitem

\bibitem[\protect\citeauthoryear{Lindgren, Rue and
  Lindstr\"{o}m}{2011}]{LingrenRueLindstrom2011}
\begin{barticle}[author]
\bauthor{\bsnm{Lindgren},~\bfnm{F.}\binits{F.}},
  \bauthor{\bsnm{Rue},~\bfnm{H.}\binits{H.}} \AND
  \bauthor{\bsnm{Lindstr\"{o}m},~\bfnm{J.}\binits{J.}}
(\byear{2011}).
\btitle{An explicit link between {G}aussian fields and {G}aussian {M}arkov
  random fields: the stochastic partial differential equation approach}.
\bjournal{J. R. Stat. Soc. Ser. B Stat. Methodol.}
\bvolume{73}
\bpages{423--498}.
\bdoi{10.1111/j.1467-9868.2011.00777.x}
\bmrnumber{2853727}
\end{barticle}
\endbibitem

\bibitem[\protect\citeauthoryear{Liu and Wu}{2019}]{LiuWu2019}
\begin{barticle}[author]
\bauthor{\bsnm{Liu},~\bfnm{Y.}\binits{Y.}} \AND
  \bauthor{\bsnm{Wu},~\bfnm{C.}\binits{C.}}
(\byear{2019}).
\btitle{Point-wise estimation for anisotropic densities}.
\bjournal{J. Multivariate Anal.}
\bvolume{171}
\bpages{112--125}.
\bmrnumber{3892029}
\end{barticle}
\endbibitem

\bibitem[\protect\citeauthoryear{Marron and Ruppert}{1994}]{MarronRuppert1994}
\begin{barticle}[author]
\bauthor{\bsnm{Marron},~\bfnm{J.~S.}\binits{J.~S.}} \AND
  \bauthor{\bsnm{Ruppert},~\bfnm{D.}\binits{D.}}
(\byear{1994}).
\btitle{Transformations to reduce boundary bias in kernel density estimation}.
\bjournal{J. Roy. Statist. Soc. Ser. B}
\bvolume{56}
\bpages{653--671}.
\bdoi{10.1111/j.2517-6161.1994.tb02006.x}
\bmrnumber{1293239}
\end{barticle}
\endbibitem

\bibitem[\protect\citeauthoryear{McIntyre and Barry}{2018}]{McIntyreBarry2018}
\begin{barticle}[author]
\bauthor{\bsnm{McIntyre},~\bfnm{J.}\binits{J.}} \AND
  \bauthor{\bsnm{Barry},~\bfnm{R.~P.}\binits{R.~P.}}
(\byear{2018}).
\btitle{A lattice-based smoother for regions with irregular boundaries and
  holes}.
\bjournal{J. Comput. Graph. Statist.}
\bvolume{27}
\bpages{360--367}.
\bdoi{10.1080/10618600.2017.1375935}
\bmrnumber{3816271}
\end{barticle}
\endbibitem

\bibitem[\protect\citeauthoryear{McSwiggan, Baddeley and
  Nair}{2017}]{McSwigganBaddeleyNair2017}
\begin{barticle}[author]
\bauthor{\bsnm{McSwiggan},~\bfnm{G.}\binits{G.}},
  \bauthor{\bsnm{Baddeley},~\bfnm{A.}\binits{A.}} \AND
  \bauthor{\bsnm{Nair},~\bfnm{G.}\binits{G.}}
(\byear{2017}).
\btitle{Kernel density estimation on a linear network}.
\bjournal{Scand. J. Stat.}
\bvolume{44}
\bpages{324--345}.
\bdoi{10.1111/sjos.12255}
\bmrnumber{3658517}
\end{barticle}
\endbibitem

\bibitem[\protect\citeauthoryear{Miller and Wood}{2014}]{MillerWood2014}
\begin{barticle}[author]
\bauthor{\bsnm{Miller},~\bfnm{D.~L.}\binits{D.~L.}} \AND
  \bauthor{\bsnm{Wood},~\bfnm{S.~N.}\binits{S.~N.}}
(\byear{2014}).
\btitle{Finite area smoothing with generalized distance splines}.
\bjournal{Environ. Ecol. Stat.}
\bvolume{21}
\bpages{715--731}.
\bdoi{10.1007/s10651-014-0277-4}
\bmrnumber{3279587}
\end{barticle}
\endbibitem

\bibitem[\protect\citeauthoryear{M\"{u}ller}{1991}]{Muller1991}
\begin{barticle}[author]
\bauthor{\bsnm{M\"{u}ller},~\bfnm{H.~G.}\binits{H.~G.}}
(\byear{1991}).
\btitle{Smooth optimum kernel estimators near endpoints}.
\bjournal{Biometrika}
\bvolume{78}
\bpages{521--530}.
\bdoi{10.2307/2337021}
\bmrnumber{1130920}
\end{barticle}
\endbibitem

\bibitem[\protect\citeauthoryear{Niu et~al.}{2019}]{Niu_et_al2019}
\begin{barticle}[author]
\bauthor{\bsnm{Niu},~\bfnm{M.}\binits{M.}},
  \bauthor{\bsnm{Cheung},~\bfnm{P.}\binits{P.}},
  \bauthor{\bsnm{Lin},~\bfnm{L.}\binits{L.}},
  \bauthor{\bsnm{Dai},~\bfnm{Z.}\binits{Z.}},
  \bauthor{\bsnm{Lawrence},~\bfnm{N.}\binits{N.}} \AND
  \bauthor{\bsnm{Dunson},~\bfnm{D.}\binits{D.}}
(\byear{2019}).
\btitle{Intrinsic {G}aussian processes on complex constrained domains}.
\bjournal{J. R. Stat. Soc. Ser. B. Stat. Methodol.}
\bvolume{81}
\bpages{603--627}.
\bdoi{10.1111/rssb.12320}
\bmrnumber{3961500}
\end{barticle}
\endbibitem

\bibitem[\protect\citeauthoryear{Petrone}{1999}]{Petrone1999}
\begin{barticle}[author]
\bauthor{\bsnm{Petrone},~\bfnm{S.}\binits{S.}}
(\byear{1999}).
\btitle{Bayesian density estimation using {B}ernstein polynomials}.
\bjournal{Canad. J. Statist.}
\bvolume{27}
\bpages{105--126}.
\bdoi{10.2307/3315494}
\bmrnumber{1703623}
\end{barticle}
\endbibitem

\bibitem[\protect\citeauthoryear{Petrone and
  Wasserman}{2002}]{PetroneWasserman2002}
\begin{barticle}[author]
\bauthor{\bsnm{Petrone},~\bfnm{S.}\binits{S.}} \AND
  \bauthor{\bsnm{Wasserman},~\bfnm{L.}\binits{L.}}
(\byear{2002}).
\btitle{Consistency of {B}ernstein polynomial posteriors}.
\bjournal{J. Roy. Statist. Soc. Ser. B}
\bvolume{64}
\bpages{79--100}.
\bdoi{10.1111/1467-9868.00326}
\bmrnumber{1881846}
\end{barticle}
\endbibitem

\bibitem[\protect\citeauthoryear{Ramsay}{2002}]{Ramsay2002}
\begin{barticle}[author]
\bauthor{\bsnm{Ramsay},~\bfnm{T.}\binits{T.}}
(\byear{2002}).
\btitle{Spline smoothing over difficult regions}.
\bjournal{J. R. Stat. Soc. Ser. B Stat. Methodol.}
\bvolume{64}
\bpages{307--319}.
\bdoi{10.1111/1467-9868.00339}
\bmrnumber{1904707}
\end{barticle}
\endbibitem

\bibitem[\protect\citeauthoryear{Rebelles}{2015a}]{Rebelles2015}
\begin{barticle}[author]
\bauthor{\bsnm{Rebelles},~\bfnm{G.}\binits{G.}}
(\byear{2015}a).
\btitle{Pointwise adaptive estimation of a multivariate density under
  independence hypothesis}.
\bjournal{Bernoulli}
\bvolume{21}
\bpages{1984--2023}.
\bdoi{10.3150/14-BEJ633}
\bmrnumber{3378457}
\end{barticle}
\endbibitem

\bibitem[\protect\citeauthoryear{Rebelles}{2015b}]{Rebelles2015supp}
\begin{barticle}[author]
\bauthor{\bsnm{Rebelles},~\bfnm{G.}\binits{G.}}
(\byear{2015}b).
\btitle{Pointwise adaptive estimation of a multivariate density under
  independence hypothesis}.
\bjournal{Bernoulli}
\bvolume{21}
\bpages{1984--2023}.
\bdoi{10.3150/14-BEJ633}
\bmrnumber{3378457}
\end{barticle}
\endbibitem

\bibitem[\protect\citeauthoryear{Robinson}{2015}]{Robinson2015}
\begin{bmisc}[author]
\bauthor{\bsnm{Robinson},~\bfnm{D.}\binits{D.}}
(\byear{2015}).
\btitle{View package downloads over time with {S}hiny.}
\bnote{Blog post, available online at
  \href{http://varianceexplained.org/r/cran-view}{http://varianceexplained.org/r/cran-view}
  [This link was accessed on 2024--08--21.]}.
\end{bmisc}
\endbibitem

\bibitem[\protect\citeauthoryear{Ruppert and Cline}{1994}]{RupperCline1994}
\begin{barticle}[author]
\bauthor{\bsnm{Ruppert},~\bfnm{D.}\binits{D.}} \AND
  \bauthor{\bsnm{Cline},~\bfnm{D.~B.~H.}\binits{D.~B.~H.}}
(\byear{1994}).
\btitle{Bias reduction in kernel density estimation by smoothed empirical
  transformations}.
\bjournal{Ann. Statist.}
\bvolume{22}
\bpages{185--210}.
\bdoi{10.1214/aos/1176325365}
\bmrnumber{1272080}
\end{barticle}
\endbibitem

\bibitem[\protect\citeauthoryear{Sangalli}{2021}]{Sangalli2021}
\begin{barticle}[author]
\bauthor{\bsnm{Sangalli},~\bfnm{L.~M.}\binits{L.~M.}}
(\byear{2021}).
\btitle{Spatial regression with partial differential equation regularisation}.
\bjournal{Int. Stat. Rev.}
\bvolume{89}
\bpages{505--531}.
\bdoi{10.1111/insr.12444}
\bmrnumber{4411916}
\end{barticle}
\endbibitem

\bibitem[\protect\citeauthoryear{Sangalli, Ramsay and
  Ramsay}{2013}]{SangalliRansayRamsay2013}
\begin{barticle}[author]
\bauthor{\bsnm{Sangalli},~\bfnm{L.~M.}\binits{L.~M.}},
  \bauthor{\bsnm{Ramsay},~\bfnm{J.~O.}\binits{J.~O.}} \AND
  \bauthor{\bsnm{Ramsay},~\bfnm{T.~O.}\binits{T.~O.}}
(\byear{2013}).
\btitle{Spatial spline regression models}.
\bjournal{J. R. Stat. Soc. Ser. B. Stat. Methodol.}
\bvolume{75}
\bpages{681--703}.
\bdoi{10.1111/rssb.12009}
\bmrnumber{3091654}
\end{barticle}
\endbibitem

\bibitem[\protect\citeauthoryear{Schuster}{1985}]{Schuster1985}
\begin{barticle}[author]
\bauthor{\bsnm{Schuster},~\bfnm{E.~F.}\binits{E.~F.}}
(\byear{1985}).
\btitle{\hspace{-0.1mm}Incorporating \hspace{-0.1mm}support
  \hspace{-0.1mm}constraints \hspace{-0.1mm}into \hspace{-0.1mm}nonparametric
  \hspace{-0.1mm}estimators \hspace{-0.1mm}of \hspace{-0.1mm}densities}.
\bjournal{Commun. Statist. Theor. Meth.}
\bvolume{14}
\bpages{1123--1136}.
\bdoi{10.1080/03610928508828965}
\bmrnumber{797636}
\end{barticle}
\endbibitem

\bibitem[\protect\citeauthoryear{Scott-Hayward
  et~al.}{2015}]{Scott-Hayward_et_al2015}
\begin{barticle}[author]
\bauthor{\bsnm{Scott-Hayward},~\bfnm{L.~A.~S.}\binits{L.~A.~S.}},
  \bauthor{\bsnm{Mackenzie},~\bfnm{M.~L.}\binits{M.~L.}},
  \bauthor{\bsnm{Ashe},~\bfnm{E.}\binits{E.}} \AND
  \bauthor{\bsnm{Williams},~\bfnm{R.}\binits{R.}}
(\byear{2015}).
\btitle{Modelling killer whale feeding behaviour using a spatially adaptive
  complex region spatial smoother ({CR}e{SS}) and generalised estimating
  equations ({GEE}s)}.
\bjournal{J. Agric. Biol. Environ. Stat.}
\bvolume{20}
\bpages{305--322}.
\bdoi{10.1007/s13253-015-0209-2}
\bmrnumber{3396553}
\end{barticle}
\endbibitem

\bibitem[\protect\citeauthoryear{Shiryaev}{1996}]{MR1368405}
\begin{bbook}[author]
\bauthor{\bsnm{Shiryaev},~\bfnm{A.~N.}\binits{A.~N.}}
(\byear{1996}).
\btitle{Probability},
\bedition{second} ed.
\bseries{Graduate Texts in Mathematics}
\bvolume{95}.
\bpublisher{Springer-Verlag, New York}.
\bdoi{10.1007/978-1-4757-2539-1}
\bmrnumber{1368405}
\end{bbook}
\endbibitem

\bibitem[\protect\citeauthoryear{Tsybakov}{1998}]{Tsybakov1998}
\begin{barticle}[author]
\bauthor{\bsnm{Tsybakov},~\bfnm{A.~B.}\binits{A.~B.}}
(\byear{1998}).
\btitle{Pointwise and sup-norm sharp adaptive estimation of functions on the
  {S}obolev classes}.
\bjournal{Ann. Statist.}
\bvolume{26}
\bpages{2420--2469}.
\bdoi{10.1214/aos/1024691478}
\bmrnumber{1700239}
\end{barticle}
\endbibitem

\bibitem[\protect\citeauthoryear{Tsybakov}{2004}]{Tsybakov2004frenchsupp}
\begin{bbook}[author]
\bauthor{\bsnm{Tsybakov},~\bfnm{A.~B.}\binits{A.~B.}}
(\byear{2004}).
\btitle{Introduction \`a l'estimation non-param\'{e}trique [french]}.
\bseries{Math\'{e}matiques \& Applications (Berlin) [Mathematics \&
  Applications]}
\bvolume{41}.
\bpublisher{Springer-Verlag, Berlin}.
\bmrnumber{2013911}
\end{bbook}
\endbibitem

\bibitem[\protect\citeauthoryear{Tsybakov}{2009}]{Tsybakov2009supp}
\begin{bbook}[author]
\bauthor{\bsnm{Tsybakov},~\bfnm{A.~B.}\binits{A.~B.}}
(\byear{2009}).
\btitle{Introduction to {N}onparametric {E}stimation}.
\bseries{Springer Series in Statistics}.
\bpublisher{Springer, New York}.
\bdoi{10.1007/b13794}
\bmrnumber{2724359}
\end{bbook}
\endbibitem

\bibitem[\protect\citeauthoryear{Varet
  et~al.}{2023}]{VaretLacourMassartRivoirard2023}
\begin{barticle}[author]
\bauthor{\bsnm{Varet},~\bfnm{S.}\binits{S.}},
  \bauthor{\bsnm{Lacour},~\bfnm{C.}\binits{C.}},
  \bauthor{\bsnm{Massart},~\bfnm{P.}\binits{P.}} \AND
  \bauthor{\bsnm{Rivoirard},~\bfnm{V.}\binits{V.}}
(\byear{2023}).
\btitle{Numerical performance of penalized comparison to overfitting for
  multivariate kernel density estimation}.
\bjournal{ESAIM Probab. Stat.}
\bvolume{27}
\bpages{621--667}.
\bdoi{10.1051/ps/2022018}
\bmrnumber{4602387}
\end{barticle}
\endbibitem

\bibitem[\protect\citeauthoryear{Vitale}{1975}]{Vitale1975}
\begin{binproceedings}[author]
\bauthor{\bsnm{Vitale},~\bfnm{R.~A.}\binits{R.~A.}}
(\byear{1975}).
\btitle{Bernstein polynomial approach to density function estimation}.
In \bbooktitle{Statistical inference and related topics ({P}roc. {S}ummer
  {R}es. {I}nst. {S}tatist. {I}nference for {S}tochastic {P}rocesses, {I}ndiana
  {U}niv., {B}loomington, {I}nd., 1974, {V}ol. 2; dedicated to {Z}. {W}.
  {B}irnbaum)}
\bpages{87--99}.
\bpublisher{Academic Press, New York}.
\bdoi{10.1016/B978-0-12-568002-8.50011-2}
\bmrnumber{0397977}
\end{binproceedings}
\endbibitem

\bibitem[\protect\citeauthoryear{Wand and Jones}{1995}]{MR1319818}
\begin{bbook}[author]
\bauthor{\bsnm{Wand},~\bfnm{M.~P.}\binits{M.~P.}} \AND
  \bauthor{\bsnm{Jones},~\bfnm{M.~C.}\binits{M.~C.}}
(\byear{1995}).
\btitle{Kernel smoothing}.
\bseries{Monographs on Statistics and Applied Probability}
\bvolume{60}.
\bpublisher{Chapman and Hall, Ltd., London}.
\bdoi{10.1007/978-1-4899-4493-1}
\bmrnumber{1319818}
\end{bbook}
\endbibitem

\bibitem[\protect\citeauthoryear{Wang and Ranalli}{2007}]{WangRanalli2007}
\begin{barticle}[author]
\bauthor{\bsnm{Wang},~\bfnm{H.}\binits{H.}} \AND
  \bauthor{\bsnm{Ranalli},~\bfnm{M.~G.}\binits{M.~G.}}
(\byear{2007}).
\btitle{Low-rank smoothing splines on complicated domains}.
\bjournal{Biometrics}
\bvolume{63}
\bpages{209--217}.
\bdoi{10.1111/j.1541-0420.2006.00674.x}
\bmrnumber{2345591}
\end{barticle}
\endbibitem

\bibitem[\protect\citeauthoryear{Wood, Bravington and
  Hedley}{2008}]{WoodBravingtonHedley2008}
\begin{barticle}[author]
\bauthor{\bsnm{Wood},~\bfnm{S.~N.}\binits{S.~N.}},
  \bauthor{\bsnm{Bravington},~\bfnm{M.~V.}\binits{M.~V.}} \AND
  \bauthor{\bsnm{Hedley},~\bfnm{S.~L.}\binits{S.~L.}}
(\byear{2008}).
\btitle{Soap film smoothing}.
\bjournal{J. R. Stat. Soc. Ser. B Stat. Methodol.}
\bvolume{70}
\bpages{931--955}.
\bdoi{10.1111/j.1467-9868.2008.00665.x}
\bmrnumber{2530324}
\end{barticle}
\endbibitem

\bibitem[\protect\citeauthoryear{Xu and Samworth}{2021}]{XuSamworth2021}
\begin{barticle}[author]
\bauthor{\bsnm{Xu},~\bfnm{M.}\binits{M.}} \AND
  \bauthor{\bsnm{Samworth},~\bfnm{R.~J.}\binits{R.~J.}}
(\byear{2021}).
\btitle{High-dimensional nonparametric density estimation via symmetry and
  shape constraints}.
\bjournal{Ann. Statist.}
\bvolume{49}
\bpages{650--672}.
\bdoi{10.1214/20-aos1972}
\bmrnumber{4255102}
\end{barticle}
\endbibitem

\bibitem[\protect\citeauthoryear{Yu et~al.}{2021}]{Yu_et_al2021}
\begin{barticle}[author]
\bauthor{\bsnm{Yu},~\bfnm{S.}\binits{S.}},
  \bauthor{\bsnm{Wang},~\bfnm{G.}\binits{G.}},
  \bauthor{\bsnm{Wang},~\bfnm{L.}\binits{L.}} \AND
  \bauthor{\bsnm{Yang},~\bfnm{L.}\binits{L.}}
(\byear{2021}).
\btitle{Multivariate spline estimation and inference for image-on-scalar
  regression}.
\bjournal{Statist. Sinica}
\bvolume{31}
\bpages{1463--1487}.
\bdoi{10.5705/ss.202019.0188}
\bmrnumber{4297712}
\end{barticle}
\endbibitem

\bibitem[\protect\citeauthoryear{Zhang and
  Karunamuni}{1998}]{ZhangKarunamuni1998}
\begin{barticle}[author]
\bauthor{\bsnm{Zhang},~\bfnm{S.}\binits{S.}} \AND
  \bauthor{\bsnm{Karunamuni},~\bfnm{R.~J.}\binits{R.~J.}}
(\byear{1998}).
\btitle{On kernel density estimation near endpoints}.
\bjournal{J. Statist. Plann. Inference}
\bvolume{70}
\bpages{301--316}.
\bdoi{10.1016/S0378-3758(97)00187-0}
\bmrnumber{1649872}
\end{barticle}
\endbibitem

\bibitem[\protect\citeauthoryear{Zhang, Karunamuni and
  Jones}{1999}]{ZhangKarunamuniJones1999}
\begin{barticle}[author]
\bauthor{\bsnm{Zhang},~\bfnm{S.}\binits{S.}},
  \bauthor{\bsnm{Karunamuni},~\bfnm{R.~J.}\binits{R.~J.}} \AND
  \bauthor{\bsnm{Jones},~\bfnm{M.~C.}\binits{M.~C.}}
(\byear{1999}).
\btitle{An improved estimator of the density function at the boundary}.
\bjournal{J. Amer. Statist. Assoc.}
\bvolume{94}
\bpages{1231--1241}.
\bdoi{10.2307/2669937}
\bmrnumber{1731485}
\end{barticle}
\endbibitem

\bibitem[\protect\citeauthoryear{Zhang and
  Karunamuni}{2000}]{ZhangKarunamuni2000}
\begin{barticle}[author]
\bauthor{\bsnm{Zhang},~\bfnm{S.}\binits{S.}} \AND
  \bauthor{\bsnm{Karunamuni},~\bfnm{R.~J.}\binits{R.~J.}}
(\byear{2000}).
\btitle{On nonparametric density estimation at the boundary}.
\bjournal{J. Nonparametr. Statist.}
\bvolume{12}
\bpages{197--221}.
\bdoi{10.1080/10485250008832805}
\bmrnumber{1752313}
\end{barticle}
\endbibitem

\bibitem[\protect\citeauthoryear{Zhou and Pan}{2014}]{ZhouPan2014}
\begin{barticle}[author]
\bauthor{\bsnm{Zhou},~\bfnm{L.}\binits{L.}} \AND
  \bauthor{\bsnm{Pan},~\bfnm{H.}\binits{H.}}
(\byear{2014}).
\btitle{Smoothing noisy data for irregular regions using penalized bivariate
  splines on triangulations}.
\bjournal{Comput. Statist.}
\bvolume{29}
\bpages{263--281}.
\bdoi{10.1007/s00180-013-0448-z}
\bmrnumber{3260122}
\end{barticle}
\endbibitem

\end{thebibliography}

\end{document}